\documentclass[11pt,oneside]{amsart}

\usepackage[margin=1in, letterpaper]{geometry}
\usepackage{amscd}
\usepackage{amsfonts}
\usepackage{amsmath}
\usepackage{amssymb}
\usepackage{amsthm}
\usepackage[mathscr]{euscript}
\usepackage{ifthen}
\usepackage{mathtools}
\usepackage{tikz-cd}
\usepackage{graphicx}
\usepackage{xcolor}
\usepackage{rotating}
\usepackage{tikz}
\usepackage{hyperref}
\usepackage[capitalize]{cleveref}
\usetikzlibrary{matrix, arrows}
\usepackage{url}
\usepackage{wrapfig}
\usetikzlibrary{matrix, arrows, decorations.markings, decorations.pathreplacing, positioning, arrows.meta}
\usepackage{todonotes}
\usepackage{stmaryrd}
\usepackage{mathdots}
\usepackage{mathrsfs}
\usepackage{manfnt}

\tikzset{
	etail/.style={postaction={
  	decorate,
      decoration={markings, mark=at position 0.1 with
    	{\node {\hspace{0pt}\raisebox{2.5pt}{$\circ$}};}}
  }}
} 
\tikzset{
	etailhor/.style={postaction={
  	decorate,
      decoration={markings, mark=at position 0 with
    	{\node {\hspace{-2.5pt}\raisebox{-1.6pt}{$\circ$}};}}
  }}
} 

\numberwithin{equation}{section}
\theoremstyle{plain}
\newtheorem{theorem}[equation]{Theorem}
\newtheorem{lemma}[equation]{Lemma}
\newtheorem{proposition}[equation]{Proposition}
\newtheorem{corollary}[equation]{Corollary}

\theoremstyle{definition}
\newtheorem{definition}[equation]{Definition}

\theoremstyle{remark}
\newtheorem{remark}[equation]{Remark}
\newtheorem{example}[equation]{Example}
\newtheorem{non-example}[equation]{Counterexample}

\theoremstyle{plain}

\renewcommand{\bar}[1]{\ensuremath{\overline{#1}}}

\newcommand{\A}{\mathcal{A}}
\newcommand{\C}{\mathcal{C}}
\newcommand{\CC}{\mathbf{C}}
\newcommand{\D}{\mathbb{D}}

\newcommand{\T}{\mathfrak{T}}

\newcommand{\E}{\mathcal{V}(\D)}
\newcommand{\cE}{\mathcal{E}}
\newcommand{\M}{\mathcal{H}(\D)}
\newcommand{\squareE}{\mathcal{H}(\D_v)}
\newcommand{\squareM}{\mathcal{V}(\D_h)}
\newcommand{\triE}{\mathcal{V}^\triangleleft(\D_v)}
\newcommand{\triM}{\mathcal{H}^\triangleleft(\D_h)}

\renewcommand{\H}{\mathbb{H}}
\newcommand{\V}{\mathbb{V}}
\newcommand{\vop}{{\mathsf{vop}}}
\newcommand{\hop}{{\mathsf{hop}}}
\newcommand{\opp}{{\mathsf{op}}}

\newcommand{\tlift}{\mathsf{2copy}}
\newcommand{\ilift}{\mathsf{isocopy}}
\newcommand{\ulift}{\mathsf{agree}}
\newcommand{\arrlift}{\mathsf{partcyl}}
\newcommand{\arrfill}{\mathsf{fillcyl}}

\newcommand{\hv}{\mathbf{hv}}
\newcommand{\vh}{\mathbf{vh}}

\newcommand{\flatten}{\mathsf{simp}}
\newcommand{\truncateh}{\mathsf{trunc}_h}
\newcommand{\truncatev}{\mathsf{trunc}_v}
\newcommand{\rotatel}{\mathsf{rot}_{\ulcorner}}
\newcommand{\rotater}{\mathsf{rot}_{\lrcorner}}
\newcommand{\invert}{\mathsf{inv}}
\newcommand{\doublerotate}{\mathsf{rotfold}}
\newcommand{\hiso}{\mathsf{hiso}}
\newcommand{\viso}{\mathsf{viso}}
\newcommand{\toprightcorner}{\mathsf{trcor}}
\newcommand{\bottomleftcorner}{\mathsf{blcor}}
\newcommand{\uniquetoprightcorner}{\ulift(\toprightcorner)}
\newcommand{\uniquebottomleftcorner}{\ulift(\bottomleftcorner)}
\newcommand{\mono}{\mathsf{mono}}
\newcommand{\uniquearrow}{\mathsf{uarr}}

\newcommand{\cat}{\mathcal{C}\mathsf{at}}
\newcommand{\Cat}{\mathbf{Cat}}
\newcommand{\DCat}{\mathbf{DCat}}
\newcommand{\DCati}{\mathbf{DCat}^{\cong}}
\newcommand{\TCat}{\mathbf{TCat}}

\newcommand{\Set}{\mathbf{Set}}
\newcommand{\SSet}{s\mathbf{Set}}

\newcommand{\SSSet}{ss\mathbf{Set}}
\newcommand{\ssSet}{\SSSet}

\newcommand{\Ar}{\operatorname{Ar}}
\newcommand{\cl}{\operatorname{cl}}

\newcommand{\coker}{\operatorname{coker}}

\newcommand{\Deltaop}{\Delta^{\op}}
\newcommand{\Hom}{\operatorname{Hom}}
\newcommand{\Homm}{\underline\Hom}

\newcommand{\id}{\operatorname{id}}
\newcommand{\Iso}{\operatorname{Iso}}
\newcommand{\iso}{\operatorname{iso}}
\renewcommand{\ker}{\operatorname{ker}}
\newcommand{\Map}{\operatorname{Map}}
\newcommand{\mor}{\operatorname{mor}}
\newcommand{\nerve}{\operatorname{nerve}}
\newcommand{\fund}{\tau_1}
\newcommand{\ob}{\operatorname{ob}}
\newcommand{\op}{\operatorname{op}}

\newcommand{\squares}{\operatorname{square}}

\DeclareMathOperator*{\colim}{colim}

\newcommand{\rto}{\to}

\newcommand{\mrto}[1][]{\hspace{1pt}\begin{tikzcd}[sep=20pt,cramped, ampersand replacement=\&, text height=1ex, text depth=.3ex]\rar[tail]{#1}\&{}\end{tikzcd}}

\newcommand{\erto}[1][]{\hspace{1pt}\begin{tikzcd}[sep=20pt,cramped, ampersand replacement=\&, text height=1ex, text depth=.3ex]\rar[etailhor]{#1}\&{}\end{tikzcd}}

\newcommand{\elto}[1][]{\hspace{1pt}\begin{tikzcd}[sep=20pt,cramped, ampersand replacement=\&, text height=1ex, text depth=.3ex]{}\&\lar[etailhor,swap]{#1}{}\end{tikzcd}}

\newcommand{\distsquare}[8]{\begin{tikzcd}[ampersand replacement=\&]
#1 \rar[tail]{#5} \dar[etail,swap]{#6} \ar[phantom]{dr}{\square} \& #2 \dar[etail]{#7} \\
#3 \rar[tail,swap]{#8} \& #4
\end{tikzcd}}

\newcommand{\parallelogram}{%
  \tikz[scale=0.25, baseline=-0.5ex]{%
    \draw (0,0) -- (1.1,0) -- (2.2,.7) -- (1.1,.7) -- cycle;
  }}

\newcommand{\parallelogramvert}{%
  \tikz[scale=0.25, baseline=-0.5ex]{%
    \draw (0,0) -- (0,1) -- (.9,1.6) -- (.9,.6) -- cycle;
  }}

\newcommand{\mycube}{\mbox{\mancube}}

\newsavebox{\mrtonested}
\newsavebox{\ertonested}
\newsavebox{\hvsquare}
\newsavebox{\hisquare}
\newsavebox{\visquare}
\newsavebox{\hvicube}

\newsavebox{\mrtonestedb}
\newsavebox{\ertonestedb}
\newsavebox{\hvsquareb}
\newsavebox{\hisquareb}
\newsavebox{\visquareb}
\newsavebox{\hvicubeb}

\newcommand{\dist}[2]{\ar[phantom,from=#1,to=#2,"\square"]}
\newcommand{\comm}[2]{\ar[phantom,from=#1,to=#2,"\circlearrowleft"]}
\newcommand{\vcong}{\cong}

\begin{document}

\title{Recognizing CGW categories among pointed stable double Segal spaces}

\author{Julia E.\ Bergner, Brandon T.\ Shapiro, and Inna Zakharevich}

\maketitle

\begin{abstract}
In this paper, we establish a precise relationship between CGW categories and pointed stable double Segal spaces, both of which were developed as general input for algebraic K-theory.  In particular, we show that any CGW category can be regarded as a pointed stable double Segal space, and they can be identified using a classifying diagram construction for double categories with shared isomorphisms. 
\end{abstract}

\tableofcontents

\section{Introduction}

CGW categories were first defined by Campbell and the third-named author as a framework for algebraic $K$-theory, generalizing exact categories to also include finite sets, algebraic varieties, and a range of other mathematical objects admitting abstract notions of kernels and cokernels or complements. Similarly to an exact category, a CGW category admits the construction of a $K$-theory space which can be defined by a generalization of Quillen's $Q$-construction or equivalently by a generalization of Waldhausen's $S_\bullet$-construction. These $K$-theory spaces were shown to satisfy appropriate versions of the D\'evissage and Localization theorems from classical algebraic $K$-theory, while Sarazola and the second-named author showed that under additional conditions the $S_\bullet$ construction also satisfies the Additivity and Fibration theorems \cite{CZ-cgw,SS-ecgw}.

From a quite different point of view, the theory of 2-Segal spaces was developed by Dyckerhoff and Kapranov, and independently under the name of decomposition spaces by G\'alvez-Carrillo, Kock, and Tonks.  Although these two sets of authors had quite different motivations for developing this theory, interestingly, they both identified the output of applying the $S_\bullet$-construction to an exact category as a key example.  In subsequent work, the first-named author, Osorno, Ozornova, Rovelli, and Scheimbauer proved that any 2-Segal space can be obtained from this construction, given a suitably general input, namely that of an augmented stable double Segal space.

Since both CGW categories and augmented stable double Segal spaces were developed to be general input for a $K$-theory construction, a natural question is then what the relationship is between the two.  Answering this question is the primary goal of this paper.

With an appropriate reworking of the definition of a CGW category, it is not hard to see that every CGW category can be regarded as a pointed stable double Segal space, but the converse is not true.  The next question is whether we can neatly characterize which pointed stable double Segal spaces correspond to CGW categories.  Somewhat surprisingly, the answer can be framed in terms of a double-categorical generalization of the classifying diagram construction of Rezk.

\subsection{Double categories as input for algebraic $K$-theory}

The notion of a double category, or category internal to categories, is central to this work, as CGW categories are examples of such and pointed stable double Segal spaces are suitable homotopical generalizations of them.  One should think of a double category as consisting of objects, two different kinds of morphisms, usually suggestively called horizontal and vertical morphisms, and squares, depicted as 
	\[ \distsquare{A}{B}{C}{D.}{}{}{}{} \] 
This structure is well-behaved, in that the objects together with either kind of morphism forms a category, and likewise each type of morphisms can be taken as the objects in a category whose morphisms are given by the squares.

In this paper, we consider double categories, or generalizations of them, that have an associated $K$-theory space, given by an application of a version of Waldhausen's $S_\bullet$-construction.  Quite broadly, such a construction associates to a double category some space of diagrams of the form 
\begin{equation} \label{sdotdiagram}
	\begin{tikzcd}
		\ast \rar[tail] & A_{0,1} \rar[tail] \dar[etail] & A_{0,2} \rar[phantom]{\cdots} \dar[etail] & A_{0,n \text{-} 1} \rar[tail] \dar[etail] & A_{0,n} \dar[etail] \\
		& \ast \rar[tail] & A_{1,2} \rar[phantom]{\cdots} & A_{1,n 	\text{-} 1} \rar[tail] \dar[phantom]{\vdots} & A_{1,n} \dar[phantom]{\vdots} \\
		& & \ddots & A_{n \text{-} 2,n \text{-} 1} \rar[tail] \dar[etail] & A_{n \text{-} 2,n} \dar[etail] \\
		& & & \ast \rar[tail] & A_{n \text{-} 1,n} \dar[etail] \\
		& & & & \ast
		\dist{1-2}{2-3} \dist{1-4}{2-5} \dist{3-4}{4-5}
	\end{tikzcd} 
\end{equation}
in which the object $\ast$ is initial in the horizontal category and terminal in the vertical category, so that the double category is \emph{pointed}.  We also want an analogue of the squares being bicartesian, meaning either of the pairs of morphisms 
\begin{equation} \label{regstable}
	\begin{tikzcd}
		\bullet \rar[tail] \dar[etail] & \bullet \\
 		\bullet & 
	\end{tikzcd}
	\qquad \qquad
	\begin{tikzcd}
		\bullet \rar[tail] \dar[etail] & \bullet \dar[etail] \\
		\bullet \rar[tail] & \bullet
		\dist{1-1}{2-2}
	\end{tikzcd}
	\qquad \qquad
	\begin{tikzcd}
		& \bullet \dar[etail] \\
		\bullet \rar[tail] & \bullet
	\end{tikzcd} 
\end{equation}
of a square, called a \emph{mixed span} and \emph{mixed cospan}, respectively, determine the square uniquely up to unique isomorphism. We refer to such a double category as \emph{stable}.

Alternatively, we can think of diagrams of the form \eqref{sdotdiagram} with the direction of the vertical morphisms reversed.  With this change, the squares are essentially uniquely determined by either of two pairs of boundary morphisms 
\begin{equation} \label{reversestable} 
	\begin{tikzcd}
		\bullet \rar[tail] & \bullet \dar[etail] \\
		& \bullet
	\end{tikzcd}
	\qquad  \qquad
	\begin{tikzcd}
		\bullet \rar[tail] \dar[etail] & \bullet \dar[etail] \\
		\bullet \rar[tail] & \bullet
		\dist{1-1}{2-2}
	\end{tikzcd}
	\qquad  \qquad
	\begin{tikzcd}
		\bullet \dar[etail] &  \\
		\bullet \rar[tail] & \bullet.
	\end{tikzcd} 
	\end{equation}
In this case the squares are no longer bicartesian, since the morphisms form neither a span nor a cospan, but we can still define stability in this context via such configurations determining a square uniquely.  The first formulation is meant to be suggestive of inclusions and quotients, as in short exact sequences in an exact category, whereas the second is better suited to situations in which we think of two different kinds of inclusion, such as open and closed immersions of algebraic varieties.
%but the desired ``bicartesian-ness'' is in the sense of an exact category where the vertical morphisms are reversed, so these boundary morphisms are different from the span and cospan of a square. \jbnoteil{reorganized this sentence to fit in diagram}

Let us now turn to the particular kinds of double categories that we consider in this paper.  A CGW category is a double category with additional structure encoding that every morphism of one type has a complementary morphism of the other type, and the squares induce isomorphisms on complements.  In particular we use the framework of \eqref{reversestable}.  An example is the double category of finite sets, where both types of morphisms are injections, the complementary pairs of morphisms are just complementary pairs of inclusions, and the squares are exactly the bicartesian squares.  
		
In order to compare with double categories for which stability is defined via \eqref{regstable}, we define \emph{reverse CGW} categories by dualizing a CGW categories in the vertical direction. We also show that this stability-like property is equivalent to the original definition in terms of complementary morphisms.  For example, an exact category can be thought of as a CGW category in which the horizontal morphisms are admissible monomorphisms and the vertical morphisms from $A$ to $B$ are admissible epimorphisms from $B$ to $A$; complementary pairs of morphisms are kernel-cokernel pairs, and squares are again bicartesian.  Reversing the direction of the vertical morphisms, we get a double category whose vertical morphisms are simply the admissible epimorphisms.  

%\subsection{Pointed stable double Segal spaces}

The other general input for algebraic $K$-theory that we consider is that of pointed stable double Segal spaces, which can be thought of as homotopical double categories, in the following sense.

Segal spaces were first defined by Rezk \cite{rezk2001model}, inspired by conditions first defined by Segal \cite{segal1968classifying}.  One can think of them as categories internal to spaces, so categories with a space of morphisms and a space of objects, but with composition only defined up to homotopy, and properties such as associativity and unitality only holding up to homotopy.  They are defined more precisely as simplicial spaces, or functors $X \colon \Deltaop \rightarrow \SSet$ such that certain maps
\[ X_n \rightarrow \underbrace{X_1 \times_{X_0} \cdots \times_{X_0} X_1}_n \]
are weak equivalences of simplicial sets for all $n \geq 2$.  

Generalizing this definition, a \emph{double Segal space} is a bisimplicial space $Y \colon \Deltaop \times \Deltaop \rightarrow \SSet$ that is a Segal space in each variable.  It can be thought of as modeling a double category internal to spaces, so that the objects, horizontal morphisms, vertical morphisms, and squares all form spaces, but composition is only defined up to homotopy in each direction.  A double Segal space is \emph{stable} if it satisfies a condition analogous to that for stable double categories: if its squares, or elements of the space $Y_{1,1}$ are essentially uniquely determined by their span or cospan.  Similarly, it is \emph{pointed} if there is an essentially unique object that is initial in the horizontal direction and terminal in the vertical direction.  

Pointed stable double Segal spaces were defined with the goal of establishing an equivalence, via the $S_\bullet$-construction, with reduced 2-Segal spaces.  The idea is that, just as a Segal space encodes the structure of a simplicial category up to homotopy, a 2-Segal space encodes a similar but weaker structure, where composition of morphisms need not always exist, or be unique, but is still associative, in an appropriate homotopical sense.  We explain this equivalence in more detail in Section \ref{sdotsection}.

From these definitions, it is not too hard to see the following result, using a suitable definition of the double nerve of a double category.

\begin{proposition}
	The double nerve of a reverse CGW category is a pointed stable double Segal space.
\end{proposition}

However, in general, pointed stable double Segal spaces are more general than reverse CGW categories.  Nonetheless, we can give an elegant characterization of those that arise from CGW categories using the theory of classifying diagrams.	

\subsection{Classifying diagrams}

Let $\mathcal C$ be a small category.  Rezk defined its \emph{classifying diagram} to be the simplicial space $N\mathcal C$ given by
\[ (N\mathcal C)_n = \nerve(\iso \mathcal C^{[n]}), \]
where $\mathcal C^{[n]}$ denotes the category of functors 
\[ [n] = \underbrace{\bullet \rightarrow \bullet \rightarrow \cdots \rightarrow \bullet}_n \rightarrow \mathcal C \]
and $\iso \mathcal C^{[n]}$ denotes the maximal subgroupoid of $\mathcal C^{[n]}$, consisting only of its isomorphisms.  We can thus think of the set $(N\mathcal C)_{n,m}$ as consisting of all grids of the form 
\[ \begin{tikzcd}
      \bullet \dar{\cong} \rar & \bullet \dar{\cong} \rar[phantom]{\cdots} \rar[phantom, shift left=3]{n} & \bullet \dar{\cong} \rar & \bullet \dar{\cong} \\
      \bullet \dar[phantom]{\vdots} \dar[phantom, shift right=3]{m} \rar & \bullet \dar[phantom]{\vdots} \dar[phantom, shift right=3]{} \rar[phantom]{\cdots} \rar[phantom, shift left=3]{} & \bullet \dar[phantom]{\vdots} \dar[phantom, shift left=3]{} \rar & \bullet \dar[phantom]{\vdots} \dar[phantom, shift left=3]{} \\
      \bullet \dar{\cong} \rar & \bullet \dar{\cong} \rar[phantom]{\cdots} \rar[phantom, shift right=3]{} & \bullet \dar{\cong} \rar & \bullet \dar{\cong} \\
      \bullet \rar & \bullet \rar[phantom]{\cdots} \rar[phantom, shift
      right=3]{} & \bullet \rar & \bullet.  \comm{1-1}{2-2}
      \comm{1-3}{2-4} \comm{3-1}{4-2} \comm{3-3}{4-4}
    \end{tikzcd} \]  

Rezk's motivation for defining classifying diagrams was to give a refinement of the nerve construction.  Recall that an equivalence of categories $\mathcal C \rightarrow \mathcal D$ induces a weak equivalence on nerves, but the converse statement is false; for example, the unique functor $[1] \rightarrow [0]$ is not an equivalence of categories, yet induces a map of contractible spaces on nerves.  However, it is true that $\mathcal C \rightarrow \mathcal D$ is an equivalence of categories if and only if $N\mathcal C \rightarrow N\mathcal D$ is a levelwise weak equivalence of simplicial sets \cite[Proposition 3.3.4]{jbbook}.  This construction is meant to provide a concrete example of the notion of complete Segal space, which is a model for $(\infty,1)$-categories.

In this paper, we develop a more general notion of classifying diagram of double categories, which can be defined for double categories with shared isomorphisms, which means that the categories of horizontal and vertical morphisms have the same subcategory of isomorphisms, in a suitably compatible way; see Definition~\ref{sharedisos}. 

We then show that pointed stable double Segal spaces that arise as classifying diagrams of double categories correspond to CGW categories. Our main results can be summarized as follows.  

\begin{theorem}
	\begin{itemize}
		\item If $\D$ is a CGW category, then its classifying diagram $N^\square\D$ is a pointed stable double Segal space.
	
		\item If $\D$ is a double category with shared isomorphisms in which all horizontal morphisms are monic and all vertical morphisms are epic, and if its classifying $N^\square\D$ is a pointed stable double Segal space, then $\D$ is a CGW category. 
	\end{itemize}
\end{theorem}

In the course of the paper, we actually develop a categorical version of the classifying diagram that takes a category to a double category; applying the nerve functor in one direction recovers Rezk's original definition.  The generalization of this construction to double categories takes values in triple categories, and we first prove the above theorem in this context, and then take a suitable nerve functor in two of the directions to give the explicit connection with stable double Segal spaces.

%\subsection{Lifting properties}

A key feature of this paper is the use of lifting properties to describe various objects with additional structure.  A significant example is the description of classifying diagrams, both for categories and for double categories, in terms of lifting properties.  We also make use of some notions of lifting properties for 2-categories that might be of independent interest.	

\subsection{Outline of the paper}

In Section \ref{liftingsection}, we review the notion of lifting properties in the setting of categories and give some examples and essential results.  In Section \ref{classifyingdcsection}, we recall some basic definitions for double categories and introduce the classifying double category of an ordinary category.  The main result of this section is the characterization of such double categories in terms of lifting properties.  We then generalize this construction and recognition theorem to the classifying triple category of a double category in Section \ref{classifyingtcsection}.  In Section \ref{cgwsection}, we recall the definition of CGW categories, give an alternate formulation of them, and characterize them in terms of lifting properties.  We then translate these results to the simplicial context in Section \ref{simplicialsection} and prove our main comparison result in this setting in Section \ref{comparisonsection}, where we also introduce two notions of lifting properties in the setting of 2-categories.  Finally, we briefly discuss the $S_\bullet$-construction in light of this comparison.

\section{Lifting properties} \label{liftingsection}

%\jbnoteil{I wonder if there is a very familiar running example that we could include here to make this material more concrete.  The key might be finding something that also has a good example of a fully faithful functor as well.}
%
%\bsnoteil{I've added the example of every morphism being monic in a category, and the fully faithful functor from categories to simplicial sets. I figure this has a similar flavor to our main examples later on and ends up getting used for CGW categories. Likewise I added categories with pushouts as an example for lifting properties in 2-categories, which I think gives some intuition even though it doesn't quite meet the definition for reasons that actually motivate the use of double categories (happy to swap out this one).}

A standard and often useful method of describing properties of mathematical structures is through lifting properties.  This approach is often also described in terms of localization with respect to the set of maps with which we require the lifting properties.

In this section, we give a brief summary of these techniques and a few motivating examples.%, first in the context of ordinary categories, and then in the context of 2-categories.

%\subsection{Lifting properties in categories}

We begin with some terminology for an object in a category possessing a lifting property.  

\begin{definition}
  	Let $\C$ be a category with a terminal object, and let $S$ be a set of morphisms in $\C$.  An object $X$ is \emph{$S$-local} if for any $f \colon A \rto B$ in $S$ and any solid-arrow diagram
  	\[ \begin{tikzcd}
      	A \rar{} \dar[swap]{i} &X \\ 
      	B \ar[densely dashed]{ur}
    \end{tikzcd}\] 
    the dashed arrow exists.  We say that $X$ is \emph{strictly $S$-local} if this arrow is unique.  For a single morphism $f$, we say that $X$ is \emph{(strictly) $f$-local} if it is (strictly) $\{f\}$-local.
\end{definition}

\begin{remark}
	The definition of an $S$-local object is often instead given in terms of
        appropriate mapping objects, so that $X$ is strictly $S$-local precisely
        when, for any morphism $A \rightarrow B$ in $S$, the induced map
        $\Hom_\C(B,X) \to \Hom_\C(A,X)$ is an isomorphism.  More generally, if working in a simplicially enriched setting, for example in model categories, an object $X$ is defined to be $S$-local precisely when the induced map on mapping spaces is a weak equivalence of simplicial sets.
\end{remark}
	
This definition encodes the notion that locally $X$ behaves in a way compatible with each $f$, in the sense that objects related by $f$ cannot be distinguished by their mapping into $X$.  Throughout this paper, we give several examples of morphisms $f$ such that $X$ is $f$-local if and only if certain useful properties hold inside $X$, and we characterize several different types of objects in terms of being local with respect to a carefully chosen set of morphisms. 

\begin{example} \label{monolift}
	For a category $\C$, we claim that every morphism in $\C$ is monic if and only if $\C$ is local with respect to the identity-on-objects functor
	\[ \mono \colon \left(
	\begin{tikzcd}
		A \rar[shift left=1]{f} \rar[shift right=1,swap]{f'} \ar[shift right=1,bend right=35]{rr}[swap]{gf=gf'} & B \rar{g} & C
	\end{tikzcd}
	\right) \to \left(
	\begin{tikzcd}
		A \rar{f=f'} & B \rar{g} & C
	\end{tikzcd}
	\right). \]
	The domain category of $\mono$ precisely captures the property that the morphisms $f$ and $f'$ are coequalized by the morphism $g$.  If we choose any morphisms of $\C$ fitting into the configuration of this domain category, then the lifting property captures that the morphism in the image of $g$ must be monic.  Since any morphism of $\C$ can be obtained in this way, we can conclude that every morphism of $\C$ is a monomorphism.  We refer to such a category as a \emph{monomorphism category}.
	
	%For any morphism $g$ in $\C$, and any morphisms $f$ and $f'$ in $\C$ coequalized by $g$, in the sense of forming a functor from the domain category of $\mono$ into $\C$, $f$ and $f'$ must agree by the existence of a (necessarily unique) lift with respect to $\mono$. Hence $g$ is monic.
\end{example}

Our next example is in the context of simplicial sets; throughout this paper we frequently use their close relationship with categories.  Recall the simplicial indexing category $\Delta$ whose objects are the finite ordered sets $[n]=\{0 \leq 1 \leq \cdots \leq n\}$ for $n \geq 0$ and whose morphisms are the order-preserving functions between them.  A \emph{simplicial set} is a functor $\Deltaop \rightarrow \Set$, and we denote by $\SSet$ the category of simplicial sets and natural tranformations between them.

\begin{example} \label{nervelocal} 
	Given any $n \geq 0$, recall that the $n$-\emph{simplex} $\Delta[n]$ is the representable simplicial set, so $\Delta[n]_k = \Hom_\Delta([k],[n])$.  For any $n$, let $G[n]$ denote the \emph{spine} of $\Delta[n]$, given by the colimit of the diagram 
	\[ \Delta[1] \xleftarrow{d^0} \Delta[0] \xrightarrow{d^1} \cdots \leftarrow \Delta[0] \xrightarrow{d^1} \Delta[1] \] with $n$ copies of $\Delta[1]$.  A simplicial set is (strictly) local with respect to the inclusions 
	\[ \iota_n \colon G[n] \rightarrow \Delta[n] \]
	if and only if it can be obtained as the nerve of a category.  We elaborate further on this example in Section \ref{subsection.segalspaces}.
\end{example}

This type of characterization can be thought of as a recognition problem, and we throughout we state several results to the effect that an object in a category $\C'$ is in the image of some functor $F \colon \C \to \C'$ if and only if it is $S$-local, for some suitable set $S$ of morphisms in $\C'$.  

One useful approach to proving such statements is to factor the functor $F$ as the composite of two functors and recognize the images of the two factors separately. The following lemmas make this idea more precise.

\begin{lemma} \label{functorlift}
	If $F \colon \C \to \C'$ is a fully faithful functor, then an object $X$ in $\C$ is (strictly) $S$-local if and only if $FX$ is (strictly) $FS$-local in $\C'$.
\end{lemma}

\begin{proof}
	Given a morphism $f \colon A \to B$ in $S$, every morphism in $\C'$ from either $FA$ or $FB$ to $FX$ is the image of a unique morphism from $\C$, since $F$ is fully faithful, from which the result follows.
\end{proof}

Moreover, this lemma tells us how to recognize the image of a composite of two functors when the second is fully faithful.  The following lemma can be proved similarly to the previous one. %The goal of this is to ultimately characterize the image of a composite functor $\C \xrightarrow{F} \C' \xrightarrow{G} \C''$ by identifying lifting conditions in $\C'$ for the image of $F$ and lifting conditions in $\C''$ for the image of $\C'$.

\begin{lemma} \label{compositelift}
	Given a pair of functors $\C \xrightarrow{F} \C' \xrightarrow{G} \C''$ such that
	\begin{itemize}
		\item $G$ is fully faithful, 
	
		\item an object in $\C'$ is in the essential image image of $F$ if and only if it is (strictly) $S$-local, and
	
		\item an object in $\C''$ is in the essential image of $G$ if and only if it is (strictly) $T$-local,
	\end{itemize}
	then an object in $\C''$ is in the image of $GF$ if and only if it is (strictly) $(FS \cup T)$-local.
\end{lemma}

\begin{example}\label{mononerve}
	The nerve functor taking categories to simplicial sets is fully faithful, and, as described in Example \ref{nervelocal}, a simplicial set is isomorphic to the nerve of a category when it is strictly local with respect to the set $\{\iota_n \colon G[n] \to \Delta[n]\}$ of spine inclusions into simplices.  Combining this result with the previous lemma shows that a simplicial set is isomorphic to the nerve of a monomorphism category precisely when it is strictly local with respect to the set $\{\iota_n \colon G[n] \to \Delta[n]\} \cup \{\nerve(\mono)\}$.
\end{example}

It is also helpful to note how local properties interact with adjunctions.

\begin{lemma} \label{adjunctionlift}
	Let $F \colon \C \to \C'$ be a functor with left adjoint $G \colon \C' \to \C$.  If $S$ is a set of morphisms in $\C'$ and $X$ is an object of $\C$, then the object $FX$ of $\C'$ is (strictly) $S$-local if and only if $X$ is (strictly) $GS$-local.
\end{lemma}

\begin{proof}
	Given a morphism $i \colon A \to B$ in $S$, observe that morphisms $A \to FX$ and $B \to FX$ correspond to morphisms $GA \to X$ and $GB \to X$ via the adjunction.  The result follows.
\end{proof}

\begin{example}\label{monoadj}
	The nerve functor from categories to simplicial sets has a left adjoint $\fund$ called the fundamental category functor. In the context of Example~\ref{monolift} then, Lemma~\ref{adjunctionlift} says that a nerve $\nerve(\C)$ is $\nerve(\mono)$-local if and only if $\C$ is $\fund(\nerve(\mono))$-local. As the fundamental category functor is a left-inverse left adjoint, meaning a category is isomorphic to the fundamental category of its nerve, this further shows that a category is a monomorphism category (local wih respect to $\fund(\nerve(\mono))=\mono$) if and only if its nerve is local with respect to $\nerve(\mono)$. 
	%Let us consider a simplicial set $K$ that is (strictly) local with respect to only the morphism $\nerve(\mono)$.  The left adjoint to the nerve functor, which sends a simplicial set to a category, takes $K$ to a monomorphism category.  
\end{example}

Finally, we discuss how lifting properties with respect to distinct morphisms can be related to one another. These results are well known from the setting of model categories; see for instance \cite[Lemmas 7.2.11(1) and 7.2.8(1)]{hirschhorn}.

\begin{lemma}\label{pushoutlift}
	Let $i' \colon A' \to B'$ and $i \colon A \to B$ be morphisms in $\C$ related by a pushout square
	\[ \begin{tikzcd}
      	A \rar \dar[swap]{i} \ar[phantom]{dr}[pos=.9]{\ulcorner} & A' \dar[swap]{i'}\\ 
      	B \rar & B'.
      \end{tikzcd}\] 
	If an object $X$ of $\C$ is (strictly) $i$-local, then it is also (strictly) $i'$-local.  	
\end{lemma}

\begin{proof}
	Suppose that $X$ is (strictly) $i$-local.  In any diagram of the form 
	\[ \begin{tikzcd}
      	A \rar \dar[swap]{i} \ar[phantom]{dr}[pos=.9]{\ulcorner} & A' \rar{} \dar[swap]{i'} & X \\ 
      	B \rar & B' \ar[densely dashed]{ur} 
      \end{tikzcd}\] 
	precomposing the morphism $A' \to X$ with $A \to X$ induces a (unique) lift to a morphism $B \to X$ making that portion of the diagram commute.  Since $B'$ is a pushout, there is an induced morphism $B' \to X$ making the right triangular diagram commute; hence $X$ is $i'$-local.  Moreover, by the universal property of pushouts, this map is unique relative to the lift $B \to X$, so if $X$ is strictly $i$-local it is also strictly $i'$-local.
\end{proof}

\begin{lemma} \label{retractlift}
	If a morphism $i' \colon A' \rightarrow B'$ is a retract of $i \colon A \rightarrow B$ in $\C$, then a (strictly) $i$-local object is also (strictly) $i'$-local.	
%	Given a pair of morphisms $i' \colon A' \to B'$ and $i \colon A \to B$ in $\C$ related by commutative squares as below
%	\[ \begin{tikzcd}
%      	A' \rar \dar[swap]{i'} & A \rar \dar[swap]{i} & A' \dar[swap]{i'} \\ 
%      	B' \rar & B \rar & B'
%      \end{tikzcd}\] 
%	whose horizontal morphisms compose to identities on $A'$ and $B'$, if an object $X$ is (strictly) $i$-local then it is also (strictly) $i'$-local.  	
\end{lemma}

\begin{proof}
	Consider a diagram
	\[ \begin{tikzcd}
      	A' \rar \dar[swap]{i'} & A \rar \dar[swap]{i} & A' \rar{} \dar[swap]{i'} & X \\ 
      	B' \rar & B \rar & B', \ar[densely dashed]{ur}
      \end{tikzcd}\] 
	where the left-hand rectangular diagram presents the retraction.  A similar argument as for the previous lemma, using the fact that the lower horizontal maps compose to the identity, establishes the existence of the dotted-arrow lift.
	%a morphism $A' \to X$ composes into a morphism $A \to X$ which induces a (unique) lift to a morphism $B \to X$ which composes into a morphism $B' \to X$. As any other lift $B' \to X$ of the map $A' \to X$ would compose to a lift $B \to X$ of the composite $A \to A' \to X$, this lift is unique when $X$ is strictly $i$-local.
\end{proof}

\section{The classifying double category of a category} \label{classifyingdcsection}

Our goal in this section is to define and characterize classifying double categories using lifting properties.  Before doing so, we recall the definition and several examples of double categories, as well as some properties they can possess.

\subsection{Double categories}

Let us begin by recalling their definition of a double category.

\begin{definition} 
	A \emph{double category} $\D$ is a category internal to categories; see for instance \cite[\S 2]{fiorepaolipronk}.  As such, it consists of a pair of categories with the same objects, whose morphisms are called \emph{horizontal morphisms} and \emph{vertical morphisms}, respectively, and equipped with \emph{squares} of the form
	\[ \distsquare{A}{B}{C}{D}{}{}{}{} \] 
	between horizontal and vertical morphisms.  The squares have identities and composites in both the horizontal and vertical directions, subject to unit, associativity, and interchange conditions.

A double functor between double categories is a mapping on objects, both types of morphisms, and squares preserving all sources, targets, identities, and composites.
\end{definition}

If $\D$ is a double category, we denote by:
\begin{itemize}
	\item $\mathcal H(\D)$ the category of objects and horizontal morphisms,
	
	\item $\mathcal V(\D)$ the category of objects and vertical morphisms,
	
	\item $\mathcal H(\D_v)$ the category whose objects are vertical morphisms and whose morphisms are squares, and 
	%\bsnoteil{These are not standard, I made them up in vague reference to the double nerve. I also now realize that in section 5 we use Inna's notation where $\D_h$ is called $\M$ and $\D_v^1$ is called $\Ar_\square \M$, so we should probably make it consistent. In principle I prefer everything to have a $\D$ in it, but if we are going to use that notation in the CGW section (which is convenient to compare with $\Ar_\triangle$ and consistent with Inna's CGW paper) then we could consider using it here too I guess. It might even simplify the triple category notation where we could write something like $\M \boxtimes \E$ for the horizontal-vertical double category of a triple category (or use different letters for these categories less tied to monos and epis). Open to suggestions though, neither feels optimal.}
	
	\item $\mathcal V(\D_h)$ the category whose objects are horizontal morphisms and whose morphisms are squares.
\end{itemize}
A square in $\D$ is called \emph{trivial} if it is a horizontal or vertical identity; in other words, an identity in either the category $\mathcal H(\D_v)$ or the category $\mathcal V(\D_h)$.

%When it can be inferred from context, we do not mark squares from $\D$ with $\square$; when a square is drawn with mixed vertical and horizontal morphisms it is assumed to be a square from $\D$, unless otherwise specified. \jbnoteil{Don't we always indicate the squares?  This comment can be removed if we don't actually ever omit the notation.}

\begin{example} \label{boxtimesdefn}
	Generalizing the categories $[n]$ of strings of $n$ composable morphisms, we can define double categories $[m] \boxtimes [n]$ given by $(m \times n)$-grids of squares:
  	
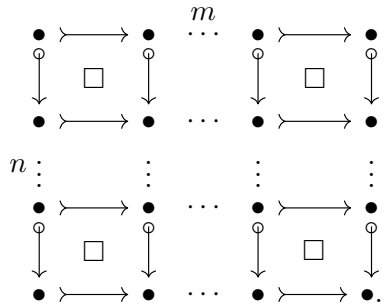
\[ \begin{tikzcd}
      \bullet \dar[etail] \rar[tail] & \bullet \dar[etail] \rar[phantom]{\cdots} \rar[phantom, shift left=3]{m} & \bullet \dar[etail] \rar[tail] & \bullet \dar[etail] \\
      \bullet \dar[phantom]{\vdots} \dar[phantom, shift right=3]{n} \rar[tail] & \bullet \dar[phantom]{\vdots} \dar[phantom, shift right=3]{} \rar[phantom]{\cdots} \rar[phantom, shift left=3]{} & \bullet \dar[phantom]{\vdots} \dar[phantom, shift left=3]{} \rar[tail] & \bullet \dar[phantom]{\vdots} \dar[phantom, shift left=3]{} \\
      \bullet \dar[etail] \rar[tail] & \bullet \dar[etail] \rar[phantom]{\cdots} \rar[phantom, shift right=3]{} & \bullet \dar[etail] \rar[tail] & \bullet \dar[etail] \\
      \bullet \rar[tail] & \bullet \rar[phantom]{\cdots} \rar[phantom, shift
      right=3]{} & \bullet \rar[tail] & \bullet.  \dist{1-1}{2-2}
      \dist{1-3}{2-4} \dist{3-1}{4-2} \dist{3-3}{4-4}
    \end{tikzcd} \]
  	The idea is that $[m] \boxtimes [n]$ is the double category freely generated by such a diagram, in the sense that we assume the existence of all composites of the various morphisms and squares; see \cite[1.34]{BOORSobj} for more details.
\end{example}

\begin{example}
  	For a double category $\D$ there is a double category $\D^t$, the \emph{transpose of $\D$}, that has the same objects and squares but with the roles of horizontal and vertical morphisms reversed.  Thus, for example, $([m] \boxtimes [n])^t = [n] \boxtimes [m]$.
\end{example}

\begin{example}
	For a double category $\D$ there are double categories $\D^\hop$, in which the direction of the horizontal morphisms is reversed, and $\D^\vop$, in which the direction of the vertical morphisms is reversed; in each case the squares are redirected accordingly. It is straightforward to check that there is also a double category $(\D^\vop)^\hop \cong (\D^\hop)^\vop$ in which both horizontal and vertical morphisms and squares are reversed.
\end{example}

\begin{example}\label{semidirect}
	Suppose that $\C$ and $\C'$ are categories with the same set of objects, and let $F \colon \C' \rto \C$ be a functor that is the identity on objects.  We can define a double category $\C\rtimes_F \C'$ with:
  	\begin{itemize}
  		\item the objects given by the objects of $\C$,
  		
  		\item the horizontal morphisms given by morphisms in $\C$,
  		
  		\item vertical morphisms given by morphisms in $\C'$, and
  		
  		\item squares given by diagrams
    	\[ \distsquare{A}{B}{C}{D}{f}{g}{h}{j} \]
   	 	such that the square 
    	\[ \begin{tikzcd}
        	A \rar{f} \dar[swap]{F(g)} & B \dar{F(h)} \\ C \rar{j} & D
      	\end{tikzcd}\]
    	commutes in $\C$.
  	\end{itemize}
	The special case when $\C'=\C$ and $F=\id_C$ is the double category in which both horizontal and vertical morphisms are the morphisms of $\C$ and the squares are commutative squares in $\C$.

	More generally, given two categories $\C$ and $\C'$ with the same objects, and two functors $F \colon \C \rto \mathcal{D}$ and $F' \colon \C' \rto \mathcal{D}$ that agree on objects, we can construct the double category $\C \mathrel{{}_{F}\!\bowtie_{F'}} \C'$ with objects, horizontal morphisms, and vertical morphisms as before, but now squares given by the diagrams
    \[ \distsquare{A}{B}{C}{D}{f}{g}{h}{j} \]
    such that the square
    \[\begin{tikzcd}
    F(A) \rar{F(f)} \dar{F'(g)} & F(B) \dar{F'(h)} \\ F(C) \rar{F(j)} & F(D)
    \end{tikzcd}\] 
    commutes in $\mathcal{D}$.
\end{example}

An important example of this construction is the following, which is a double categorical variant of the classifying diagram construction of Rezk, as described in Definition~\ref{classifyingdef}. 

\begin{definition}
	Given a category $\C$, its \emph{classifying double category} is $\C \rtimes_I \iso \C$, where $\iso \C$ is the maximal subgroupoid of $\C$ and $I \colon \iso \C \rightarrow \C$ is the natural inclusion.  
\end{definition}

In all of the above examples of double categories, the squares carry no information beyond commutativity of a square in an ambient category.  More generally, when there need not be such an ambient category, we can encode a similar property via the following definition.

\begin{definition}
  	A double category $\D$ is \emph{simple} if any square is uniquely determined by its boundary.  
\end{definition}

Simple double categories are useful for encoding compatibility conditions, but there are many examples of double categories that are not simple.  For example, double categories can be regarded as a generalization of $2$-categories, as follows.

\begin{example} \label{ex:2catdouble}
  	Given a $2$-category $\mathbf{C}$, we can construct a double category $\H(\mathbf{C})$ by defining:
  	\begin{itemize}
  		\item objects to be the objects of $\mathbf{C}$,
  		
  		\item horizontal morphisms to be the morphisms of $\mathbf{C}$,
  		
  		\item vertical morphisms to be identity morphisms, and
  		
  		\item squares to be $2$-morphisms in $\mathbf{C}$ from the top horizontal morphism of the square to the bottom one.
  	\end{itemize} 
	There is also a double category $\V(\mathbf{C})$ with the roles of horizontal and vertical morphisms are reversed, so we can define it as $\H(\mathbf{C})^t$. The squares in the double categories $\H(\mathbf{C})$ and $\V(\mathbf{C})$ have the property that one pair of parallel morphisms are identities.

We also write $\H(\C)$ and $\V(\C)$ when $\C$ is a 1-category, applying the above definition by regarding $\C$ as a 2-category with only identity 2-cells. All of the squares in $\H(\C)$ are identities on horizontal morphisms, while all of the squares in $\V(\C)$ are identities on vertical morphisms.

	%  	Conversely, a double category in which all vertical morphisms are identities gives a $2$-category via the reverse of this construction.
	%  	More generally, given any double category $\D$ it is possible to construct a $2$-category by setting
	%  	\begin{description}
	%  		\item[objects] the objects of $\D$,
	%  		
	%  		\item[$1$-morphisms] the horizontal morphisms of $\D$, and
	%  
	%  		\item[$2$-morphisms] a $2$-morphism from $f \colon A \rto B$ 		to $g \colon A \rto B$ is a
	%    		square
	%   			\[\distsquare{A}{B}{A}{B}{f}{1_A}{1_B}{g}.\]
	%  	\end{description}
\end{example}

\begin{definition}
	A square in a double category is \emph{horizontally globular} if both of its vertical morphisms are identities, as in the square:
	\[ \distsquare{A}{B}{A}{B}{f}{1_A}{1_B}{g}.\]
	Likewise, a square is \emph{vertically globular} if both of its horizontal morphisms are identities.
\end{definition}

\subsection{Recognizing classifying double categories}

Let us now focus on double categories of the form $\C \rtimes_I \iso \C$, as described in Example \ref{semidirect}, in which the squares are given by commutative squares in a category $\C$ in which the vertical morphisms are isomorphisms.  We now proceed to characterize such double categories using lifting properties.

\begin{definition} \label{ex:rho}
  	Let $\D$ be the double category $[1] \boxtimes [1]$.  Let $\D'$ be the double category with the same objects and morphisms, but with two nontrivial squares instead of one.  There is a double functor 
	\[ \flatten \colon \D' \rto \D \]
	given by the identity on objects, horizontal morphisms, and vertical morphisms.
\end{definition}

We claim that a double category is $\flatten$-local if and only if it is simple. A functor $\D' \rto \mathbb{C}$ is a choice of two squares that agree on the boundary, so it factors through $\flatten$ exactly when these two squares are identical.  Such a factorization exists for all squares in $\mathbb{C}$ precisely when $\mathbb{C}$ is simple. Moreover, since $\flatten$ is an epimorphism, a double category is $\flatten$-local if and only if it is strictly $\flatten$-local.

\begin{definition}
  	Let $\D$ be the double category that we might call the ``walking horizontal globular square'', given by the diagram
	\[ \begin{tikzcd}
      0 \rar[tail]{\alpha} \dar[etail,swap]{=} \ar[phantom]{dr}{\square} & 1 \dar[etail]{=} \\ 0 \rar[tail,swap]{\beta} & 1 
	\end{tikzcd} \]
	corresponding to the 2-category with a single 2-cell as described in Example~\ref{ex:2catdouble}.
	%obtained from the $2$-category with two objects $0$ and $1$, two morphisms $\alpha, \beta \colon 0 \rto 1$, and a $2$-morphism $\alpha \Rightarrow \beta$,  
	Considering $[1]$ as a $2$-category with only an identity 2-morphism, there is a double functor 
	\[ \truncateh \colon \D \rto \H([1]) \]
	given by the identity on objects.  
\end{definition}

Then a double category is $\truncateh$-local if and only if every horizontal globular square is the identity square for a horizontal morphism, in which case we say the double category is \emph{antiglobular}.

We can similarly define a double functor $\truncatev \colon \D^t \rto \V([1])$ such that a double category is $\truncatev$-local when every vertical globular square is the identity square for a vertical morphism.  We say that a double category is \emph{totally antiglobular} if it is $\{\truncateh,\truncatev\}$-local, so that it has no non-identity globular squares in either direction.

We now define two more particular morphisms.

\begin{definition}
  	Let $\D$ be the double category given by the diagram 
  	\[ \begin{tikzcd}
      0 \rar[tail] \dar[etail] \ar[phantom]{dr}{\square} & 1 \dar[etail]{=} \\ 1 \rar[tail]{=} & 1, 
    \end{tikzcd} \]
  	and let 
	\[ \rotater \colon \V([1]) \rto \D \] 
	be the functor induced by the identity on objects, which can be thought of as mapping the square of $\V([1])$ to the identity square on the left-hand vertical map of $\D$.

	Symmetrically, let $\D'$ be the double category given by the diagram 
  	\[ \begin{tikzcd}
      	0 \rar[tail]{=} \dar[etail,swap]{=} \ar[phantom]{dr}{\square} & 0 \dar[etail] \\ 0 \rar[tail] & 1,
    \end{tikzcd}\]
  	and let 
	\[\rotatel \colon \V([1]) \rto \D' \]
	be the functor induced by the identity on objects, now mapping the square to the identity square on the right-hand vertical map of $\D$. 
\end{definition}

Being local with respect to $\rotater$ enables one to migrate a vertical morphism to a horizontal morphism; being local with respect to $\rotatel$ is similar but puts the vertical morphism in a different position in the square. A double category that is local with respect to both $\rotater$ and $\rotatel$ such that the two squares associated to a vertical morphism $f$ compose to identity squares in both directions is sometimes called \emph{companionable}, where the single horizontal morphism corresponding to $f$ is its \emph{companion}. For our purposes, however, it suffices to be local only with respect to one or the other. 
%\jbnoteil{Trying to explain this intuitively - can probably be done much better.  Again, it would also be nice to have a concise name for this property - or maybe for when it is combined with $\truncateh$-local, as below.}\bsnoteil{Added a little here. I'm not sure how serious this ``companionable'' terminology is (I only saw it on nlab) but it seemed relevant to mention since `'all vertical morphisms having companions'' is a fairly common property.}

\begin{definition}
  	Let $\D$ be a double category that is strictly $\{\rotatel, \truncateh\}$-local.  The \emph{rotation functor of $\D$}, denoted by $R_\D \colon \D_v \rto \D_h$, is defined as follows.  On objects, $R_\D$ is the identity.  On a morphism $f \colon A \erto B$, we define $R_\D(f)$ to be the unique morphism $A \mrto B$ such that
  	\[\distsquare{A}{B}{B}{B}{R_\D(f)}{f}{=}{=}\]
  	is a square.  
\end{definition}

An analogous definition can be made for a double category that is strictly $\{\rotater,\truncateh\}$-local.  When $\D$ is understood from context we omit it from the notation.

\begin{lemma}
  	The rotation function $R_\D$ is well-defined.  Moreover, if $\D$ is both $\rotatel$-local and $\rotater$-local, then the two definitions of $R_\D$ agree.
\end{lemma}

\begin{proof}
  	Suppose that $\D$ is strictly $\{\rotater,\truncateh\}$-local.  If $f \colon A \erto B$ is any vertical morphism, we define $R_\D(f)$ to be the unique morphism such that
  	\[ \begin{tikzcd}
      	A \rar[tail]{R_\D(f)} \dar[etail] & B \dar[etail]{=} \\ B \rar[tail]{=} & B
    \end{tikzcd}\]
  	is a square.  To check that we obtain a functor, we must check that this definition respects composition and preserves identities.  First, suppose that $f = \id_A$.  Then the above square has the form
  	\[ \begin{tikzcd}
      	A \rar[tail]{R_\D(1_A)} \dar[etail]{=} & A \dar[etail]{=} \\ A \rar[tail]{=} & A.
    \end{tikzcd}\]
  	Since $\D$ is $\truncateh$-local, $R_\D(\id_A) = \id_A$, as desired.

  	Now suppose that we are given two morphisms $f \colon A \rto B$ and $g \colon B \rto C$.  Then there is a diagram of squares
  	\[ \begin{tikzcd}
      	A \rar[tail]{R_\D(f)} \dar[etail,swap]{f} & B \dar[etail]{=} \rar[tail]{R_\D(g)} & C \dar[etail]{=} \\
     	B \rar[tail]{=} \dar[etail,swap]{g}  & B \dar[etail]{g} \rar[tail]{R_\D(g)}&  C \dar[etail]{=}\\
      	C \rar[tail]{=} & C \rar[tail]{=} & C.
    \end{tikzcd}\] 
    The outside of the square has identities down the right side and across the bottom, and $gf$ across the left side.  Thus, since $\D$ is strictly $\rotater$-local, the morphism across the top must be $R_\D(gf)$; the fact that it is equal to $R_\D(g)R_\D(f)$ implies that $F$ is a functor, as desired.  The second case can be proved analogously.

  	Now suppose that $\D$ is both $\rotater$-local and $\rotatel$-local, and let $R^\vdash$ and $R^\dashv$ be the two functors defined by the structures.  To prove that they are equal, it suffices to check that they are equal on
  	morphisms.  Let $f \colon A \erto B$ be a morphism, and consider the outside square of the diagram
  	\[ \begin{tikzcd}
      	A \rar[equals] \dar[equals] & A \dar[etail]{f} \rar[tail]{R^\vdash(f)}  & B \dar[equals] \\
      	A \rar[tail]{R^\dashv(f)} & B \rar[equals] & B.
    \end{tikzcd}\]
  	Since $\D$ is $\truncateh$-local, the top and bottom of the outside square must be equal; thus $R^\vdash = R^\dashv$, as desired.
\end{proof}

\begin{definition}
  	A double category $\D$ is \emph{vertically flexible} if it is strictly $\{\rotater,\rotatel\}$-local. 
\end{definition}

\begin{proposition}
  	Suppose that $\D$ is simple, antiglobular, and vertically flexible. Then $R_\D$ induces an isomorphism of double categories $\iota \colon \D \rto \mathcal H(\D) \rtimes_R \mathcal V(\D)$.
\end{proposition}

\begin{proof}
  	We define $\iota$ to be the identity on objects and on morphisms.  To check that $\iota$ is well-defined it suffices to check that every square in $\D$ maps to a well-defined square in $\mathcal H(\D) \rtimes_R \mathcal V(\D)$.   Note that since both $\D$ and $\mathcal H(\D) \rtimes_R \mathcal V(\D)$ are simple, composition of squares is always compatible whenever it is defined.

  	Given a square
  	\[\distsquare{A}{B}{C}{D}{f}{g}{h}{j},\]
  	precomposing with the left square below, or postcomposing with the right square
  	\[ \begin{tikzcd}
      	A \rar[equals] \dar[equals] & A \dar[etail]{g} \\ A \rar[tail]{R(g)} & C
  	\end{tikzcd} \qquad
  	\begin{tikzcd}
   		B \rar[tail]{R(h)} \dar[etail]{h} & D \dar[equals] \\ D \rar[equals] & D 
  	\end{tikzcd} \]  
  	gives the square below left: 
  	\[ \begin{tikzcd}
    	A \dar[equals] \rar[tail]{R(h)f} & D \dar[equals] \\ A \rar[tail]{jR(g)}& D
  	\end{tikzcd}
  	\qquad\qquad\qquad
  	\begin{tikzcd}
    	A \dar[tail,swap]{R(g)} \rar[tail]{f} & B \dar[tail]{R(h)} \\ C \rar[tail]{j}& D.
    	\comm{1-1}{2-2}
  	\end{tikzcd} \]
 	Since $\D$ is antiglobular, the top and bottom arrows of this square are equal; thus $R(h)f = jR(g)$ and so we have a square in $\mathcal H(\D) \rtimes_R \mathcal V(\D)$ as above right. Therefore the functor $\iota$ is well-defined. 

	Now let $f \colon A \mrto B$, $j \colon C \mrto D$, $g \colon A \erto C$ and $h \colon B \erto D$ be morphisms such that $R(h)f = jR(g)$.  We would like to show that these morphisms assemble into a square in $\D$.  Such a square is exactly the vertical composite of the following two squares in $\D$, where by assumption the bottom arrow of the left composite and top arrow of the right composite agree:
  	\[ \begin{tikzcd}
      	A \rar[tail]{f} \dar[equals] & B \dar[equals] \rar[equals] & B \dar[etail]{h} \\
      	A \rar[tail]{f} & B \rar{R(h)} & D
    \end{tikzcd}\qquad\hbox{and}\qquad
    \begin{tikzcd}
      	A \rar[tail]{R(g)} \dar[etail,swap]{g} & C \rar[tail]{j} \dar[equals] & D \dar[equals] \\
      	C \rar[equals] & C \rar[tail]{j} & D.
    \end{tikzcd}\]
  	Thus $\iota$ is an isomorphism of double categories, as desired.
\end{proof}

This proposition gives a method for detecting categories of the form $\C_1 \rtimes_F \C_2$.  We thus immediately get the following observation. 

\begin{theorem}\label{classifyingcondition}
   	A double category $\D$ is isomorphic to the classifying double category of a category if is is simple, antiglobular, and vertically flexible, and the rotation functor is the inclusion of the maximal subgroupoid of $\mathcal H(\D)$.
\end{theorem}

In this discussion, we have discussed properties such as being ``vertically a groupoid,'' ``the inclusion of a subcategory,'' and ``the inclusion of the maximal subgroupoid''.  As an exercise in analyzing properties through lifting conditions, and because it is useful when making comparisons with simplicial objects later in the paper, we reformulate these conditions in terms of lifting properties.  We state these results without proof, as the arguments are very similar to others in this section.

\begin{definition} 
  	Let $\cE$ be the category with two objects $0$ and $1$, a single morphism $0 \rto 1$ and a single morphism $1 \rto 0$, and let $\invert \colon [1] \rto \cE$ be the functor that is the identity on objects.
\end{definition}

\begin{lemma}\label{invertiblelift}
  	A category $\C$ is a groupoid if and only if it is $\invert$-local.  Consequently, for a double category $\D$, its vertical category $\mathcal V(\D)$ is a groupoid if and only if it is $\{\V(\invert) \colon \V([1]) \rto \V(\cE)\}$-local. 
\end{lemma}

% \begin{proof}
%   A functor $[1] \rto \C$ picks out a single morphism in $\C$.  If $\C$ is
%   $\iota$-local, this means that any morphism of $\C$ is invertible---i.e. that
%   $\C$ is a groupoid.  Conversely, if $\C$ is a groupoid then any morphism $[1]
%   \rto \C$ extends to a morphism $\E \rto \C$ by sending the morphism $1 \rto 0$
%   to the inverse of the morphism which is the image of $0 \rto 1$.
  
%   The second part follows immediately from the first.
% \end{proof}

Since the functor $R$ is defined as the identity on objects, it is the inclusion of a subcategory exactly when it is faithful.

\begin{definition} 
	Let $\doublerotate$ denote the identity-on-objects double functor depicted by
  	\[ \doublerotate \colon
	\begin{tikzcd}
      	1 \rar[equals] & 1 \dar[equals] \\
      	0 \dar[etail,swap]{f} \ar[u,etail,swap]{g} \rar[tail]{h} & 1 \dar[equals] \\
      	1 \rar[equals] & 1
    \end{tikzcd} \rightarrow
    \begin{tikzcd}
     	0 \rar[tail] \dar[etail] & 1 \dar[equals] \\ 1 \rar[equals] & 1.
    \end{tikzcd} \]
\end{definition}

\begin{lemma}
  	Let $\D$ be a double category that is simple, antiglobular, and vertically flexible.  The rotation functor $R$ is faithful if and only if $\D$ is local with respect to $\doublerotate$.
\end{lemma} 

It remains to describe a lifting property for the inclusion of the maximal subgroupoid.  As we have already given a lifting property for the property that $\mathcal V(\D)$ is a groupoid, it suffices to check maximality.

\begin{definition}
	Let $\hiso$ denote the inclusion double functor
  	\[\hiso \colon \H(\E) \rto \E\rtimes_{\id_\E} \E\]
	of the horizontal isomorphism into the double category of morphisms and commutative squares in the category $\E$.
\end{definition}

\begin{lemma}
  	Let $\D$ be a double category with $\mathcal V(\D)$ a groupoid that is simple, antiglobular, and vertically flexible.  Then the rotation functor is the inclusion of the maximal subgroupoid if it is $\{\doublerotate,\hiso\}$-local.
\end{lemma}

We have now completely characterized which double categories are classifying double categories using lifting properties. However, we can reduce the number of lifting properties one has to check.

\begin{proposition}\label{eitherconnection} 
  	Suppose that $\D$ is a simple, antiglobular double category in which $\mathcal V(\D)$ is a groupoid.  Then $\D$ is $\rotatel$-local if and only if it is $\rotater$-local.
\end{proposition}

In particular, to check vertical flexibility it is only necessary to check one of $\rotater$ and $\rotatel$

\begin{proof}
  	The forward implication is direct from the definitions.  We thus focus on proving that if $\D$ is strictly $\{\rotater,\truncateh,\flatten\}$-local then it is vertically flexible; the other case has an analogous proof.
  
  	Since $\D$ is strictly $\rotater$-local, any morphism $f \colon A \erto B$ extends to a square
  	\[ \begin{tikzcd}
      	A \rar[tail]{R(f)} \dar[etail]{f} & B \dar[equals] \\ B \rar[equals] & B.
    \end{tikzcd}\]
  	Since $f$ is an isomorphism, so is $R(f)$.  Thus the diagram 
  	\[ \begin{tikzcd}
      	B \dar[etail,swap]{f^{-1}} \ar[rr,equals] & & B \dar[etail]{f^{-1}} \\
      	A \rar[tail]{R(f)} \dar[etail,swap]{f} & B \rar[tail]{R(f)^{-1}} \dar[equals] & A \dar[equals] \\
      	B \rar[equals] & B \rar[tail]{R(f)^{-1}} & B
    \end{tikzcd}\]
    exists, and the outside square establishes that $\D$ is $\rotatel$-local.  Moreover, this second square is unique, because this process is reversible; a square with all identities on the boundary must be an identity square, by $\flatten$-locality.  Thus $\D$ is strictly $\rotatel$-local, as desired.
\end{proof}

In summary then, classifying double categories have the following characterization.

\begin{corollary}\label{classifyingliftingcondition}
	A double category $\D$ is isomorphic to the classifying double category of its horizontal category if and only if it is $\{\flatten,\truncateh,\V(\invert),\doublerotate,\hiso\}$-local and either strictly $\rotater$-local or strictly $\rotatel$-local.
\end{corollary}

If either of these equivalent conditions holds, $\D$ is strictly local with respect to
\[ \{\flatten, \truncateh, \rotater, \rotatel, \V(\invert), \doublerotate, \hiso\}, \] 
since $\flatten, \truncateh, \V(\invert), \doublerotate$ are epimorphisms of double categories and $\hiso$ is an epimorphism among strictly $\{\truncateh,\rotater,\rotatel,\doublerotate\}$-local double categories. 

We conclude this section by giving a name to this collection of maps.

\begin{definition}
	Let $S_{\cl}$ denote the set $\{\flatten, \truncateh, \rotater, \rotatel, \V(\invert), \doublerotate, \hiso\}$. 
\end{definition}

\section{Classifying triple categories of double categories with shared isomorphisms} \label{classifyingtcsection}

In this section, our goal is to develop a similar construction as in the previous section, so we can associate to any double category a classifying triple category.  A key point here is that, to make sense of a classifying diagram construction, we need the horizontal and vertical categories to have the same isomorphisms, in a fairly strict sense.  We begin by defining this notion, and then give an introductory treatment of triple categories before making the definition of a classifying triple category and characterizing their properties.

\subsection{Double categories with shared isomorphisms}

Let us now specify the sense in which we want the horizontal and vertical categories to have the same isomorphisms. We write $\mathsf{HIso}(A,B)$ and $\mathsf{VIso}(A,B)$ for the sets of horizontal and vertical isomorphisms, respectively, between objects $A$ and $B$ in a double category.

\begin{definition} \label{sharedisos}
	A double category \emph{has shared isomorphisms} if every horizontal isomorphism $f \colon A \mrto[\cong] B$ has a corresponding vertical isomorphism $\varphi(f) \colon A \erto[\cong] B$, such that $\varphi \colon \mathsf{HIso}(A,B) \cong \mathsf{VIso}(A,B)$ is a bijection for all objects $A,B$, the square of horizontal arrows
	\[ \begin{tikzcd}
		A \rar[tail]{f} \dar[tail]{\vcong}[swap]{h} & A' \dar[tail]{k}[swap]{\vcong} \\
		B \rar[tail,swap]{h} & B' 
	\end{tikzcd} \]
	commutes if and only if there is a square 
	\[ \begin{tikzcd}
		A \rar[tail]{f} \dar[etail]{\vcong}[swap]{\varphi(h)} & A' \dar[etail]{\varphi(k)}[swap]{\vcong} \\
		B \rar[tail,swap]{g} & B' 
		\dist{1-1}{2-2}
	\end{tikzcd} \]
	between the corresponding vertical isomorphisms, this square is unique relative to its boundary, and likewise for commutative squares between vertical isomorphisms.
\end{definition}

In particular, in a double category with shared isomorphisms it makes sense to talk about isomorphisms and natural isomorphisms without specifying whether they are horizontal or vertical, as they can be regarded as both.

More precisely, given a double category $\D$, there are two notions of a natural morphism of squares, which we can depict via cubical diagrams
\begin{equation}\label{naturalcubes}
	\begin{tikzcd}[row sep=scriptsize,column sep=scriptsize]
	&[18] \bullet \ar[tail,color=red]{dl} \ar[tail]{rr} \ar[etail]{dd} &[-18] &[18] \bullet \ar[tail,color=red]{dl} \ar[etail]{dd} \\
	\bullet \ar[etail]{dd} & & \bullet \\
	& \bullet \ar[tail,color=red]{dl} \ar[tail]{rr} & & \bullet \ar[tail,color=red]{dl} \\
	\bullet \ar[tail]{rr} & & \bullet 
	\ar[phantom,from={2-1},to={4-3},"\square"] \ar[phantom,from={1-2},to={3-4},"\square"]
	\ar[phantom,from={1-2},to={4-1},color=red,"\parallelogramvert"] \ar[phantom,from={1-4},to={4-3},color=red,"\parallelogramvert"]
	\ar[phantom,from={1-2},to={2-3},color=red,"\circlearrowleft"] \ar[phantom,from={3-2},to={4-3},color=red,"\circlearrowleft"]
	\ar[tail,from={2-1},to={2-3},crossing over] \ar[etail,from={2-3},to={4-3},crossing over]
	\end{tikzcd} 
	\qquad\qquad\qquad
	\begin{tikzcd}[row sep=scriptsize,column sep=scriptsize]
	&[18] \bullet \ar[etailhor,color=red]{dl} \ar[tail]{rr} \ar[etail]{dd} &[-18] &[18] \bullet \ar[etailhor,color=red]{dl} \ar[etail]{dd} \\
	\bullet \ar[etail]{dd} & & \bullet \\
	& \bullet \ar[etailhor,color=red]{dl} \ar[tail]{rr} & & \bullet \ar[etailhor,color=red]{dl} \\
	\bullet \ar[tail]{rr} & & \bullet .
	\ar[phantom,from={2-1},to={4-3},"\square"] \ar[phantom,from={1-2},to={3-4},"\square"]
	\ar[phantom,from={1-2},to={4-1},color=red,"\circlearrowleft"] \ar[phantom,from={1-4},to={4-3},color=red,"\circlearrowleft"]
	\ar[phantom,from={1-2},to={2-3},color=red,"\parallelogram"] \ar[phantom,from={3-2},to={4-3},color=red,"\parallelogram"]
	\ar[tail,from={2-1},to={2-3},crossing over] \ar[etail,from={2-3},to={4-3},crossing over]
	\end{tikzcd} 
\end{equation}

In the left-hand cube, we ask that the red morphisms be horizontal morphisms so that the top and bottom faces commute in the horizontal category, and the front and back faces be squares in $\D$.  In the right-hand cube, we instead ask that the red morphisms be vertical morphisms so that the front and back faces commute in the vertical category, and the top and bottom faces be squares in $\D$.  If we restrict to natural isomorphisms, then the definition of shared isomorphisms guarantees that these two notions coincide.  %This observation is key to making sense of our next definition. \jbnoteil{Will need to change the wording, now that this definition has moved...}

Recall the identity-on-objects double functor $\hiso \colon \H(\E) \to \E \rtimes_{\id_\E} \E$ and consider the analogous functor $\viso \colon \V(\E) \to \E \rtimes_{\id_\E} \E$ that includes the vertical isomorphism into $\E \rtimes_{\id_\E} \E$. The double category $\E \rtimes_{\id_\E} \E$ contains the non-identity squares
%\footnoteil{The double category also contains the additional 4 squares with the roles of 0 and 1 reversed, among others, but we do not explicitly need them.} \jbnoteil{I'm confused by the purpose of this footnote.}\bsnoteil{I was worried about giving the false impression that these are all of the non-identity squares in this double category, but if you don't think that's an issue I'm happy to remove it.}
\begin{equation}\label{walkingsharediso}
	\distsquare{0}{1}{1}{1}{}{}{}{} \qquad
	\distsquare{0}{0}{0}{1}{}{}{}{} \qquad
	\distsquare{0}{1}{0}{0}{}{}{}{} \qquad
	\distsquare{0}{0}{1}{0,}{}{}{}{} \qquad 
\end{equation} 
so that a double functor of the form $\E \rtimes_{\id_\E} \E \to \D$ encodes a horizontal isomorphism $f$ and vertical isomorphism $f'$ in $\D$ equipped with squares in $\D$
\begin{equation}\label{basicsquares}
	\distsquare{A}{B}{B}{B}{f}{f'}{\id_B}{\id_B} \qquad
	\distsquare{A}{A}{A}{B}{\id_A}{\id_A}{f'}{f} \qquad
	\distsquare{A}{B}{A}{A}{f}{\id_A}{f'^{-1}}{\id_A} \qquad
	\distsquare{A}{A}{B}{A.}{\id_A}{f'}{\id_A}{f^{-1}.} \qquad
\end{equation}
We call such a pair $(f,f')$ a \emph{shared isomorphism}. In a $\hiso$-local double category every horizontal isomorphism belongs to a shared isomorphism, while in a $\viso$-local double category the same holds for vertical isomorphisms. If a double category is strictly $\{\hiso,\viso\}$-local, then there is a bijection $\varphi \colon \mathsf{HIso}(A,B) \cong \mathsf{VIso}(A,B)$ for each pair of objects $A,B$ and every isomorphism belongs to a unique shared isomorphism of the form $(f,\varphi(f))$.

\begin{proposition}\label{sharedisochar}
	A double category has shared isomorphisms if and only if it is totally antiglobular and every horizontal or vertical isomorphism belongs to a unique shared isomorphism.
\end{proposition}

In other words, a double category has shared isomorphisms if and only if it is strictly local with respect to the set $\{\truncateh, \truncatev, \hiso, \viso\}$. 

\begin{proof}
	A double category with shared isomorphisms is totally antiglobular as any square in $\D$ between horizontal or vertical identities corresponds to a commutative square between identities, which is always an identity square. From Definition~\ref{sharedisos}, every horizontal or vertical isomorphism belongs to a pair $(f,\varphi(f))$; this pair is a shared isomorphism as the squares in \eqref{basicsquares} correspond under Definition~\ref{sharedisos} to commutative squares in $\mathcal H(\D)$.

	For the converse, we want to show that a double category $\D$ has shared isomorphisms if it is equipped with a bijection $\varphi \colon \mathsf{HIso}(A,B) \cong \mathsf{VIso}(A,B)$ for all objects $A,B$ such that each pair $(f,\varphi(f))$ is a shared isomorphism.  Specifically, we show there is a correspondence between squares in $\D$ of the form
	\[ \begin{tikzcd}
		\bullet \rar[tail]{f} \dar[etail]{\vcong}[swap]{\varphi(h)} & \bullet \dar[etail]{\varphi(k)}[swap]{\vcong} \\
		\bullet \rar[tail,swap]{g} & \bullet 
		\dist{1-1}{2-2}
	\end{tikzcd}
	\qquad\textrm{and}\qquad
	\begin{tikzcd}
		\bullet \rar[tail]{f} \dar[tail,swap]{h} & \bullet \dar[tail]{k} \\
		\bullet \rar[tail,swap]{g} & \bullet,
		\comm{1-1}{2-2}
	\end{tikzcd} \]
	where an analogous argument provides the additional correspondence with the roles of horizontal and vertical morphisms reversed. Note that such a commutative square of horizontal morphisms exists if and only if $k \circ f = g \circ h$. 
		
	From a distinguished square as above left, we get a composite square
	\[ \begin{tikzcd}[row sep=40,column sep=40]
		\bullet \rar[equals] \dar[equals] & \bullet \rar[tail]{f} \dar[etail,swap]{\varphi(h)} & \bullet \rar[tail]{k} \dar[etail]{\varphi(k)} & \bullet \dar[equals] \\
		\bullet \rar[tail,swap]{h} & \bullet \rar[tail,swap]{g} & \bullet \rar[equals] & \bullet. 
		\dist{1-1}{2-2} \dist{1-2}{2-3} \dist{1-3}{2-4} 
	\end{tikzcd} \]
	By antiglobularity, this composite is a vertical identity square, so we can conclude that $k \circ f = g \circ h$. As the squares on the left and right of the composite are isomorphisms in $\mathcal H(\D_v)$, the middle square must be unique relative to its boundary, so to complete the desired correspondence it suffices to assume that $k \circ f = g \circ h$ and build the composite square
	\[ \begin{tikzcd}
		\bullet \rar[tail]{f} \dar[equals] & \bullet \rar[equals] \dar[equals] & \bullet \dar[etail]{\varphi(k)} \\
		\bullet \rar[tail,swap]{f} \dar[equals] & \bullet \rar[tail,swap]{k} & \bullet \dar[equals] \\
		\bullet \rar[tail]{h} \dar[etail,swap]{\varphi(h)} & \bullet \rar[tail]{g} \dar[equals] & \bullet \dar[equals] \\
		\bullet \rar[equals] & \bullet \rar[tail,swap]{g} & \bullet. 
		\dist{1-1}{2-2} \dist{1-2}{2-3} \dist{2-1}{3-3} \dist{3-1}{4-2} \dist{3-2}{4-3}
	\end{tikzcd} \]
\end{proof}

\subsection{Triple categories}

To make an analogous construction for double categories, we need to work on the level of triple categories, which we now introduce briefly.  Triple categories do not seem to have been described explicitly in literature in the way that we describe them here, but the notion is not new.  They can be taken to be the $n=3$ case of the $n$-fold categories of \cite[Definition 2.2]{nfold}; they are also a special case of the intercategories of \cite[\S 1]{intercategories}, in which all associativity, unit, and interchange equations hold strictly.

%\jbnoteil{What do you think of the above?  I took the material out of the footnote and added a bit more exposition.  It seems reasonable to have in the intro paragraph of the section.}

\begin{definition} 
	A \emph{triple category} $\T$ is a category internal to double categories.  As such, it consists of a triple of categories with the same objects whose morphisms are called \emph{horizontal morphisms}, \emph{vertical morphisms}, and \emph{transverse morphisms}; composable \emph{squares} between each pair of morphism types, called \emph{hv-, ht-}, and \emph{vt-squares} forming three double categories,
	\[ \begin{tikzcd}
		\bullet \rar[tail] \dar[etail] & \bullet \dar[etail] \\
		\bullet \rar[tail] & \bullet
		\dist{1-1}{2-2}
	\end{tikzcd}
	\qquad\qquad
	\begin{tikzcd}
		\bullet \rar[tail] \dar & \bullet \dar \\
		\bullet \rar[tail] & \bullet
		\dist{1-1}{2-2}
	\end{tikzcd}
	\qquad\qquad
	\begin{tikzcd}
		\bullet \rar[etailhor] \dar & \bullet \dar \\
		\bullet \rar[etailhor] & \bullet,
		\dist{1-1}{2-2}
	\end{tikzcd} \]
	respectively; and a notion of \emph{cubes} of the form
	
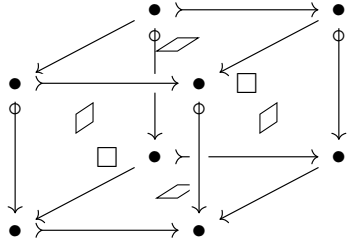
\[ \begin{tikzcd}[row sep=scriptsize,column sep=scriptsize]
		&[18] \bullet \ar{dl} \ar[tail]{rr} \ar[etail]{dd} &[-18] &[18] \bullet \ar{dl} \ar[etail]{dd} \\
		\bullet \ar[etail]{dd} & & \bullet \\
		& \bullet \ar{dl} \ar[tail]{rr} & & \bullet \ar{dl} \\
		\bullet \ar[tail]{rr} & & \bullet 
		\ar[phantom,from={2-1},to={4-3},"\square"] \ar[phantom,from={1-2},to={3-4},"\square"]
		\ar[phantom,from={1-2},to={4-1},"\parallelogramvert"] \ar[phantom,from={1-4},to={4-3},"\parallelogramvert"]
		\ar[phantom,from={1-2},to={2-3},"\parallelogram"] \ar[phantom,from={3-2},to={4-3},"\parallelogram"]
		\ar[tail,from={2-1},to={2-3},crossing over] \ar[etail,from={2-3},to={4-3},crossing over]
	\end{tikzcd} \]
	that are composable in all three directions subject to unit, associativity, and interchange conditions.
\end{definition}

	% 	A square in a double category is \emph{trivial} if it represents the identity morphism in the Hom-category. \jbnoteil{do we want to say more about the internal category viewpoint if we are going to talk about it this way?}

If $\T$ is a triple category, let $\mathcal H(\T), \mathcal V(\T)$, and $\mathcal T(\T)$ be the constituent categories of horizontal, vertical, and transverse morphisms, respectively, and let $\mathbb{HV}(\T), \mathbb{HT}(\T)$, and $\mathbb{VT}(\T)$ be the double categories associated to each pair of morphism types. We also denote by $\mathbb{HV}(\T_t), \mathbb{HT}(\T_v)$, and $\mathbb{VT}(\T_h)$ the double categories with objects the transverse, vertical, and horizontal morphisms, respectively, of $\T$ and whose squares are cubes in $\T$. For instance, in $\mathbb{HV}(\T_t)$, the objects are transverse morphisms of $\T$, the horizontal morphisms are the ht-squares in $\T$, the vertical morphisms are the vt-squares in $\T$, and the squares are cubes in $\T$.

\begin{example}
	We can define triple categories $[\ell] \boxtimes [m] \boxtimes [n]$ freely generated by an $(\ell \times m \times n)$-grid of cubes. It contains a composite cube for each $(\ell' \times m' \times n')$-subgrid.%, along with squares when one of $\ell',m',n'$ are 0 and and morphisms when one or two, respectively, out of $\ell',m',n'$ are 0. 
%\jbnoteil{I think too much is crammed into this sentence with too much assumed in ``respectively"}\bsnoteil{Do you think this is enough detail now? I was going to separate out the description of composite morphisms and squares but I'm not sure it actually adds much intuitively. Another option would be to separately describe the horizontal, vertical, and transverse categories.}
\end{example}

\begin{example}
  	For a triple category $\T$ there is a triple category $\T^\sigma$ for any permutation $\sigma$ of the set $\{h,v,t\}$ that swaps the roles of the three types of morphisms. For example, $([\ell] \boxtimes [m] \boxtimes [n])^{(ht)} = [n] \boxtimes [m] \boxtimes [\ell]$.
\end{example}

\begin{example}
	For a category $\C$ with three subcategories $\C_1,\C_2$, and $\C_3$, each containing all of the objects of $\C$, there is a triple category $\C_1 \bowtie \C_2 \bowtie \C_3$ with $\C_1,\C_2$, and $\C_3$ as the horizontal, vertical, and transverse categories, respectively; squares given by commutative squares in $\C$ between morphisms of the appropriate pair of subcategories; and cubes given by commutative cubes in $\C$ whose horizontal, vertical, and transverse arrows come from $\C_1$, $\C_2$, and $\C_3$, respectively.
\end{example}

\begin{example}\label{doubletotriple}
	For a double category $\D$ there are three standard ways to functorially regard it as a triple category. 
	%\jbnoteil{the latter two notations don't make sense to me}\bsnoteil{I think when these are used the double category is one where the vertical arrows are really behaving more like transverse arrows than vertical ones: unlike a CGW category where horizontal and vertical are on equal footing, in a classifying double category the vertical arrows correspond to the transverse arrows in a classifying triple category, so both $\H$ and $\V$ treat them as transverse while the horizontal arrows become either horizontal or vertical. At least that's what I was thinking when I wrote these down, but looking now it does seem confusing why it's like that. Do you think it would make sense to change how we talk about certain double categories to match this, maybe calling the vertical direction transverse instead?}
	\begin{itemize}
		\item Let $\mathfrak{HV}(\D)$ denote the triple category that agrees with $\D$ in the horizontal and vertical directions and has only identities in the transverse direction.
	
		\item Let $\mathfrak{HT}(\D)$ be the triple category with the same horizontal morphisms as $\D$ and the vertical morphisms of $\D$ as transverse morphisms and have only identities in the vertical direction.
	
		\item Let $\mathfrak{VT}(\D)$ be the triple category with the horizontal morphisms of $\D$ as vertical morphisms and the vertical morphisms of $\D$ as transverse, with only identities in the horizontal direction.
	\end{itemize}
\end{example}

\begin{example}\label{triplebox}
	If $\D$ is a double category, let $[1] \boxtimes_t \D$ be the triple category $\mathfrak{HV}(\D) \times ([0] \boxtimes [0] \boxtimes [1])$ consisting of two copies of $\mathfrak{HV}(\D)$ and a transverse natural transformation between them. Similarly, denote by $[1] \boxtimes_v \D$ the double category $\mathfrak{HT}(\D) \times ([0] \boxtimes [1] \boxtimes [0])$ consisting of two copies of $\mathfrak{HT}(\D)$ and a vertical natural transformation between them.  Finally, we can define $[1] \boxtimes_h \D$ analogously as $\mathfrak{VT}(\D) \times ([1] \boxtimes [0] \boxtimes [0])$.
\end{example}

\begin{definition}
	For a double category $\D$ with shared isomorphisms, the \emph{classifying triple category} of $\D$ has $\D$ as its double category of horizontal and vertical morphisms, isomorphisms in $\D$ as transverse morphisms, commutative squares in $\mathcal H(\D)$ and $\mathcal V(\D)$ between isomorphisms as ht- and vt-squares, and as cubes those of one of the equivalent forms of \eqref{naturalcubes}.
\end{definition}

\subsection{Recognizing classifying triple categories}

We first show how classifying triple categories can be characterized in terms of classifying double categories.

\begin{proposition}\label{triplechar}
	A triple category $\T$ is isomorphic to the classifying triple category of a double category with shared isomorphisms $\D$ if and only if $\mathbb{HT}(\T), \mathbb{VT}(\T), \mathbb{HT}(\T_v)$, and $\mathbb{VT}(\T_h)$ are each isomorphic to the classifying double category of a category. 
\end{proposition}

\begin{proof}
	If $\T$ is isomorphic to the classifying triple category of $\D$, then $\mathbb{HT}(\T), \mathbb{VT}(\T), \mathbb{HT}(\T_v)$, and $\mathbb{VT}(\T_h)$ are isomorphic to the classifying diagrams of respectively the categories $\mathcal H(\D), \mathcal V(\D), \mathcal H(\D_v)$, and $\mathcal V(\D_h)$, where $\mathcal H(\D_v)$ and $\mathcal V(\D_h)$ denote the categories of vertical and horizontal morphisms, respectively, in $\D$ with squares in $\D$ as morphisms. It remains only to prove the converse statement.  

	Assume $\T$ satisfies the given conditions, and let $\D$ be the double category $\mathbb{HV}(\T)$. The double category $\mathbb{HT}(\T)$ is assumed to be isomorphic the classifying double category of the category $\mathcal H(\T)$, so the category $\mathcal T(\T)$ must agree with the groupoid of horizontal isomorphisms of $\T$, and hence of $\D$, and the ht-squares of $\T$ must agree with the commutative squares between horizontal morphisms in $\D$ with one opposite pair invertible.  The analogous argument for $\mathbb{VT}(\T)$ shows that $\mathcal T(\T)$ must also agree with the groupoid of vertical isomorphisms of $\D$, and the vt-squares agree with the appropriate commutative squares of vertical morphisms in $\D$.
	
	In the following diagram, which shows visual representations of the different features in the triple category $\T$ and arrows denoting the various source and target maps between them, we have so far shown all of the pictured isomorphisms:
	\begin{lrbox}{\mrtonested}
	\raisebox{.15\height}{\scalebox{.8}{\begin{tikzcd}[ampersand replacement=\&] \bullet \rar[tail] \& \bullet \end{tikzcd}}}
	\end{lrbox}
	\begin{lrbox}{\ertonested}
	\raisebox{.15\height}{\scalebox{.7}{\begin{tikzcd}[ampersand replacement=\&] \bullet \dar[etail] \\ \bullet \end{tikzcd}}}
	\end{lrbox}
	\begin{lrbox}{\hvsquare}
	\raisebox{.15\height}{\scalebox{.7}{\distsquare{\bullet}{\bullet}{\bullet}{\bullet}{}{}{}{}}}
	\end{lrbox}
	\begin{lrbox}{\hisquare}
	\raisebox{.15\height}{\scalebox{.7}{\begin{tikzcd}[ampersand replacement=\&] \bullet \rar[tail] \dar[swap]{\cong} \ar[phantom]{dr}{\circlearrowleft} \& \bullet \dar{\cong} \\ \bullet \rar[tail] \& \bullet \end{tikzcd}}}
	\end{lrbox}
	\begin{lrbox}{\visquare}
	\raisebox{.15\height}{\scalebox{.7}{\begin{tikzcd}[ampersand replacement=\&] \bullet \dar[etail] \rar{\cong} \ar[phantom]{dr}{\circlearrowleft} \& \bullet \dar[etail] \\ \bullet \rar[swap]{\cong} \& \bullet \end{tikzcd}}}
	\end{lrbox}
	\begin{lrbox}{\hvicube}
	\raisebox{.15\height}{\scalebox{.6}{\begin{tikzcd}[ampersand replacement=\&] 
	\& \bullet \ar[tail]{rr} \ar[etail]{dd} \ar[phantom]{ddrr}{\square} \&[-22] \& \bullet \ar[etail]{dd} \\[-11]
	\bullet \ar[tail, crossing over]{rr} \ar[etail]{dd} \ar{ur}[swap,outer sep=-2]{\cong} \ar[phantom]{ddrr}{\square} \& \& \bullet \ar{ur}[swap,outer sep=-2]{\cong} \\ 
	\& \bullet \ar[tail]{rr} \& \& \bullet \\[-11]
	\bullet \ar[tail]{rr} \ar{ur}[swap,outer sep=-2]{\cong} \& \& \bullet \ar{ur}[swap,outer sep=-2]{\cong}
	\ar[etail,crossing over,from={2-3},to={4-3}]
	\end{tikzcd}}}
	\end{lrbox}
	\begin{equation} \label{classifyingtc}
	\begin{tikzcd}[column sep=-20]
	& \mor(\mathcal T(\T)) \cong \left\{ \bullet \cong \bullet \right\}
	\ar[shift left=1,shorten =5]{dl} \ar[shift right=1,shorten =5]{dl} 
	&[-25] &[-15] \squares(\mathbb{HT}(\T))  \cong \left\{ \usebox{\hisquare} \right\}
	\ar[shift left=1]{ll} \ar[shift right=1]{ll} \ar[shift left=1,shorten <=-10]{dl} \ar[shift right=1,shorten <=-10]{dl} \\
	\ob(\T) \cong \{ \bullet \}
	& & \mor(\mathcal H(\T)) \cong \left\{ \usebox{\mrtonested} \right\} \\
	& \squares(\mathbb{VT}(\T))  \cong \left\{ \usebox{\visquare} \right\}
	\ar[shift left=1,shorten <=-10]{uu} \ar[shift right=1,shorten <=-10]{uu} \ar[shift left=1,shorten <=-15]{dl} \ar[shift right=1,shorten <=-15]{dl} 
	& & \left\{ \usebox{\hvicube} \right\} \ar[shift left=1]{ll} \ar[shift right=1]{ll} \ar[shift left=1,shorten >=-5]{uu} \ar[shift right=1,shorten >=-5]{uu} \ar[shift left=1,shorten <=-5]{dl} \ar[shift right=1,shorten <=-5]{dl} 
	\\
	\mor(\mathcal V(\T)) \cong \left\{ \usebox{\ertonested} \right\} 
	\ar[shift left=1,shorten <=-5]{uu} \ar[shift right=1,shorten <=-5]{uu} 
	& & \squares(\mathbb{HV}(\T)) \cong \left\{ \usebox{\hvsquare} \right\}
	\ar[shift left=1]{ll} \ar[shift right=1]{ll} \ar[shift left=1,crossing over,shorten <=-5]{uu} \ar[shift right=1,crossing over,shorten <=-5]{uu} 
	\ar[shift left=1,from={2-3},to={2-1},crossing over] \ar[shift right=1,from={2-3},to={2-1},crossing over]
	\end{tikzcd}. 
	\end{equation}
		
	It remains to show that $\D$ has shared isomorphisms and the cubes of $\T$ agree with those depicted informally above and more precisely below. The double category $\mathbb{HT}(\T_v)$ has as its squares the cubes of $\T$. As $\mathbb{HT}(\T_v)$ is assumed to be a classifying double category, namely that of the category with objects vertical morphisms in $\T$ and morphisms hv-squares, the cubes of $\T$ must have the form of the diagram on the left in Figure~\ref{classifyingcubes}. 
	\begin{figure}
		\[\begin{tikzcd}[row sep=scriptsize,column sep=scriptsize]
			&[18] \bullet \ar[tail,color=red]{dl}[swap]{\cong} 	\ar[tail]{rr} 	\ar[etail]{dd} &[-18] &[18] \bullet \ar[tail,color=red]{dl}[swap]{\cong} \ar[etail]{dd} \\
			\bullet \ar[etail]{dd} & & \bullet \\
			& \bullet \ar[tail,color=red]{dl}[swap]{\cong} \ar[tail]{rr} & & \bullet \ar[tail,color=red]{dl}[swap]{\cong} \\
			\bullet \ar[tail]{rr} & & \bullet 
			\ar[phantom,from={2-1},to={4-3},"\square"] \ar[phantom,from={1-2},to={3-4},"\square"]
			\ar[phantom,from={1-2},to={4-1},color=red,"\parallelogramvert"] \ar[phantom,from={1-4},to={4-3},color=red,"\parallelogramvert"]
			\ar[phantom,from={1-2},to={2-3},color=red,"\circlearrowleft"] \ar[phantom,from={3-2},to={4-3},color=red,"\circlearrowleft"]
			\ar[tail,from={2-1},to={2-3},crossing over] \ar[etail,from={2-3},to={4-3},crossing over]
		\end{tikzcd} 
		\qquad\qquad\qquad
		\begin{tikzcd}[row sep=scriptsize,column sep=scriptsize]
			&[18] \bullet \ar[etailhor,color=red]{dl}[swap]{\cong} \ar[tail]{rr} \ar[etail]{dd} &[-18] &[18] \bullet \ar[etailhor,color=red]{dl}[swap]{\cong} \ar[etail]{dd} \\
			\bullet \ar[etail]{dd} & & \bullet \\
			& \bullet \ar[etailhor,color=red]{dl}[swap]{\cong} \ar[tail]{rr} & & \bullet \ar[etailhor,color=red]{dl}[swap]{\cong} \\
			\bullet \ar[tail]{rr} & & \bullet .
			\ar[phantom,from={2-1},to={4-3},"\square"] \ar[phantom,from={1-2},to={3-4},"\square"]
			\ar[phantom,from={1-2},to={4-1},color=red,"\circlearrowleft"] \ar[phantom,from={1-4},to={4-3},color=red,"\circlearrowleft"]
			\ar[phantom,from={1-2},to={2-3},color=red,"\parallelogram"] \ar[phantom,from={3-2},to={4-3},color=red,"\parallelogram"]
			\ar[tail,from={2-1},to={2-3},crossing over] \ar[etail,from={2-3},to={4-3},crossing over]
		\end{tikzcd}\]
		\caption{The two forms of a cube in $\T$.}
		\label{classifyingcubes}
	\end{figure}
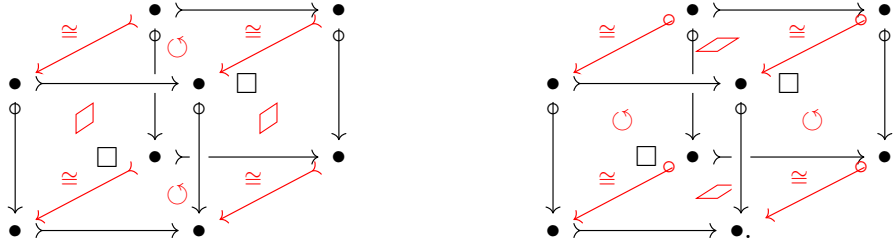
	
	Analogously, as the double category $\mathbb{VT}(\T_h)$ is assumed to be a classifying double category of the category with objects horizontal morphisms in $\T$ and morphisms hv-squares, cubes in $\T$ must equivalently have the form on the right in Figure~\ref{classifyingcubes}. While we have already shown that horizontal and vertical isomorphisms in $\T$ agree, comparing these two definitions of cubes in $\T$ shows that in $\D = \mathbb{HV}(\T)$ the horizontally invertible squares agree with commutative squares in $\mathcal V(\T)$ with a pair of oppposite edges isomorphisms.  Likewise the vertically invertible squares agree with the analogous commutative squares in $\mathcal H(\T)$, which shows that $\D$ has shared isomorphisms, completing the proof.
\end{proof}

We can now use this result and the constructions of Examples~\ref{doubletotriple} and \ref{triplebox} to characterize classifying triple categories in terms of lifting properties.

\begin{corollary}\label{classifyingtriplelift}
	A triple category $\T$ is isomorphic to the classifying triple category of a double category with shared isomorphisms if and only if it is local with respect to the set 
	\[ \mathfrak{HT}(S_{\cl}) \cup \mathfrak{VT}(S_{\cl}) \cup ([1] \boxtimes_v S_{\cl}) \cup ([1] \boxtimes_h S_{\cl}). \]
\end{corollary}

\begin{proof}
	We claim that $\mathfrak{HT} \colon \DCat \to \TCat$ is left adjoint to the functor $\mathbb{HT} \colon \TCat \to \DCat$, and similarly there are adjunctions 
	\[ \mathfrak{VT} \dashv \mathbb{VT}, \qquad ([1] \boxtimes_v -) \dashv \mathbb{HT}(-_v), \qquad  ([1] \boxtimes_h -) \dashv \mathbb{VT}(-_h). \] 
	Under this assumption, by Lemma~\ref{adjunctionlift} a triple category $\T$ is local with respect to $\mathfrak{HT}(S_{\cl})$, $\mathfrak{VT}(S_{\cl})$, $[1] \boxtimes_v S_{\cl}$ and $[1] \boxtimes_h S_{\cl}$ precisely when $\mathbb{HT}(\T)$, $\mathbb{VT}(\T)$, $\mathbb{HT}(\T_v)$, and $\mathbb{VT}(\T_h)$, respectively, are classifying diagrams of categories.  Then an application of Proposition~\ref{triplechar} completes the proof.

	It remains to establish the proposed adjunctions.  Given a double category $\D$ and triple category $\T$, a triple functor $\mathfrak{HT}(\D) \to \T$ consists of precisely the data of a double functor $\D \to \mathbb{HT}(\T)$ as $\mathfrak{HT}(\D)$ contains no non-identity vertical morphisms, hv- or vt-squares, or cubes; therefore $\mathfrak{HT}$ is left adjoint to $\mathbb{HT}$.  The proof that $\mathfrak{VT}$ is left adjoint to $\mathbb{VT}$ is entirely analogous, with the identities instead in the horizontal direction. 
	
	To see that $([1] \boxtimes_v -)$ is left adjoint to $\mathbb{HT}(-_v)$, observe that a triple functor $[1] \boxtimes_v \D \to \T$ picks out a vertical morphism in $\T$ for each object in $\D$, an hv-square in $\T$ for each horizontal morphism in $\D$, a vt-square in $\T$ for each vertical morphism in $\D$, and a cube in $\T$ for each square in $\D$, respecting identities and composition. This data is precisely that of a double functor $\D \to \mathbb{HT}(\T_v)$. The argument for the adjunction $([1] \boxtimes_h -) \dashv \mathbb{VT}(-_h)$ is analogous.
\end{proof}

\section{Recognition of CGW categories} \label{cgwsection}

\subsection{CGW categories}

In \cite{CZ-cgw}, CGW categories are defined as double categories with additional structure and properties that allow them to admit analogues of Quillen's $Q$-construction and Waldhausen's $S_\bullet$ construction that generalize those for an exact category.

Following our existing convention, we refer to the horizontal and vertical categories of a double category $\D$ as $\M$ and $\E$, respectively.  This notation differs from that of \cite{CZ-cgw}, where these categories are denoted by $\mathcal{M}$ and $\mathcal{E}$, suggestive of the admissible monomorphisms and epimorphisms in an exact category.

\begin{definition}
	For $\D$ a double category, let $\triM$ be the category whose objects are horizontal morphisms in $\D$ and whose morphisms are commutative squares in $\M$ of the form
	\[
	\begin{tikzcd}
	\bullet \rar[tail] \dar[tail,swap]{\cong} \ar[phantom]{dr}{\circlearrowleft} & \bullet \dar[tail] \\
	\bullet \rar[tail] & \bullet,
	\end{tikzcd} 
	\]
	where the horizontal morphism between source objects is an isomorphism. The category $\triE$ is defined analogously for vertical morphisms.
%
%	Let $\D$ be a double category with horizontal category $\M$ and vertical category $\E$.  Define $\Ar_\square \M$ to be the category whose objects are horizontal morphisms in $\D$ and whose morphisms are squares in $\D$, as depicted below on the left, and $\Ar_\triangle \M$ to be the category whose objects are also horizontal morphisms in $\D$ but whose morphisms are commutative squares in $\M$ consisting of an isomorphism between the domains, as depicted below right: \jbnoteil{probably should be a figure - in which case we could include the dual diagrams as well}
%	\[ \distsquare{\bullet}{\bullet}{\bullet}{\bullet}{}{}{}{}
%	\qquad\qquad\qquad\qquad
%	\begin{tikzcd}
%	\bullet \rar[tail] \dar[tail,swap]{\cong} \ar[phantom]{dr}{\circlearrowleft} & \bullet \dar[tail] \\
%	\bullet \rar[tail] & \bullet.
%	\end{tikzcd} \]
%	Dually, define $\Ar_\square \E$ to be the category of vertical morphisms and squares in $\D$, and $\Ar_\triangle \E$ to be the category of vertical morphisms and commutative squares in $\E$ consising of an isomorphism between the domains.
\end{definition} 

In \cite{CZ-cgw}, the category $\triM$ is denoted by $\mathsf{Ar}_\triangle\mathcal{M}$, and similarly the category $\squareM$ is denoted by $\mathsf{Ar}_\square\mathcal{M}$ .

We can now state the definition of a CGW category, which was first introduced by Campbell and Zakharevich \cite[Definition 2.5]{CZ-cgw}.

\begin{definition} \label{cz-cgw}
	A \emph{CGW category} is a double category $\D$ equipped with 
	\begin{itemize}
		\item an identity-on-objects isomorphism $\varphi \colon \iso(\M) \to \iso(\E)$,
		
		\item an equivalence of categories $k \colon \squareE \to \triM$ called the \emph{kernel}, and
		
		\item an equivalence of categories $c \colon \squareM \to \triE$ called the \emph{cokernel},
	\end{itemize}
	that satisfy the following axioms.
	\begin{itemize}
		\item[(Z)] There is an object $\varnothing$ in $\D$ that is initial in both $\M$ and $\E$.
		
		\item[(I)] For every horizontal isomorphism $f$, the pair $(f,\varphi(f))$ is a shared isomorphism.
%For any horizontal isomorphism $f \colon A \to B$ in $\D$, there are four distinguished squares of the forms
%		\[
%			\distsquare{A}{B}{B}{B}{f}{\varphi(f)}{\id_B}{\id_B} \qquad
%			\distsquare{A}{A}{A}{B}{\id_A}{\id_A}{\varphi(f)}{f} \qquad
%			\distsquare{A}{B}{A}{A}{f}{\id_A}{\varphi(f^{-1})}{\id_A} \qquad
%			\distsquare{A}{A}{B}{A.}{\id_A}{\varphi(f)}{\id_A}{f^{-1}} \qquad
%		\] 
		
		\item[(M)] Every morphism in both $\M$ and $\E$ is a monomorphism.
		
		\item[(K)] For every vertical morphism $g \colon A \erto B$, the codomain of $k(g)$ is $B$ and there is a unique distinguished square
		\[
			\distsquare{\varnothing}{A}{\ker(g)}{B.}{}{}{g}{k(g)} 
		\]
		The analogous property holds for a horizontal morphism $f \colon A \mrto B$, with distinguished square 
		\[
			\distsquare{\varnothing}{\coker(f)}{A}{B.}{}{}{c(f)}{f}
		\]

%\jbnoteil{not happy about the left/right business here, either, but a figure seems excessive}\bsnoteil{Unless we end up worried about page count I don't mind separating them like this to improve the flow.}
		
		%\item[(A)] For every pair of objects $A,B$ in $\D$ there exist squares:
		%\[
		%	\distsquare{\varnothing}{A}{B}{X}{}{}{}{} \qquad\qquad\qquad
		%	\distsquare{\varnothing}{B}{A}{R.}{}{}{}{}
		%\]
	\end{itemize}
\end{definition}

\begin{remark}
	Compared to the original definition in \cite[Definition 2.5]{CZ-cgw}, we have omitted axiom (A) which asserts the existence of certain distinguished squares for every pair of objects. The primary purpose of this axiom is to ensure that in the associated $K$-theory of a CGW category, the group $K_0$ is abelian.  As this consideration is not relevant to our purposes in this paper, we do not include this additional assumption here.
\end{remark}

\subsection{Recognizing CGW categories among double categories}

In \cite[Lemma 2.9]{CZ-cgw}, Campbell and Zakharevich prove that a CGW category satisfies a certain lifting condition that appears as axiom (K') in \cref{cgw} below.  In order to compare CGW categories to pointed stable double Segal spaces, it is more convenient to use a modified, but ultimately equivalent, definition of CGW category that uses this lifting condition in place of the explicitly chosen equivalences $k$ and $c$. This modified definition also uses the notion of a double category with shared isomorphisms in place of the isomorphism $\varphi$, which simplifies the construction of a double Segal space from a CGW category and allows all of the properties to be expressed using lifting properties.

\begin{definition}\label{cgw}
	A double category $\D$ is an \emph{abstract CGW category} if it satisfies the following axioms. 
	\begin{itemize} 
		\item[(I')] The double category $\D$ has shared isomorphisms.
		
		\item[(M)] All horizontal and vertical morphisms in $\D$ are monic.
		
		\item[(Z)] The double category $\D$ has an object $\varnothing$ that is initial with respect to both the horizontal and the vertical morphisms.
		
		\item[(K')] Any adjacent pair of morphisms $A \mrto[f] B \erto[k] D$ or $A \erto[h] C \mrto[g] D$ extends to a square 
		\[ \distsquare{A}{B}{C}{D}{f}{h}{k}{g} \]
		that is unique up to unique isomorphism which is the identity on the original three objects of the respective diagram. 
	\end{itemize}
\end{definition}

The idea of an abstract CGW category is that it admits the structure of a CGW category but is not equipped with a particular choice of that structure, namely, the equivalences $k$ and $c$ and isomorphism $\varphi$, as we prove in \cref{cgwcomparison}. First, we prove a preliminary technical lemma, for which we need the following definition.

%\begin{definition}
%	 A square in a double category is \emph{globular} if its vertical morphisms are identities or, dually, if its horizontal morphisms are identities.
%\end{definition}
%
%We sometimes refer to a \emph{vertically} globular square to emphasize that it is the vertical morphisms that are identities, and likewise to \emph{horizontally} globular squares.

\begin{lemma}\label{globularidentity}
	A double category satisfying axiom (K') is totally antiglobular.
\end{lemma}

\begin{proof}
	We prove that such a double category is antiglobular; the proof of the dual statement is entirely analogous. By the uniqueness property in axiom (K'), a vertically globular square is isomorphic to $\id_f$, as the two agree on the pair of morphisms $A \mrto[f] B = B$. Therefore there is an isomorphism $h \colon A \cong A$ such that $h \circ \id_A = \id_A$ and $g \circ h = f$, so, as $h$ must then be the identity, the two squares are equal. 
	%I'm slightly worried there's an implicit reliance on axiom (I') here that could be an issue in the following proposition, will think more about this
\end{proof}

\begin{proposition}\label{cgwcomparison}
	A double category admits the structure of a CGW category if and only if it is an abstract CGW category.
\end{proposition}

\begin{proof}
	Axioms (M) and (Z) are the same in both definitions, so we show that axioms (I') and (K') from \cref{cgw} are equivalent to the remaining parts of \cref{cz-cgw}, namely %\jbnoteil{maybe can streamline, now that we have the details of the defn above?}
	\begin{itemize}
		\item the existence of an identity-on-objects isomorphism $\varphi \colon \iso(\M) \to \iso(\E)$;
		
		\item the existence of equivalences of categories $k \colon \Ar_\square \E \to \Ar_\triangle \M$ and $c \colon \Ar_\square \M \to \Ar_\triangle \E$;
		
		\item axiom (I); and%, which asserts the existence of four distinguished squares built out of a horizontal isomorphism $f$, the vertical isomorphism $f$, and their inverses; and
		
		\item axiom (K).%, which asserts the existence of a distinguished square between a horizontal morphism, its cokernel, and two morphisms from $\varnothing$, and likewise for kernels of vertical morphisms.
	\end{itemize}
	We establish this equivalence through the following four implications, which suffice to prove the equivalence.
	\begin{enumerate}
		\item \label{kcimplyk'}
		The existence of the equivalences $k$ and $c$ implies axiom (K').

		\item \label{ik'implyi'}
		Axioms (I) and (K') imply axiom (I').
		
		\item \label{i'impliesphii}
		Axiom (I') implies the existence of the isomorphism $\varphi$ and axiom (I).

		\item \label{i'k'implykci}
		Axioms (I') and (K') imply the existence of the equivalences $k$ and $c$ and axiom (K).
		%.\bsnoteil{ultimately these two implications do suffice but if the logic seems too convoluted I'm happy to make it more explicit}
	\end{enumerate}

	Implication \eqref{kcimplyk'} is proven in \cite[Lemma 2.9]{CZ-cgw}, and \eqref{i'impliesphii} is immediate from \cref{sharedisos} defining shared isomorphisms.  Since we can use Lemma~\ref{globularidentity} to show that axiom (K') implies total antiglobularity, \eqref{ik'implyi'} follows from Proposition~\ref{sharedisochar}.

	To prove \eqref{i'k'implykci}, we assume axioms (K') and (I') and need to construct equivalences $k \colon \squareE \to \triM$ and $c \colon \squareM \to \triE$ and show that axiom (K) holds for them.  Here, we give the argument for $k$; the one for $c$ is completely analogous. 
	
	To define $k$ on objects, for each morphism $f \colon A \erto B$ we use axiom (K') to choose a square of the form
	\[
	\distsquare{\varnothing}{A}{C}{B.}{}{}{f}{\bar f}
	\]
	Setting $k(f) = \bar f$ ensures that the appropriate part of axiom (K) holds.  On morphisms, this procedure associates to a square
	\[ \distsquare{A}{A'}{B}{B',}{g}{f}{f'}{h} \]
	the two squares
	\[
	\begin{tikzcd}
		\varnothing \rar[tail] \dar[etail] & A \rar[tail]{g} \dar[etail,swap]{f} & A' \dar[etail]{f'} \\
		C \rar[tail,swap]{\bar f} & B \rar[tail,swap]{h} & B' 
		\dist{1-1}{2-2} \dist{1-2}{2-3} 
	\end{tikzcd}
	\qquad\textrm{and}\qquad
	\distsquare{\varnothing}{A'}{C'}{B'}{}{}{f'}{\bar f'}
	\] 
	that agree on the pair of morphisms $\varnothing \mrto A' \erto[f'] B'$.  Therefore there is a unique isomorphism $C \cong C'$ making the square 
	\[
	\begin{tikzcd}
		C \rar[equals]{\cong} \dar[tail,swap]{\bar f} & C' \dar[tail]{\bar f'} \\
		B \rar[tail,swap]{h} & B' 
		\comm{1-1}{2-2} 
	\end{tikzcd}
	\]
	commute in $\M$.  We define $k$ to take the original square in $\D$ to this commutative square in $\M$.  %The map $c$ is defined analogously, with the roles of $\M$ and $\E$ reversed. 
	
	Axiom (K') implies that the above morphism in $\triM$ is uniquely determined by $\bar f,\bar f'$, and $h$, and that the original morphism in $\squareE$ is determined up to unique isomorphism by $f$ and $h$; the fact that $\E$ consists of only monomorphisms makes it uniquely determined by $f,f'$, and $h$. Therefore, given $f$ and $f'$, a square of either type is uniquely determined by a map $h \colon B \mrto B'$ that belongs to a square of the corresponding type. It follows that $k$ is faithful.  

	Fullness follows from \cite[Lemma 2.9]{CZ-cgw}, which constructs from such a commutative square in $\M$ a distinguished square with $h$ as its bottom horizontal morphism and with $\E$-morphisms isomorphic to $f$ and $f'$ in $\E/B$ and $\E/B'$, respectively. Those isomorphisms give rise, by axiom (I'), to distinguished squares that compose into a distinguished square between $f,f'$, and $h$ as desired. 
	
	Essential surjectivity of $k$ follows from axiom (K'), as for any morphism $\bar f \colon C \mrto B$, there is a square of the form
	\[
	\distsquare{\varnothing}{A}{C}{B}{}{}{f}{\bar f}
	\]
	and therefore a commutative square
	\[
	\begin{tikzcd}
		C \rar[equals]{\cong} \dar[tail,swap]{\bar f} & C' \dar[tail]{k(f)} \\
		B \rar[tail,swap]{h} & B 
		\comm{1-1}{2-2} 
	\end{tikzcd}
	\]
	that is an isomorphism between $\bar f$ and $k(f)$ in $\triM$. Therefore $k$ is an equivalence of categories. 
\end{proof}

\begin{remark}
	The CGW category structure on an abstract CGW category is unique up to isomorphism. To see that the equivalence $k$ is unique up to unique natural isomorphism, the argument above for essential surjectivity also provides an isomorphism in $\triM$ between $k(f)$ and $k'(f)$, where $k'$ is defined analogously to $k$ for any other choice of distinguished squares used in the definition.  An argument similar to the construction of the action of $k$ on morphisms shows that these isomorphisms are natural, using the uniqueness statement of axiom (K'). %\bsnoteil{I can write this argument out if there's interest.}
\end{remark}

We can now characterize CGW categories using lifting properties that correspond to each of the axioms (I'), (M), (Z), and (K'). We already showed in Proposition~\ref{sharedisochar} that a double category satisfies axiom (I') precisely when it is $\{\truncateh,\truncatev,\hiso,\viso\}$-local, or, equivalently, when it is totally antiglobular with all horizontal and vertical isomorphisms shared in the sense of axiom (I). 

Axiom (K') appears to have the form of a lifting property, but of a different form from those we have discussed so far: every diagram in a double category of the form $\bullet \mrto \bullet \erto \bullet$ or $\bullet \erto \bullet \mrto \bullet$ extends to a square, uniquely only up to unique isomorphism. In other words, it is strongly local in the sense of Definition~\ref{stronglocality}, but in a way that requires us to be precise about the 2-category of double categories we are considering.  

\begin{definition}\label{cornerdefs}
	Let $\hv$ denote the double category $\bullet \mrto \bullet \erto \bullet$ and $\vh$ denote the double category $\bullet \erto \bullet \mrto \bullet$, and define the double functors
	\[ \toprightcorner \colon \hv \to [1] \boxtimes [1] \qquad \textrm{and} \qquad \bottomleftcorner \colon \vh \to [1] \boxtimes [1] \]
	by the canonical inclusions of $\hv$ and $\vh$ into the walking square $[1] \boxtimes [1]$:
	\[ \begin{tikzcd}
		\bullet \rar[tail] & \bullet \dar[etail] \\
 		& \bullet
	\end{tikzcd}
	\qquad \xrightarrow{\toprightcorner} \qquad
	\begin{tikzcd}
		\bullet \rar[tail] \dar[etail] & \bullet \dar[etail] \\
		\bullet \rar[tail] & \bullet
		\dist{1-1}{2-2}
	\end{tikzcd}
	\qquad \xleftarrow{\bottomleftcorner} \qquad
	\begin{tikzcd}
		\bullet \dar[etail] &  \\
		\bullet \rar[tail] & \bullet.
	\end{tikzcd} \]
\end{definition}

Now consider the double category $\tlift(\toprightcorner)$ consisting of two squares that agree on their top and right morphisms. The double category $\ilift(\toprightcorner)$ contains the squares and commutative diagrams
\[ \begin{tikzcd}
A \rar[tail]{f} \dar[etail,swap]{h} & B \dar[etail]{k} \\
C \rar[tail,swap]{g} & D
\dist{1-1}{2-2}
\end{tikzcd}
\qquad
\begin{tikzcd}
A \rar[tail]{f} \dar[etail,swap]{h'} & B \dar[etail]{k} \\
C' \rar[tail,swap]{g'} & D
\dist{1-1}{2-2}
\end{tikzcd}
\qquad
\begin{tikzcd}
A \rar[etailhor]{h'} \dar[equals] & C' \rar[tail]{g'} \dar[tail]{i} \dar[tail,swap]{\cong} & D \dar[equals] \\
A \rar[etailhor,swap]{h} & C \rar[tail,swap]{g} & D
\dist{1-1}{2-2}
\comm{1-2}{2-3}
\end{tikzcd}
\qquad
\begin{tikzcd}
A \rar[etailhor]{h'} \dar[equals] & C' \rar[tail]{g'} \dar[etail]{i'} \dar[etail,swap]{\cong} & D \dar[equals] \\
A \rar[etailhor,swap]{h} & C \rar[tail,swap]{g} & D,
\comm{1-1}{2-2}
\dist{1-2}{2-3}
\end{tikzcd}
\]
amounting to two squares which agree on their top and right morphisms and a natural shared isomorphism between them. We can analogously define $\tlift(\bottomleftcorner)$ and $\ilift(\bottomleftcorner)$, where instead it is the bottom and left morphisms of the squares that agree.

We then have the canonical inclusions
\[ \ulift(\toprightcorner) \colon \tlift(\toprightcorner) \to \ilift(\toprightcorner) 
\quad\textrm{and}\quad 
\ulift(\bottomleftcorner) \colon \tlift(\bottomleftcorner) \to \ilift(\bottomleftcorner) 
\]
of two distinct extensions of $\hv$ or $\vh$ to a square into two isomorphic extensions of $\hv$ or $\vh$ to a square.

\begin{lemma}\label{axiomklift}
	A double category $\D$ satisfies axiom (K') if and only if it is strongly $\{\toprightcorner, \bottomleftcorner\}$-local in $\DCati$.  %and strictly $\{\uniquetoprightcorner,\uniquebottomleftcorner\}$-local.
\end{lemma}

\begin{proof}
	By definition, a double category is $\toprightcorner$-local when every pair of morphisms $A \mrto[f] B \erto[k] D$ in $\D$ extends to a square 
	\[ \distsquare{A}{B}{C}{D}{f}{h}{k}{g} \]
	in $\D$, and it is strictly $\uniquetoprightcorner$-local when this extension is unique up to unique isomorphism.  Similarly, it is $\bottomleftcorner$-local when $A \erto[h] C \mrto[g] D$ to the same square, with analogous statement for being strictly $\uniquebottomleftcorner$-local.  Together, these conditions comprise axiom (K').
\end{proof}

The remaining axioms (M) and (Z) arise from the 1-categorical properties of all morphisms being monic and having an initial object. Recall from Example~\ref{monolift} that for a category $\C$, every morphism in $\C$ is monic if and only if $\C$ is local with respect to a functor $\mono$. In a double category $\D$, horizontal morphisms and vertical morphisms are all monic when $\D$ is $\{\H(\mono),\V(\mono)\}$-local.

%\bsnoteil{made edits from here to end of section for pointedness. My thinking is to confine the discussion of preaugmentation to this section (for the lifting property) and not mention it again.}

Axiom (Z) is more difficult to express; while the existence of a unique morphism from the initial object can be described via a strict lifting property, it requires already knowing what the initial object is.  We can resolve this issue by changing to a different category of double categories equipped with a choice of candidate objects; we revisit this perspective in a slightly different setting in Section~\ref{subsec.cgw_psdss}.  We first illustrate this approach at the level of ordinary categories.

\begin{definition}
	A category is \emph{left pointed} if it has an initial object.  A functor between left pointed categories is \emph{left pointed} if it preserves the initial object.
\end{definition}

In order to express left pointedness as a lifting property, we need the additional structure of a preaugmentation.

\begin{definition}
	A category is \emph{preaugmented} if is equipped with a subset $\mathcal{A}$ of objects called a \emph{preaugmentation}.  A functor between preaugmented categories is \emph{preaugmented} if it preserves objects in the preaugmentation. 
\end{definition}

We are most interested in the case when every object in the preaugmentation is initial, a property that we can express using a lifting condition in the category of preaugmented categories and preaugmented functors. In this case, the underlying category is left pointed.

In what follows, we typically we use $\ast$ to designate objects in the preaugmentation and $\bullet$ to denote arbitrary objects.   For example, we denote by $\left(* \quad \bullet \right)$ the category with two objects with one in the preaugmentation.

\begin{lemma}\label{initiallift}
	Every object in the preaugmentation of a preaugmented category $\C$ is initial if and only if $\C$ is strictly local with respect to the preaugmented functor
	\[ \uniquearrow \colon \left(* \quad \bullet \right) \to \left( * \to \bullet \right). \]
\end{lemma}

\begin{proof}
	Observe that $\C$ is strictly $\uniquearrow$-local precisely when for each object $\ast$ in the preaugmentation and each object $A$, there is a unique morphism $\ast \to A$.  Then this condition says precisely that $\ast$ is an initial object $\C$.
\end{proof}

Now we extend the definition of preaugmented and pointed categories to the context of double categories.

\begin{definition}
	A double category is \emph{preaugmented} if it is equipped with a subset $\mathcal{A}$ of objects called a \emph{preaugmentation}.  A double functor between preaugmented double categories is \emph{preaugmented} if it preserves objects in the preaugmentation.
\end{definition}

%\bsnoteil{I changed this to left pointed so that pointed means horizontally initial and vertically terminal in the next section}

\begin{definition}
	A double category is \emph{left pointed} if it has an object that is initial in both the horizontal and vertical categories.  A double functor between left pointed double categories is \emph{left pointed} if it preserves these objects.
\end{definition} 

In other words, double category is left pointed precisely when it satisfies axiom (Z).
The following proposition is now straightforward to verify.
%, using the fact that $\H(-)$ and $\V(-)$ extend to functors from preaugmented categories to preaugmented double categories.

\begin{proposition}
	 A preaugmented double category is left pointed if and only if it is strictly $\{\H(\uniquearrow),\V(\uniquearrow)\}$-local. 
\end{proposition}

We would like a convenient way to migrate lifting properties that we have defined for double categories into the setting of pointed double categories.  

\begin{proposition}
	Let $\DCat$ be the category of small double categories and $\DCat_*^\ell$ the category of small left pointed double categories.  The forgetful functor $U \colon \DCat_*^\ell \rightarrow \DCat$ has a left adjoint $(-)_*$ that freely adjoins to any double category a unique initial object.
\end{proposition}

It then follows from Lemma~\ref{adjunctionlift} that if $F$ is a double functor, then a left pointed double category $\D$ is $F_*$-local if and only its underlying double category $U\D$ is $F$-local.  In what follows, we often omit the $(-)_*$ notation from a double category or double functor and leave the extra initial object implicit.  %Observe also that $\H(-)$ and $\V(-)$ extend to functors from pointed categories to pointed double categories.  

The following proposition is now straightforward to verify.

%\jbnoteil{Maybe has no content if we assume pointedness?  Or drop the pointedness assumption to get it?}

\begin{proposition}\label{leftpointedlift}
	A preaugmented double category is left pointed if and only if it is strictly $\{\H(\uniquearrow),\V(\uniquearrow)\}$-local. 
\end{proposition}

\begin{corollary}\label{cgwliftingpropertiespointed}
	A left pointed double category admits the structure of a CGW category if and only if it is local with respect to 
	\[ \{\toprightcorner,\bottomleftcorner,\H(\mono),\V(\mono)\} \]
	and strictly local with respect to 
	\[ \{\truncateh, \truncatev, \hiso, \viso, \uniquetoprightcorner, \uniquebottomleftcorner\} \]
	in the category of pointed double categories.
\end{corollary}

Note that such a double category is also strictly local with respect to $\H(\mono),\V(\mono)$, as these double functors are epimorphisms.

\begin{remark}
It is possible to express both left pointedness and the remaining properties of an abstract CGW category as lifting properties in the category of preaugmented double categories. However, for simplicity, we will continue to present lifting properties in the style of Proposition~\ref{cgwliftingpropertiespointed} where (left) pointedness is assumed.
\end{remark}

\subsection{Reverse CGW categories}

As we mentioned in the introduction, there are two different directional conventions for the vertical morphisms in double categories to which we apply the $S_\bullet$-construction; in the next section, we want to compare CGW categories to a framework that uses the opposite convention.  To facilitate this comparison, we state a vertically dual definition of an abstract CGW category. But first, we define \emph{pointed} double categories, which are vertically dual to left pointed double categories and satisfy a weaker version of the property defined in \cite[Definition 3.5]{BOORSset}.

\begin{definition}\label{pointeddoublecat}
	A double category is \emph{pointed} if it has an object, called a \emph{zero object}, that is initial in the horizontal category and terminal in the vertical category.  A double functor between pointed double categories is \emph{pointed} if it preserves zero objects. We write $\DCat_*$ for the category of small pointed double categories and pointed double functors.
\end{definition}

\begin{definition}\label{reversecgw}
	A double category $\D$ is a \emph{reverse CGW category} if it satisfies the following axioms. 
	\begin{itemize} 
		\item[(I')] The double category $\D$ has shared isomorphisms.
		
		\item[(M')] All horizontal morphisms in $\D$ are monic, and all vertical morphisms in $\D$ are epic.
		
		\item[(Z')] The double category $\D$ is pointed.
		
		\item[(K'')] Any span of morphisms $C \elto[h] A \mrto[f] B$ or $C \mrto[g] D \elto[k] B$ 
		extends to a square 
		\[ \distsquare{A}{B}{C}{D}{f}{h}{k}{g} \]
		that is unique up to unique isomorphism that is the identity on the original three objects of the respective diagram. 
	\end{itemize}
\end{definition}

This definition is designed so that $\D$ is a reverse CGW category if and only if $\D^\vop$ is an abstract CGW category. As such, we can similarly characterize them using lifting properties, applying $(-)^\vop$ to those double functors that are not invariant under this process.

\begin{corollary}\label{reversecgwlift}
	A pointed double category is a reverse CGW category if and only if it is local with respect to 
	\[ \{\toprightcorner^\vop,\bottomleftcorner^\vop\} \]
	and strictly local with respect to 
	\[ \{\truncateh, \truncatev, \hiso, \viso, \H(\mono), \V(\mono^\opp), \]
	\[ \uniquetoprightcorner^\vop, \uniquebottomleftcorner^\vop\}. \]
\end{corollary}

\section{Classifying diagrams and lifting properties in simplicial objects} \label{simplicialsection}

In this section, we consider simplicial analogues of many of the constructions
above.  In particular, we give a characterization of which simplicial spaces are
classifying diagrams of categories in the sense of Rezk, a construction from
which we formulated the classifying double category in
Section~\ref{classifyingdcsection}.  We then generalize to classifying diagrams
of double categories, which are certain bisimplicial spaces that play the role
of the classifying triple categories above.  Using these constructions in the
simplicial cotext, we get an explicit comparison between reverse CGW categories
and pointed stable double Segal spaces.  This section contains no proofs, as all
proofs are directly analogous to those in Sections~\ref{classifyingdcsection}
and \ref{classifyingtcsection}.

\subsection{Segal spaces} \label{subsection.segalspaces}

In this section, after a brief review of simplicial sets and Segal sets, we turn to simplicial spaces and the definition of the classifying diagram of a category.

Recall from Example \ref{nervelocal} the inclusion 
\[ \iota_n \colon G[n] \rightarrow \Delta[n] \]
of the spine of the $n$-simplex.  For any simplical set $K$, this inclusion induces a \emph{Segal map}
\[ K_n = \Hom_{\SSet}(\Delta[n], K) \rightarrow \Hom_{\SSet}(G[n],K) = K_1 \times_{K_0} \cdots \times_{K_0} K_1. \]

\begin{definition}
	A \emph{Segal set} is a simplicial set $K$ such that the Segal maps are all isomorphisms for $n \geq 2$.
\end{definition}

Observe that Segal sets are precisely those simplicial sets that can be obtained as nerves of categories.  Given a Segal set $K$, one can define a category with objects $K_0$, morphisms $K_1$, and composition law defined by
\[ K_1 \times_{K_0} K_1 \xleftarrow[\cong]{(d_0,d_2)} K_2 \xrightarrow{d_1} K_1 \]
using the Segal map for $n=2$.

The following result is a formalization of the discussion in Example \ref{nervelocal}.

\begin{proposition} \label{segalsetlocal}
	A simplicial set is a Segal set if and only if it is (strictly) local with respect to the set $\{\iota_n \mid n \geq 2\}$.
\end{proposition}

Let us now consider \emph{bisimplicial sets}, given by functors $\Deltaop \times \Deltaop \rightarrow \Set$; we sometimes call them \emph{simplicial spaces} when we think of them equivalently as functors $\Deltaop \rightarrow \SSet$.  We denote the category of bisimplicial sets by $\ssSet$.   For the moment, we take the latter perspective and think of the following homotopical analogue of Segal spaces.

\begin{definition} 
	A \emph{Segal space} is a simplicial space $V \colon \Deltaop \rightarrow \SSet$ such that the Segal maps
	\[ V_n \rightarrow \underbrace{V_1 \times_{V_0} \cdots \times_{V_0} V_1}_n \]
	are weak equivalences for $n \geq 2$.
\end{definition}

Segal spaces can be regarded as structures like categories with simplicial sets of objects and morphisms but for which composition, associativity, and unitality need only be defined up to homotopy.  We would like to describe a Segal space via a lifting property, as in Proposition \ref{segalsetlocal}, for which we set up some notation for how to adapt the relevant maps to the context of simplicial spaces.

Given a simplicial set $K \colon \Deltaop \rightarrow \Set$, we can think of it as a bisimplicial set in two ways, by taking the simplicial structure of $K$ in one of the $\Deltaop$ variables and making the functor constant in the other variable.  Taking the point of view that a bisimplicial set is a simplicial space $\Deltaop \rightarrow \SSet$, we denote also by $K$ the constant simplicial set with $K$ in each degree.  Alternatively, we can take the simplicial space, sometimes denoted by $K^t$ and called the \emph{transpose}, given by $(K^t)_n = K_n$, where the set $K_n$ is regarded as a discrete simplicial space.  For another perspective on this construction, see Example \ref{boxandtranspose} below.

\begin{proposition} \label{segalspacelocal}
	A simplicial space is a Segal space if and only if it is local with respect to the set $\{\iota_n^t \colon G[n]^t \rightarrow \Delta[n]^t \mid n \geq 2\}$.
\end{proposition}

\subsection{Double nerves and double Segal sets}

Now, we shift our perspective to that of bisimplicial sets.  Similarly to the nerve construction for categories, we can define a double nerve functor from double categories to bisimplicial sets.  For this purpose, recall the double categories of the form $[m] \boxtimes [n]$ from Example \ref{boxtimesdefn}.

\begin{definition}
	The \emph{double nerve} of a double category $\D$ is the bisimplicial set $\nerve^\square(\D)$ given by
	\[ \nerve^\square(\D)_{m,n} = \Hom([m] \boxtimes [n], \D), \]
	where the morphisms are taken in the category of double categories and double functors.
\end{definition}

\begin{example}
	Just as the ordinary nerve of the ordinal category $[n]$ is the representable simplex $\Delta[n]$, the double nerve of $[m] \boxtimes [n]$ is precisely the representable bisimplicial set $\Delta[m,n]$, given by
	\[ \Delta[m,n]_{k, \ell} = \Hom([k] \boxtimes [\ell], [m] \boxtimes [n]). \]
\end{example}

\begin{example}\label{commutativenerve}
	In the special case of a double category of commutative squares in $\C$, there is a bijection
	\[ \Hom_{\DCat}([m] \boxtimes [n], \C \rtimes_{\id_\C} \C) \cong \Hom_{\Cat}([n],\C^{[m]}), \]
	as the morphisms in $\C^{[m]}$ are $m \times 1$ grids of commutative squares in $\C$.  We thus obtain an equivalent description of the set $\nerve^\square(\C \rtimes_{\id_\C} \C)_{m,n}$.
\end{example}

Observing the properties that such a bisimplicial set must satisfy, we make the following definition.

\begin{definition}\label{doublesegalset}
	A \emph{double Segal set} is a bisimplicial set $X \colon \Deltaop \times \Deltaop \rightarrow \Set$ such that for any $k,\ell \geq 0$ the simplicial sets $X_{k,*}$ and $X_{*,\ell}$ are Segal sets.
\end{definition}

Indeed, we have the following result, which can be proved analogously to the version for ordinary nerves and Segal sets.

\begin{proposition}
	A bisimplicial set is a double Segal set if and only if it is isomorphic to the double nerve of a double category.  Moreover, the double nerve induces a fully faithful functor $\nerve^\square \colon \DCat \rto ss\mathbf{Set}$. 
\end{proposition}

In particular, for a double Segal set $X$ the corresponding double category has objects $X_{0,0}$, horizonal morphisms $X_{1,0}$, vertical morphisms $X_{0,1}$, and squares $X_{1,1}$.

Finally, we introduce notation for building bisimplicial sets and maps between them out of simplicial sets and maps.

\begin{definition}
	Let $\C$ be a category closed under products and $\C'$ be any category.  For any two functors $F,G \colon \C' \rto \C$, their \emph{external product} $F \boxtimes G \colon \C' \times \C' \rto \C$ is given by $(F\boxtimes G)(A,B) = F(A) \times G(B)$. The external product extends to a functor $\mathbf{Fun}(\C',\C) \times \mathbf{Fun}(\C',\C) \to \mathbf{Fun}(\C' \times \C',\C)$, applying to natural transformations as the product of morphisms in $\C$ for each component.
\end{definition}

\begin{example}
	The external product of the representable functors $\Delta[m] \colon \Deltaop \to \Set$ and $\Delta[n] \colon \Deltaop \to \Set$ is the representable bisimplicial set $\Delta[m,n] \colon \Deltaop \times \Deltaop \to \Set$. We then have 
	\[ \nerve^\square([m] \boxtimes [n]) \cong \Delta[m,n] \cong \Delta[m] \boxtimes \Delta[n] \cong \nerve([m]) \boxtimes \nerve([n]), \]
	justifying the overlap in notation. 
\end{example}

\begin{example} \label{boxandtranspose}
	In bisimplicial set notation, the constant simplicial space $K$ is given by $K_{m,n}= K_n$ for every $m$, while the transpose simplicial set $K^t$ is given by $K^t_{m,n} = K_m$ for every $n$.  Thus $\Delta[m]^t = \Delta[m] \boxtimes \Delta[0]$.
\end{example}

\begin{example}
	The functor $- \boxtimes \Delta[n] \colon \SSet \to \ssSet$ has a right adjoint $\ssSet \to \SSet$ sending $X$ to the simplicial set $X_{*,n}$.  By definition of $\boxtimes$ we can establish the required isomorphisms on morphisms via the following sequence of isomorphisms:
	\[ \begin{aligned} 
		\Hom_{\ssSet}(K \boxtimes \Delta[n],X) & \cong \lim_m\Hom_{\SSet}(K_m \times \Delta[n],X_m) \\
		& \cong \lim_m\Hom_{\SSet}(K_m,X_m^{\Delta[n]}) \\
		& \cong \lim_m\Hom_{\Set}(K_m,X_{m,n}) \\
		& \cong \Hom_{\SSet}(K,X_{*,n}).
	\end{aligned} \]
\end{example}

\begin{proposition} \label{dsegalsetlocal} \label{separatesegalmaps}
	A bisimplicial set is a double Segal set if and only if it is (strictly) local with respect to the maps \
	\[ \left\{\iota_m \boxtimes \iota_n \mid m,n \geq 0 \right\}. \]
\end{proposition}

%To construct a set $S$ of morphisms in $\ssSet$ such that a bisimplicial set $X$ is double Segal if and only if it is $S$-local, note that by definition a simplicial set $K$ is Segal if and only if it is strictly $\{\iota_n \mid n \ge 0\}$-local as locality with respect to $\iota_n$ is precisely the same as the $n$-th Segal map being invertible.

%In particular, being local with respect to 
%\[ \iota_m \boxtimes \id_{\Delta[n]} \colon G[m] \boxtimes \Delta[n] \to \Delta[m] \boxtimes \Delta[n] \]
%ensures that the induced map
%\[ \Hom_{\SSet}(\Delta[m], X_{*,n}) \cong \Hom_{\ssSet}(\Delta[m] \boxtimes \Delta[n], X) \rightarrow \Hom_{\ssSet}(G[m] \boxtimes \Delta[n],X) \cong \Hom_{\SSet}(G[m],X_{*,n}) \]
%is an isomorphism, and analogously being local with respect to $\id_{\Delta[m]} \boxtimes \iota_n$ ensures that the $n$-th Segal map for $X_{m,*}$ is an isomorphism.

\begin{remark} \label{doublesegalmaps} 
	In Proposition~\ref{dsegalsetlocal}, the lifting condition asserts that in a double Segal set $X$, 	
\[ X_{m,n} \to \Hom_{\ssSet}(G[m] \boxtimes G[n],X), \]
meaning that $X_{m,n}$ agrees with the set of $m \times n$ grids of squares in $X_{1,1}$.

Alternatively we could define the condition in terms of each row and column of $X$ separately, using the set
\[ \{\iota_m \boxtimes \id_n \colon G[m] \boxtimes \Delta[n] \rightarrow \Delta[m,n] \mid m \geq 2, n \geq 0\} \cup \{ \id_m \boxtimes \iota_n \colon \Delta[m] \boxtimes G[n] \rightarrow \Delta[m,n] \mid m \geq 0, n \geq 2\}. \]
	These lifting conditions assert that the simplicial sets $X_{m,*}$ and $X_{*,n}$ are 1-Segal for all $m,n \ge 0$, more closely matching Definition~\ref{doublesegalset}.
\end{remark}

%\bsnoteil{I can look for a good reference for this, or just prove it}

%Moreover, this proposition implies that we can recognize properties of nerves of
%double categories using locality by applying $\nerve^\square$ to the functors
%constructed in the previous section.  Thus, for example, a double Segal set is
%the nerve of a flat double category if and only if it is
%$\nerve^\square\flatten$-local. 

\subsection{Classifying diagrams}

The following definition was first made by Rezk in \cite[\S 3.5]{rezk2001model} as a motivating family of examples of Segal spaces.

\begin{definition} \label{classifyingdef}
	For a category $\C$, the \emph{classifying diagram} is the simplicial space $N\C$ defined by 
	\[ (N\C)_m = \nerve(\iso \C^{[m]}), \] 
	the nerve of the groupoid of length $m$ sequences of composable morphisms in $\C$ and natural isomorphisms between them. 
\end{definition}

%Classifying diagrams also provide good examples of Segal spaces that are complete.  
Observe that for the classifying diagram $N\mathcal C$ of a category $\mathcal C$, the Segal condition holds strictly, in that the Segal maps are actually isomorphisms of simplicial sets; indeed when we regard it as bisimplicial set it is a double Segal set.  In fact, we have the following result.

\begin{proposition}
	Let $\C$ be a small category.  Its classifying diagram $N\C$ is isomorphic to the double nerve of the classifying double category $\nerve^\square(\C \rtimes_I \iso \C)$.
\end{proposition}

\begin{proof} 
	Similarly to Example \ref{commutativenerve}, we have
	\[ \nerve(\iso \C^{[m]})_n = \Hom_{\Cat}([n],\iso \C^{[m]}) \cong \Hom_{\DCat}([m] \boxtimes [n],\C \rtimes_I \iso \C) = \nerve^\square(\C \rtimes_I \iso \C)_{m,n}. \] 
\end{proof}

%Indeed, our motivation for characterizing the classifying double category is so that we can study classifying diagrams of categories via the composite functor 
%\[ \Cat \xrightarrow{- \rtimes_I \iso -} \DCat \xrightarrow{\nerve^\square} \ssSet, \] 
%which is isomorphic to $N$. \jbnoteil{rephrase to reflect the fact that we already discuss this above} 

Combining Theorem \ref{classifyingliftingcondition} and \cref{compositelift} then gives us a description of when a bisimplicial set is a classifying diagram.

\begin{corollary} 
	Let $S$ be the set of maps from Proposition \ref{dsegalsetlocal}.  A bisimplicial set is isomorphic to the classifying diagram of a category if and only if it is strictly $(S \cup \nerve^\square S_{\cl})$-local.
\end{corollary}

\subsection{Double Segal spaces and classifying diagrams of double categories}

In order to define classifying diagrams for double categories, we need to move up a simplicial level and consider \emph{bisimplicial spaces}, or functors $Y \colon \Deltaop \times \Deltaop \rightarrow \SSet$.  We are primarily interested in following the homotopical analogue of double Segal sets. 

\begin{definition}\label{def.doublesegal}
	A \emph{double Segal space} is a bisimplicial space $Y$ such that the Segal maps
	\[ Y_{n,k} \rightarrow \underbrace{Y_{1,k} \times_{Y_{0,k}} \cdots \times_{Y_{0,k}} Y_{1,k}}_n \]
	and 
	\[ Y_{\ell,m} \rightarrow \underbrace{Y_{\ell,1} \times_{Y_{\ell,0}} \cdots \times_{Y_{\ell,0}} Y_{\ell,1}}_m \]
	are weak equivalences of simplicial sets, where $k, \ell \geq 0$ and $m,n \geq 2$.
	
	A double Segal space is \emph{strict} if the Segal maps are isomorphisms.
\end{definition}

Observe that a strict double Segal space is a bisimplicial space that is strictly local with respect to the set 
\[ \left\{\iota_k \boxtimes \iota_\ell \boxtimes \id_{\Delta[m]} \mid l,\ell,m \ge 0\right\}. \]

We now define the classifying diagram of a double category; as before, we need to require that the double category have shared isomorphisms.

\begin{definition}\label{doubleclassifyingdiagram}
	Given a double category $\D$ with shared isomorphisms, its \emph{classifying diagram} $N^\square\D$ is the bisimplicial space in which $(N^\square\D)_{m,n}$ is the nerve of the groupoid of $(m \times n)$-grids or squares in $\D$ and natural isomorphisms between them. 
\end{definition}

Regarding bisimplicial spaces as trisimplicial sets, we get an alternative description of the classifying diagram as the triple nerve of the classifying triple category of $\D$.  We begin with the definition of the triple nerve.

\begin{definition}
	The \emph{triple nerve} of a triple category $\C$ is the trisimplicial set $\nerve^{\mycube}(\T)$ given by
	\[ \nerve^{\mycube}(\T)_{\ell,m,n} = \Hom([\ell] \boxtimes [m] \boxtimes [n], \T), \]
	where the morphisms are taken in the category of triple categories and triple functors.
\end{definition}

Just as double nerves of double categories are double Segal sets, we expect triple nerves of triple categories to be triple Segal sets, a notion we now define.

\begin{definition}
	A trisimplicial set $Y$ is a \emph{triple Segal set} if for any $k,\ell,m \ge 0$, $Y_{k,\ell,*}$, $Y_{k,*,m}$, and $Y_{*,\ell,m}$ are Segal sets.
\end{definition}

Indeed, as expected from the double category setting, these two definitions are equivalent and have a description in terms of local objects.

\begin{proposition} \label{tripleSegallocal}
	The following are equivalent for a trisimplical set $Y$:
	\begin{enumerate}
		\item $Y$ is a triple Segal set; 
		
		\item $Y$ is the triple nerve of a triple category; and
		
		\item $Y$ is local with respect to the set
		\[ \left\{\iota_\ell \boxtimes \iota_m \boxtimes \iota_n \mid \ell,m,n \geq 0\right\}. \]
	\end{enumerate}
\end{proposition}

Using these definitions, we give two characterizations of the bisimplicial spaces that arise as classifying diagrams of double categories with shared isomorphisms, corresponding to the two characterizations of classifying triple categories. These characterizations rely on the observation from the definitions that the classifying diagram of $\D$ is isomorphic to the triple nerve of the classifying triple category of $\D$.

%\jbnoteil{Do we still need this, or can we deduce it directly from the categorical version now?}
%\bsnoteil{As of now the proof is essentially just saying it follows from the categorical version, but could probably be reduced to just that one line if you think it doesn't need as much fuss.}
%\jbnoteil{Yes, I think we can cut this down to a couple of sentences.}
%\bsnoteil{done}

\begin{proposition} \label{ndchar}
	A bisimplicial space $Y$ is isomorphic to $N^\square\D$ for a double category with shared isomorphisms $\D$ if and only if
	\begin{enumerate}
		\item \label{ndchar1}
		$Y$ is strictly double Segal; and
		
		\item \label{ndchar2}
		$Y_{*,0}$, $Y_{0,*}$, $Y_{*,1}$, and $Y_{1,*}$ are each isomorphic to the classifying diagram of a category. 
		
		%\item The map $Y_{1,1,1} \cong \Hom(\Delta[1,1,1],Y) \to \Hom(\sqcap_{3,0},Y)$ is an isomorphism, where $\sqcap_{3,0}$ is the union of all of the faces of $\Delta[1,1,1]$ except for $(\id,\id,d^0)$.
	\end{enumerate}
\end{proposition}

%If we are looking for double categories which are moreover thin, as in having at most one square in each boundary, then the last condition can be replaced with the map $Y_{1,1} \to Y^{\partial \Delta[1,1]}$ being a trivial fibration of simplicial sets.

\begin{proof}
	It follows from the definition of $N^\square\D$ that the two conditions are satisfied, so it remains to show that a bisimplicial space $Y$ satisfying these two conditions is isomorphic to $N^\square \D$ for some $\D$.
	
	Condition \eqref{ndchar2} ensures that the space $Y_{k,\ell,*}$ is the nerve of a groupoid when either $k$ or $\ell$ is 0 or 1.  Since by \eqref{ndchar1} $Y$ is a double Segal space, it follows that $Y_{k,\ell,*}$ is the nerve of a groupoid for every $k$.  Combined with the fact that $Y_{*,*,m}$ is a double Segal set for all $m \ge 0$ by \eqref{ndchar1}, we can conclude that that $Y$ is a triple Segal set and therefore isomorphic to the nerve of a triple category $\T$. The result then follows from Proposition~\ref{triplechar}.
\end{proof}

We encourage the reader to refer back to the diagram \eqref{classifyingtc} for intuition.

We conclude this section with a lifting property description of which bisimplicial spaces arise as classifying diagrams. The following is an immediate consequence of Lemma~\ref{compositelift}, Corollary~\ref{classifyingtriplelift}, and Proposition \ref{tripleSegallocal}. 

\begin{corollary}\label{doubleclassifyinglift} 
	A trisimplicial set is isomorphic to the classifying diagram of a double category with shared isomorphisms if and only if it is strictly local with respect to 
	\[ \left\{\iota_\ell \boxtimes \iota_m \boxtimes \iota_n \mid \ell,m,n \geq 0\right\} \cup \nerve^{\mycube} \left(\H(S_{\cl}) \cup \V(S_{\cl}) \cup ([1] \boxtimes_v S_{\cl}) \cup ([1] \boxtimes_h S_{\cl})\right). \] 
\end{corollary}

\section{Reverse CGW categories as pointed stable double Segal spaces} \label{comparisonsection}

In this section, we develop the theory of pointed stable double Segal spaces and
give a precise relationship with reverse CGW categories.  Although reverse CGW
categories give rise to stable double Segal spaces that are pointed, for some
aspects of the theory it is more useful to work in the more general setting of
augmentations.

\subsection{Pointed stable double Segal spaces} \label{sec.augmentedstable}

%\bsnoteil{made minor edits in this section for pointedness}

The purpose of pointed stable double Segal spaces is to identify those bisimplicial spaces that have sufficient structure so that we can apply a form of Waldhausen's $S_\bullet$-construction to them.  In addition to being double Segal spaces, we need a way to construct squares analogously to taking pullbacks of cospans or pushouts of spans.  To this end, we introduce the notion of stability for a double Segal space.  Heuristically, one should think of the squares in a stable double Segal space as the analogues of bicartesian commutative squares in an ordinary category.  

Following axiom (K'') for double categories, we want a square
\[ \distsquare{A}{B}{C}{D}{}{}{}{} \]
to be uniquely determined up to isomorphism by the mixed span $C \elto A \mrto B$ and by the mixed cospan $C \mrto D \elto B$. %This notion is stricter than that of axiom (K') from Definition~\ref{reversecgw}, but has the elegant property that the set of squares in the double category is in bijection with the sets of mixed spans and mixed cospans. 
%\jbnoteil{Talk about this idea above somewhere?}

Translating this condition to the double Segal space context, if we think of points in $Y_{1,1}$ as squares, we want the restriction maps $Y_{1,1} \to Y_{0,1} \times_{Y_{0,0}} Y_{1,0}$ and $Y_{1,1} \to Y_{1,0} \times_{Y_{0,0}} Y_{0,1}$, specifying the desired span and cospan, respectively, to be weak equivalences.

\begin{definition} {\cite{BOORSobj}} \label{stable} 
	A double Segal space $Y$ is \emph{stable} if the induced maps 
	\[ Y_{1,1} \to Y_{0,1} \times_{Y_{0,0}} Y_{1,0} \qquad\qquad Y_{1,1} \to Y_{1,0} \times_{Y_{0,0}} Y_{0,1} \]
	are weak equivalences. 
\end{definition}

Another feature we need for an $S_\bullet$-construction is a zero object. % or a more general collection, called an augmentation, of objects that behave locally like a zero object, similarly to Definition~\ref{augmenteddouble} for double categories. 
For bisimplicial spaces, we capture this extra data using a preaugmentation space via the following indexing category with an additional object.

\begin{definition}
	Let $\Sigma$ denote the category obtained by adjoining a new terminal object, denoted by $[-1]$, to the category $\Delta \times \Delta$:
	\[ \begin{tikzcd}
		{[-1]}  \\[-10]
		{} &[-20] {[0,0]} \ar{lu}{} 
		\rar[shift left=2] \rar[shift right=2] 
		\dar[shift left=2] \dar[shift right=2] & 
		{[1,0]} \lar \rar \rar[shift left=3] \rar[shift right=3] \dar[shift left=2] \dar[shift right=2] & 
		{[2,0]} \lar[shift left=1.5] \lar[shift right=1.5] \dar[shift left=2] \dar[shift right=2] &[-25]
		\cdots \\ 
		{} & 
		{[0,1]} \rar[shift left=2] \rar[shift right=2] \uar \dar \dar[shift left=3] \dar[shift right=3] & 
		{[1,1]} \lar \rar \rar[shift left=3] \rar[shift right=3] \uar \dar \dar[shift left=3] \dar[shift right=3] & 
		{[2,1]} \lar[shift left=1.5] \lar[shift right=1.5] \uar \dar \dar[shift left=3] \dar[shift right=3] & 
		\cdots \\
		{} & 
		{[0,2]} \rar[shift left=2] \rar[shift right=2] \uar[shift left=1.5] \uar[shift right=1.5] & 
		{[1,2]} \lar \rar \rar[shift left=3] \rar[shift right=3] \uar[shift left=1.5] \uar[shift right=1.5] & 
		{[2,2]} \lar[shift left=1.5] \lar[shift right=1.5] \uar[shift left=1.5] \uar[shift right=1.5] &
		\cdots \\[-20]
		{} & \vdots & \vdots & \vdots & \ddots. 
	\end{tikzcd} \]
\end{definition}

\begin{definition}
	A \emph{preaugmented bisimplicial space} is a functor $Y \colon \Sigma^{\op} \rightarrow \SSet$.
\end{definition}

The name is justified by the fact that the inclusion $\Delta \times \Delta \hookrightarrow \Sigma$ induces a forgetful functor on diagrams of simplicial sets; we think of the image of a preaugmented bisimplicial space under this functor as its \emph{underlying bisimplicial space}.  Alternatively, a bisimplicial space is preaugmented if it is equipped with an additional space $Y_{-1}$ and a map $Y_{-1} \to Y_{0,0}$.

If we want the space $Y_{-1}$ to behave as if it contains the zero objects in the sense of Definition~\ref{pointeddoublecat}, we need maps $Y_{1,0} \times_{Y_{0,0}} Y_{-1} \to Y_{0,0}$ and $Y_{0,1} \times_{Y_{0,0}} Y_{-1} \to Y_{0,0}$ capturing the desired local universal property up to homotopy.  The former map can be given as the composite
\[ Y_{1,0} \times_{Y_{0,0}} Y_{-1} \to Y_{1,0} \to Y_{0,0} \]
of the projection map followed by the appropriate face map, while the latter can be described as a similar composite
\[ Y_{0,1} \times_{Y_{0,0}} Y_{-1} \to Y_{0,1} \to Y_{0,0}. \]

\begin{definition}{\cite{BOORSobj}} \label{pointed} 
	A preaugmented bisimplicial space is \emph{pointed} if the maps
	\[ Y_{1,0} \times_{Y_{0,0}} Y_{-1} \to Y_{0,0} \qquad \text{and} \qquad Y_{0,1} \times_{Y_{0,0}} Y_{-1} \to Y_{0,0} \]
	described above are weak equivalences and $Y_{-1}$ is contractible.
\end{definition}

%\jbnoteil{Changing 0 to $\ast$ throughout to match notation elsewhere in the paper.}

Depicting elements of $Y_{1,0}$ as arrows $\mrto$, those of $Y_{0,1}$ as arrows $\erto$, and those of $Y_{-1}$ as $\ast$, a preaugmented bisimplicial space $Y$ is pointed when the spaces of diagrams $\ast \mrto A$ and $A \erto \ast$ are weakly equivalent to $Y_{0,0}$ via the map sending either such arrow to $A$.  In other words, a pointed bisimplicial space has a contractible space of objects that has the up-to-homotopy universal property for initial objects with respect to the horizontal morphisms, and similarly has the universal property for terminal objects with respect to vertical morphisms.% that each have the global universal property analogous to the one for the distinguished object in a pointed double category. 

%
%\begin{definition}
%	A (pre)augmented bisimplicial space is a \emph{(pre)augmented double Segal space} if its underlying bisimplicial space is a double Segal space, and it is \emph{stable} if its underlying bisimplicial space is.
%\end{definition}
%
%\begin{definition}
%	Given a double category $\D$ with shared isomorphisms and preaugmentation $\mathcal{A}$, we extend the classifying diagram $N^\square\D$ to a preaugmented bisimplicial space with preaugmentation given by the nerve of the full subgroupoid of objects in $\D$ consisting of those objects in $\mathcal{A}$.
%\end{definition}
%
%\jbnoteil{Experimental pointed version}

In what follows, we focus on the case of pointed bisimplicial spaces.

\begin{definition}
	A pointed bisimplicial space is a \emph{pointed double Segal space} if its underlying bisimplicial space is a double Segal space, and it is \emph{stable} if its underlying bisimplicial space is.
\end{definition}

We can now define a simplicial analogue of the classifying diagram for double categories.

\begin{definition}
	Given a pointed double category $\D$ with shared isomorphisms, we extend the classifying diagram $N^\square\D$ to a preaugmented bisimplicial space via the nerve of the full subgroupoid of zero objects in $\D$.
\end{definition}

We now verify that the classifying diagram of a pointed double category is a pointed bisimplicial space. In the next section we show that the classifying diagram is also stable when $\D$ satisfies axiom (K'') from Definiton~\ref{reversecgw}.

%\jbnoteil{We don't actually talk about its being stable - probably should comment on this at least!}
%\bsnoteil{Added above.}

%\jbnoteil{End of experimental pointed changes}
%\bsnoteil{moved this up from next section and simplified it}

\begin{proposition}\label{pointedproof}
	Given a pointed double category $\D$ with shared isomorphisms, the preaugmented classifying diagram $N^\square\D$ is pointed. Conversely, if $N^\square\D$ is pointed then $\D$ is a pointed double category.
%	A double category $\D$ with shared isomorphisms satisfies axiom (Z') for reverse CGW categories if and only if $\D$ admits a preaugmentation $\mathcal{A}$ such that the preaugmented bisimplicial set $N^\square\D$ is pointed.
\end{proposition}

\begin{proof}
	%If $\D$ satisfies axiom (Z') then we can take the zero objects as the preaugmentation; we fix a preaugmentation and show that axiom (Z') is satisfied with the entire preaugmentation as zero objects if and only if $N^\square\D$ is pointed.

	%Using a similar argument as above for stabliity, 
	$N^\square\D$ is pointed if and only if the restriction functors 
	\[ \Iso\left(\ast \mrto \bullet, \D\right) \to \Iso\left(\bullet ,\D\right) \qquad \Iso\left(\ast \elto \bullet ,\D\right) \to \Iso\left(\bullet ,\D\right) \]
	are equivalences of groupoids and the subgroupoid of zero objects is contractible. %Here $\ast$ denotes an object in the preaugmentation. 
We consider only the first equivalence; the argument for the second is analogous.

	%Assuming axiom (K''), 
	This functor is surjective on objects, as the space of zero objects is nonempty and every object has a horizontal morphism from each zero object. It is fully faithful as for any shared isomorphism $f \colon A \cong B$ between objects in $\D$ and horizontal morphisms $* \mrto A$, $*' \mrto B$ with $*,*'$ zero objects, there is a unique natural shared isomorphism 
	\[ \begin{tikzcd}
		* \rar[tail] \dar[swap]{\cong} & A \dar{f}[swap]{\cong} \\
		*' \rar[tail] & B,
	\end{tikzcd} \] 
	where the square commutes by initiality. There is a unique morphism between each pair of zero objects, so they form a contractible groupoid.

	Assuming $N^\square\D$ is pointed, by essential surjectivity of this functor for any object $A$ in $\D$ there is an isomorphic object $B \cong A$ in $\D$ and a horizontal morphism $\ast \to B$ with $\ast \in N^\square\D_{-1}$. The composite $\ast \to B \cong A$ then provides a horizontal morphism to $A$. By fully faithfulness any two such morphisms $\ast \to A$ must be related by an automorphism $\ast \to \ast$, but by contractibility of the preaugmentation this automorphism can only be the identity. Hence objects $N^\square\D_{-1}$ are horizontally initial. By an analogous argument these objects are also vertically terminal, so $\D$ is a pointed double category.
\end{proof}

\subsection{Stability via lifting properties in 2-categories}\label{lifting2cat}

The main goal of this subsection is to prove that for a double category $\D$ with shared isomorphisms, satisfying axiom (K'') for reverse CGW categories corresponds to stability of its classifying diagram. Rather than proving this claim directly, we develop the general theory of lifting properties in a 2-category to highlight the key features of the proof; this theory may be of independent interest. %For comparison, we give a direct proof that under the same conditions $\D$ satisfies axiom (Z') for reverse CGW categories if and only if it admits a preaugmentation making its classifying diagram pointed.

These results address the following question for a morphism $A \to B$ in a 2-category: when is the property that ``any map $A \to C$ extends to a map $B \to C$ uniquely up to unique isomorphism'' equivalent to the property that ``the restriction functor $\Homm(B,C) \to \Homm(A,C)$ is an equivalence''?  In a 2-category, we use the notation $\Homm(-,-)$ to indicate the 1-category of morphisms, as opposed to the morphism set $\Hom(-,-)$.  Both of these properties can be expressed in terms of certain kinds of lifting conditions, and to our knowledge they have not previously been defined in the literature.

We begin with a preliminary definition, in which we denote by $\Cat$ the 2-category of categories, functors, and natural transformations, and by $\cat$ its underlying 1-category of categories and functors.    

%\jbnoteil{I really don't like conflating the 2-category with the 1-category here - can this be avoided??}
%\bsnoteil{Fixed, I don't think there's any issue with this. In keeping with our font choices I think it makes sense to use mathbf for 2-categories and mathcal for 1-categories; I've tried to implement this in what follows.}

%\jbnoteil{Also, the ob and mor notation here is awful - is there any way to avoid it?}
%
%\bsnoteil{I'm definitely not attached to it, though I'm not exactly sure how to avoid it. We could maybe write something like $\mathsf{1Cell}_\CC(C,C')$ amd $\mathsf{2Cell}_\CC(C,C')$ instead since I don't think we use ob and mor much other than for $\Homm$s}
%
%\jbnoteil{Okay, I don't think that would improve matters...}

\begin{definition}
	A 2-category $\CC$ has a \emph{tensored interval} if there is a functor $[1] \otimes - \colon \C \to \C$ on the underlying 1-category $\C$ of $\CC$ together with a natural isomorphism
	\[ \Hom_{\C}([1] \otimes C,C') \cong \mor\left(\Homm_\CC(C,C')\right), \]
	along with natural morphisms $[1] \otimes C \to C$, $C \rightrightarrows [1] \otimes C$ representing the natural transformations $\id \colon \ob\left(\Homm_\CC(C,C')\right) \to \mor\left(\Homm_\CC(C,C')\right)$ and $s, t \colon \mor\left(\Homm_\CC(C,C')\right) \rightrightarrows \ob\left(\Homm_\CC(C,C')\right)$ respectively.
\end{definition}

In particular thre is a canonical morphism $C \sqcup C \to [1] \otimes C$, where $[1] \otimes C$ is similar to a cylinder object for $C$.

\begin{remark}
	Having a tensored interval is a weaker version of the notion of being \emph{tensored in $\cat$} \cite[\S4]{KellyAdjunction}, where there is a functor $\otimes \colon \cat \times \C \to \C$ and a natural isomorphism 
	\[ \Hom_{\C}(\A \otimes C,C') \cong \Hom_{\Cat} \left(\A, \Homm_\CC(C, C') \right) \]
	for any small category $\A$.
\end{remark}

\begin{example}
	The 2-category $\Cat$ has a tensored interval, given by the cartesian product $[1] \times C$.  A functor $[1] \times C \to C'$ is precisely a natural transformation of functors $C \to C'$.
\end{example}

\begin{example}
	Given two double functors $F,G \colon \D \to \D'$, a natural shared isomorphism from $F$ to $G$ consists of a natural horizontal isomorphism and a natural vertical isomorphism from $F$ to $G$, such that for each object $A$ of $\D$ the pair $(FA \mrto[\cong] GA, FA \erto[\cong] GA)$ forms a shared isomorphism. Double categories, double functors, and natural shared isomorphisms form a 2-category that we denote by $\DCati$.

%	\jbnoteil{The previous sentence is awful, and I don't really understand what it is trying to say.}
%	\bsnoteil{I've rewritten it, let me know what you think.}
 
	$\DCati$ has a tensored interval where $[1] \otimes \D$ is given by $(\E \rtimes_{\id_\E} \E) \times \D$, where $\E \rtimes_{\id_\E} \E$ is the walking shared isomorphism (see (\ref{walkingsharediso})). A functor $(\E \rtimes_{\id_\E} \E) \times \D \to \D'$ is precisely the data of a natural shared isomorphism between two functors $\D \to \D'$.
\end{example}

\begin{definition}\label{strongcopy}
	Given a morphism $i \colon A \to B$ in $\CC$, let 
	\[ \tlift(i) \coloneqq \colim(B \xleftarrow{i} A \xrightarrow{i} B) \qquad\textrm{and}\qquad \ilift(i) \coloneqq \colim\left(A \leftarrow [1] \otimes A \to [1] \otimes B \right). \]
\end{definition}

The pushout of $B \leftarrow A \to B$ agrees with the pushout of $A \leftarrow A \sqcup A \to B \sqcup B$, which maps to the pushout of $A \leftarrow [1] \otimes A \to [1] \otimes B$ by the canonical morphism $B \sqcup B \to [1] \otimes B$, and analogously for $A$. Therefore, we obtain an induced morphism $\ulift(i) \colon \tlift(i) \to \ilift(i)$ for any $i$.

\begin{lemma}
	Let $i \colon A \to B$ be a morphism in a 2-category $\CC$ with a tensored interval. If an object $X$ is strictly local with respect to $\ulift(i)$, then lifts of morphisms $A \to X$ along $i$ are unique up to a unique invertible 2-cell in $\Homm_\CC(B,X)$ that restricts to the identity on $A \to X$.
\end{lemma}

\begin{proof}
	Assume that $X$ is strictly $\ulift(i)$-local.  A morphism $\tlift(i) \to X$ amounts to a morphism $A \to X$ and two lifts $B \to X$ along $i$.  On the other hand, a morphism $\ilift(i) \to X$ amounts to a 2-cell $\alpha$ between two morphisms $B \to X$ (corresponding to a map $[1] \otimes B \to X$) and a map $f \colon A \to X$ such that the horizontal composite of $\alpha$ with $i$  (corresponding to the restriction $[1] \otimes A \to [1] \otimes B \to X$) is the identity 2-cell on $f$.  In other words, we obtain a morphism $f \colon A \to X$ and a 2-cell between two lifts $B \to X$ that restricts to the identity on $f$, as desired. To see that this 2-cell is invertible, note that the same argument produces 2-cells in both directions whose composites on either side must agree with the identities by uniqueness.
\end{proof}

We now give the definition of a strongly local object in a 2-category, for which lifts are required to be unique up to unique isomorphism.

\begin{definition}\label{stronglocality}
	Let $i \colon A \to B$ be a morphism in a 2-category $\CC$ with a tensored interval.  An object $X$ of $\CC$ is \emph{strongly local with respect to $i$} if it is $i$-local and strictly $\ulift(i)$-local.
\end{definition}

\begin{remark}\label{bicatlifting}
	In \cite[Definition 3.1]{modelbicats}, a notion of lifting property in a bicategory is introduced, in which lifts must be unique up to isomorphism but not up to \emph{unique} isomorphism as in Definition~\ref{stronglocality}. In the spirit of Lemma~\ref{2localequiv}, an object $X$ satisfies this weaker lifting property with respect to a morphism $i \colon A \to B$ when the induced functor $i^\ast \colon \Homm_\CC(B,X) \to \Homm_\CC(A,X)$ is essentially surjective.
\end{remark}

\begin{example}
	In the 2-category $\Cat$ with tensored interval, consider the inclusion $i \colon \varnothing \rightarrow \ast$ from the empty category to the 1-object category.  A category is local with respect to $i$ when it is nonempty, whereas it is local with respect to $\ulift(i)$ when it has a unique object up to unique isomorphism. In other words, a category is strongly $i$-local if and only if it is a contractible groupoid, equivalent to the terminal category.
\end{example}

\begin{example}
	In $\DCati$, consider the inclusion $\toprightcorner^\vop$ from Definition~\ref{cornerdefs} and Corollary~\ref{reversecgwlift}, which we picture as
	\[ \begin{tikzcd}
		& \bullet \dar[etail] \\
 		\bullet \rar[tail] & \bullet
	\end{tikzcd}
	\qquad \xrightarrow{\toprightcorner^\vop} \qquad
	\begin{tikzcd}
		\bullet \rar[tail] \dar[etail] & \bullet \dar[etail] \\
		\bullet \rar[tail] & \bullet.
		\dist{1-1}{2-2}
	\end{tikzcd} \]

	The double category $\tlift(\toprightcorner^\vop)$ consists of two squares that agree on their bottom and right morphisms. The double category $\ilift(\toprightcorner^\vop)$ contains the squares and commutative diagrams
	\[  \begin{tikzcd}
		C \rar[tail]{g} \dar[etail,swap]{h} & D \dar[etail]{k} \\
		A \rar[tail,swap]{f} & B
		\dist{1-1}{2-2}
	\end{tikzcd}
	\qquad
	\begin{tikzcd}
		C' \rar[tail]{g'} \dar[etail,swap]{h'} & D \dar[etail]{k} \\
		A \rar[tail,swap]{f} & B
		\dist{1-1}{2-2}
	\end{tikzcd}
	\qquad
	\begin{tikzcd}
		A \dar[equals] & \lar[etailhor,swap]{h'} C' \rar[tail]{g'} \dar[tail]{i} \dar[tail,swap]{\cong} & D \dar[equals] \\
		A & \lar[etailhor]{h} C \rar[tail,swap]{g} & D
		\dist{1-1}{2-2}
	\comm{1-2}{2-3}
	\end{tikzcd}
	\qquad
	\begin{tikzcd}
		A \dar[equals] & \lar[etailhor,swap]{h'} C' \rar[tail]{g'} \dar[etail]{i'} \dar[etail,swap]{\cong} & D \dar[equals] \\
		A & \lar[etailhor]{h} C \rar[tail,swap]{g} & D,
		\comm{1-1}{2-2}
		\dist{1-2}{2-3}
	\end{tikzcd} 
	\]
	amounting to two squares which agree on their bottom and right morphisms and a natural shared isomorphism between them.

	A double category is local with respect to $\toprightcorner^\vop$ when every ``mixed cospan'' $\bullet \mrto \bullet \elto \bullet$ belongs to a square, and strongly local when that square is unique up to unique shared isomorphism on its source. This is precisely axiom (K'') from Definition~\ref{reversecgw}.
\end{example}

We have defined strong locality to encode the existence of lifts that are unique up to unique isomorphism, but we can also ask for lifts of 2-cells between morphisms.

\begin{definition}\label{2copy}
	Given a morphism $i \colon A \to B$ in $\CC$, let 
	\[ \arrlift(i) \coloneqq \colim\left([1] \otimes A \leftarrow A \sqcup A \to B \sqcup B\right). \]
	Define $\arrfill(i) \colon \arrlift(i) \to [1] \otimes B$ for any $i$ to be the morphism induced by the maps $[1] \otimes A \to [1] \otimes B$ and $B \sqcup B \to [1] \otimes B$.
\end{definition}

\begin{lemma}\label{2localequiv}
	Let $i \colon A \to B$ be a morphism in a 2-category $\CC$ with a tensored interval.  If an object $X$ is $i$-local and strictly $\arrfill(i)$-local, then the restriction functor $i^\ast \colon \Homm_\CC(B,X) \to \Homm_\CC(A,X)$ is an equivalence of categories.
\end{lemma}

\begin{proof}
	Since $X$ is $i$-local, the functor $i^\ast$ is surjective on objects. A morphism $\arrlift(i) \to X$ amounts to two maps $B \to X$ along with a 2-cell between their restrictions along $i$ to maps $A \to X$, or in other words, a pair of objects in $\Homm_\CC(B,X)$ and a morphism between their images in $\Homm_\CC(A,X)$. A lift to a morphism $[1] \otimes B \to X$ corresponds to an extension of that 2-cell to the original maps $B \to X$, hence a morphism in $\Homm_\CC(B,X)$.  It follows that $X$ is strictly $\arrfill(i)$-local if and only if $i^\ast$ is fully faithful.
\end{proof}

\begin{remark}
	Note that the converse is not true: if $X$ is $i$-local, then $i^\ast$ is surjective on objects.  For a lifting condition equivalent to essential surjectivity of $i^\ast$, we would need $X$ to have lifts along $i$ only up to invertible 2-cells from $A$ to $X$.  However, for our purposes in this paper we do not need this weaker lifting condition, so we do not pursue this idea further here.
\end{remark}

\begin{definition}\label{2locality}
	An object $X$ in a 2-category $\CC$ with a tensored interval is \emph{(strongly) 2-local with respect to $i$} if it is $i$-local and (strictly) $\arrfill(i)$-local.
\end{definition}

%This 2-locality condition encodes the property of having lifts with respect to $i$ of both morphisms $A \to X$ and 2-cells between them; strong 2-locality further demands that the lifts of 2-cells are unique with respect to their source and target lifts $B \to X$.

\begin{example}
	For the inclusion $\toprightcorner^\vop$ of the walking mixed cospan into the walking square in $\DCati$, $\arrlift(\toprightcorner^\vop)$ consists of the squares
	
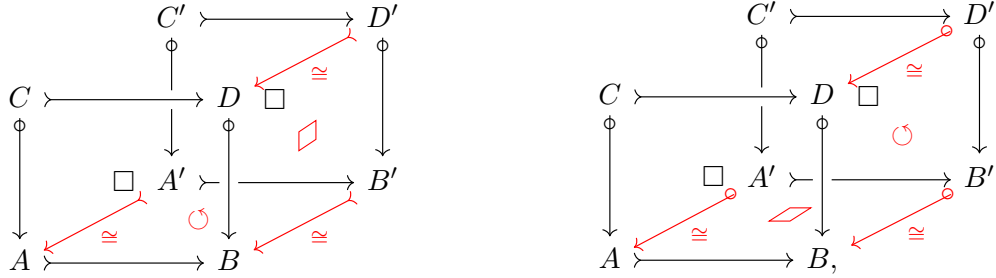
\[  \begin{tikzcd}[row sep=scriptsize,column sep=scriptsize]
		&[18] C' \ar[tail]{rr} \ar[etail]{dd} &[-18] &[18] D' 	\ar[tail,color=red]{dl}{\cong} \ar[etail]{dd} \\
		C \ar[etail]{dd} & & D \\
		& A' \ar[tail,color=red]{dl}{\cong} \ar[tail]{rr} & & B' 	\ar[tail,color=red]{dl}{\cong} \\
		A \ar[tail]{rr} & & B 
		\ar[phantom,from={2-1},to={4-3},"\square"] \ar[phantom,from={1-2},to={3-4},"\square"]
		\ar[phantom,from={1-4},to={4-3},color=red,"\parallelogramvert"]
		\ar[phantom,from={3-2},to={4-3},color=red,"\circlearrowleft"]
		\ar[tail,from={2-1},to={2-3},crossing over] 	\ar[etail,from={2-3},to={4-3},crossing over]
	\end{tikzcd} 
	\qquad\qquad\qquad
	\begin{tikzcd}[row sep=scriptsize,column sep=scriptsize]
		&[18] C' \ar[tail]{rr} \ar[etail]{dd} &[-18] &[18] D' 	\ar[etailhor,color=red]{dl}{\cong} \ar[etail]{dd} \\
		C \ar[etail]{dd} & & D \\
		& A' \ar[etailhor,color=red]{dl}{\cong} \ar[tail]{rr} & & B' 	\ar[etailhor,color=red]{dl}{\cong} \\
		A \ar[tail]{rr} & & B, 
		\ar[phantom,from={2-1},to={4-3},"\square"] \ar[phantom,from={1-2},to={3-4},"\square"]
	 	\ar[phantom,from={1-4},to={4-3},color=red,"\circlearrowleft"]
	 	\ar[phantom,from={3-2},to={4-3},color=red,"\parallelogram"]
		\ar[tail,from={2-1},to={2-3},crossing over] \ar[etail,from={2-3},to={4-3},crossing over]
	\end{tikzcd}  \]
	where the corresponding red diagonal morphisms in these two partial cubes are shared isomorphisms. The map $\arrfill(\toprightcorner^\vop)$ includes this double category into $[1] \otimes ([1] \boxtimes [1])$, which can be pictured as the two cubes
	
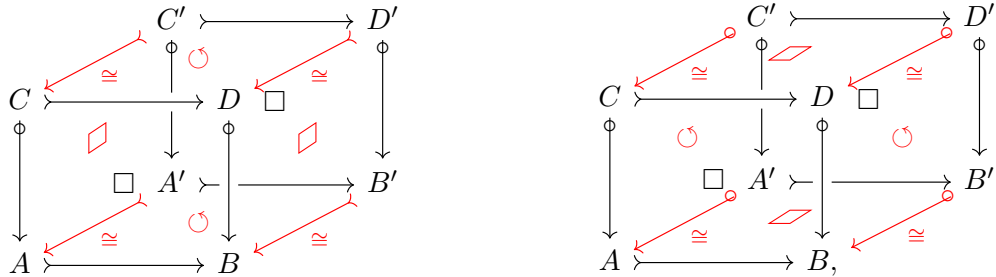
\[ \begin{tikzcd}[row sep=scriptsize,column sep=scriptsize]
		&[18] C' \ar[tail,color=red]{dl}{\cong} \ar[tail]{rr} \ar[etail]{dd} &[-18] &[18] D' \ar[tail,color=red]{dl}{\cong} \ar[etail]{dd} \\
		C \ar[etail]{dd} & & D \\
		& A' \ar[tail,color=red]{dl}{\cong} \ar[tail]{rr} & & B' \ar[tail,color=red]{dl}{\cong} \\
		A \ar[tail]{rr} & & B
		\ar[phantom,from={2-1},to={4-3},"\square"] \ar[phantom,from={1-2},to={3-4},"\square"]
		\ar[phantom,from={1-2},to={4-1},color=red,"\parallelogramvert"] \ar[phantom,from={1-4},to={4-3},color=red,"\parallelogramvert"]
		\ar[phantom,from={1-2},to={2-3},color=red,"\circlearrowleft"] 	\ar[phantom,from={3-2},to={4-3},color=red,"\circlearrowleft"]
		\ar[tail,from={2-1},to={2-3},crossing over] \ar[etail,from={2-3},to={4-3},crossing over]
	\end{tikzcd} 
	\qquad\qquad\qquad
	\begin{tikzcd}[row sep=scriptsize,column sep=scriptsize]
		&[18] C' \ar[etailhor,color=red]{dl}{\cong} \ar[tail]{rr} \ar[etail]{dd} &[-18] &[18] D' \ar[etailhor,color=red]{dl}{\cong} \ar[etail]{dd} \\
		C \ar[etail]{dd} & & D \\
		& A' \ar[etailhor,color=red]{dl}{\cong} \ar[tail]{rr} & & B' \ar[etailhor,color=red]{dl}{\cong} \\
		A \ar[tail]{rr} & & B, 
		\ar[phantom,from={2-1},to={4-3},"\square"] \ar[phantom,from={1-2},to={3-4},"\square"]
		\ar[phantom,from={1-2},to={4-1},color=red,"\circlearrowleft"] \ar[phantom,from={1-4},to={4-3},color=red,"\circlearrowleft"]
		\ar[phantom,from={1-2},to={2-3},color=red,"\parallelogram"] \ar[phantom,from={3-2},to={4-3},color=red,"\parallelogram"]
		\ar[tail,from={2-1},to={2-3},crossing over] \ar[etail,from={2-3},to={4-3},crossing over]
	\end{tikzcd} \]
	where again the corresponding red diagonal morphisms are shared isomorphisms.

	To be strongly 2-local with respect to $\toprightcorner^\vop$ therefore means that every mixed cospan belongs to a square and that every natural shared isomorphism between the cospans of two squares extends uniquely to a natural shared isomorphism of the entire squares. These two properties are satisfied precisely when the restriction functor from the category of squares in $\D$ to the category of mixed cospans in $\D$, both with natural shared isomorphisms as morphisms, is both fully faithful and essentially surjective.
\end{example}

We now provide more detail on the relationship between being strongly local and strongly 2-local.

\begin{proposition}\label{2tostrong}
	If an object $X$ in a 2-category with tensored interval is strongly 2-local with respect to a morphism $i \colon A \to B$, then it is also strongly local with respect to $i$.
\end{proposition}

This result could be proven by a straightforward argument about equivalences of categories, since the preimage of an identity morphism under an equivalence of categories is a contractible groupoid.  Here we instead give an argument based on lifting properties.

\begin{proof}
	Observe that the pushout of the map
	\[ \arrfill(i) \colon \colim\left([1] \otimes A \leftarrow A \sqcup A \to B \sqcup B\right) \to [1] \otimes B \]
	along the map 
	\[ \colim\left([1] \otimes A \leftarrow A \sqcup A \to B \sqcup B\right) \to \colim\left(A \leftarrow A \sqcup A \to B \sqcup B\right) \]
	is precisely $\ilift(i) = \colim\left(A \leftarrow [1] \otimes A \to [1] \otimes B \right)$.  Thus $\ulift(i)$ is a pushout of $\arrfill(i)$, and the result follows from an application of Lemma~\ref{pushoutlift}.
\end{proof}

\begin{example}
	For the inclusion $\toprightcorner^\vop$ in $\DCati$, the double functor $[1] \otimes ([1] \boxtimes [1]) \to \ilift(\toprightcorner^\vop)$ can be pictures as the map
	
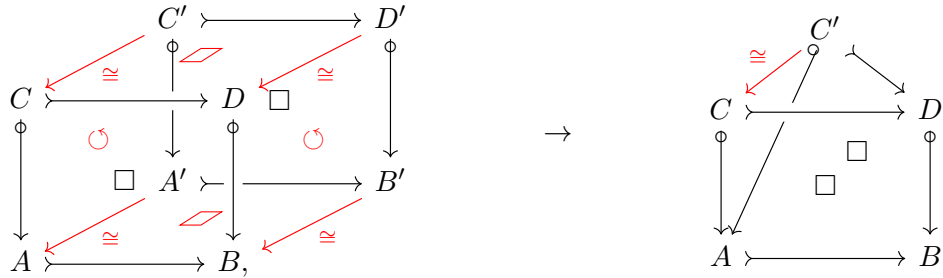
\[ \begin{tikzcd}[row sep=scriptsize,column sep=scriptsize]
		&[18] C' \ar[color=red]{dl}{\cong} \ar[tail]{rr} \ar[etail]{dd} &[-18] &[18] D' \ar[color=red]{dl}{\cong} \ar[etail]{dd} \\
		C \ar[etail]{dd} & & D \\
		& A' \ar[color=red]{dl}{\cong} \ar[tail]{rr} & & B' \ar[color=red]{dl}{\cong} \\
		A \ar[tail]{rr} & & B, 
		\ar[phantom,from={2-1},to={4-3},"\square"] \ar[phantom,from={1-2},to={3-4},"\square"]
		\ar[phantom,from={1-2},to={4-1},color=red,"\circlearrowleft"] \ar[phantom,from={1-4},to={4-3},color=red,"\circlearrowleft"]
		\ar[phantom,from={1-2},to={2-3},color=red,"\parallelogram"] \ar[phantom,from={3-2},to={4-3},color=red,"\parallelogram"]
		\ar[tail,from={2-1},to={2-3},crossing over] 	\ar[etail,from={2-3},to={4-3},crossing over]
	\end{tikzcd} 
	\qquad\qquad\to\qquad\qquad
	\begin{tikzcd}[row sep=scriptsize,column sep=scriptsize]
		& C' \ar[color=red]{dl}[swap]{\cong} \ar[tail]{dr} \ar[etailhor]{dddl} \\
		C \ar[etail]{dd} & & D \\[10] \\
		A \ar[tail]{rr} & & B
		\ar[phantom,from={2-1},to={4-3},"\square"] \ar[phantom,from={1-2},to={4-3},shift right=3,"\square"]
		\ar[tail,from={2-1},to={2-3},crossing over] \ar[etail,from={2-3},to={4-3},crossing over]
	\end{tikzcd}  \]
sending the objects $A',B'$, and $C'$ in $[1] \otimes ([1] \boxtimes [1])$ to $A,B$, and $C$ in $\ilift(\toprightcorner^\vop)$, respectively, where again the red diagonal arrows denote shared isomorphisms.
\end{example}

While the converse to Proposition~\ref{2tostrong} does not always hold, it does in examples in which the pushout square from $\arrfill(i)$ to $\ulift(i)$ has a left inverse. The following lemma is a direct consequence of Lemma~\ref{retractlift}.

\begin{lemma}\label{strongto2}
	Let $X$ be an object in a 2-category with tensored interval that is strongly local with respect to a morphism $i \colon A \to B$.  If the dashed morphisms in the diagram 
	\[ \begin{tikzcd}
      	\arrlift(i) \rar \dar[swap]{\arrfill(i)} \ar[phantom]{dr}[pos=1]{\ulcorner} & \tlift(i) \rar[densely dashed] \dar[swap]{\ulift(i)} & \arrlift(i) \dar[swap]{\arrfill(i)} \\ 
      	{[1]} \otimes B \rar & \ilift(i) \rar[densely dashed] & {[1]} \otimes B
	\end{tikzcd}\] 
	exist and present $\arrfill(i)$ as a retract of $\ulift(i)$, then $X$ is strongly 2-local with respect to $i$.
\end{lemma}

We can use this result to show that axiom (K'') from Definition~\ref{reversecgw} for a double category with shared isomorphisms corresponds to stability of its classifying diagram.

\begin{proposition}\label{stableproof}
	A double category $\D$ with shared isomorphisms satisfies axiom (K'') for reverse CGW categories if and only if is classifying diagram $N^\square\D$ is a stable double Segal space.
\end{proposition}

For the purpose of this proposition and the next, let $\Iso(\mathbb{C},\D)$ denote the groupoid of double functors $\mathbb{C} \to \D$ and natural shared isomorphisms. 

\begin{proof}
	As its components are nerves of groupoids, $N^\square\D$ is stable if and only if the restriction functors 
	\[ \Iso\left(\distsquare{\bullet}{\bullet}{\bullet}{\bullet}{}{}{}{},\D\right) \to 
	\Iso\left(\begin{tikzcd}  \bullet \rar[tail] & \bullet \dar[etail] \\ & \bullet  \end{tikzcd},\D\right) \qquad 
	\Iso\left(\distsquare{\bullet}{\bullet}{\bullet}{\bullet}{}{}{}{},\D\right) \to 
	\Iso\left(\begin{tikzcd}  \bullet \dar[etail] \\ \bullet \rar[tail] & \bullet \end{tikzcd} ,\D\right) \]
	are equivalences of groupoids. The double category $\D$ satisfies this condition if and only if it is strongly 2-local with respect to the double functors $\toprightcorner^\vop$ and $\bottomleftcorner^\vop$. As axiom (K'') is equivalent to $\D$ being strongly local with respect to the same two double functors, by Proposition~\ref{2tostrong} and Lemma~\ref{strongto2} it suffices to construct the dashed sections in the diagram
	\[ \begin{tikzcd}
      	\arrlift(\toprightcorner^\vop) \rar \dar[swap]{\arrfill(\toprightcorner^\vop)} \ar[phantom]{dr}[pos=1]{\ulcorner} & \tlift(\toprightcorner^\vop) \rar[densely dashed] \dar[swap]{\ulift(\toprightcorner^\vop)} & \arrlift(\toprightcorner^\vop) \dar[swap]{\arrfill(\toprightcorner^\vop)} \\ 
      	{[1]} \otimes ([1] \boxtimes [1]) \rar & \ilift(\toprightcorner^\vop) \rar[densely dashed] & {[1]} \otimes ([1] \boxtimes [1]);
	\end{tikzcd}\] 
	the argument for the corresponding maps using $\bottomleftcorner$ is entirely analogous.

	The section $\ilift(\toprightcorner^\vop) \to [1] \otimes ([1] \boxtimes [1])$ can be pictured as a map
	
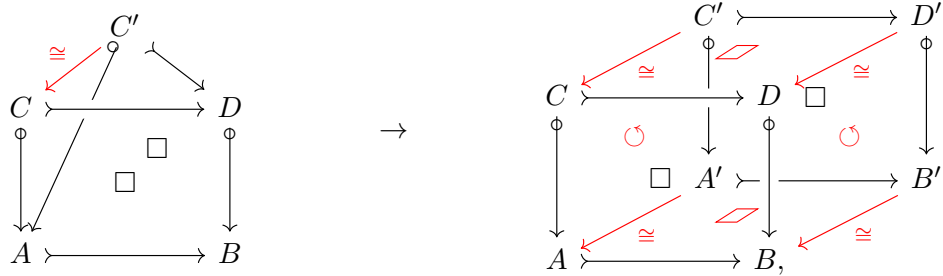
\[ \begin{tikzcd}[row sep=scriptsize,column sep=scriptsize]
		& C' \ar[color=red]{dl}[swap]{\cong} \ar[tail]{dr} \ar[etailhor]{dddl} \\
		C \ar[etail]{dd} & & D \\[10] \\
		A \ar[tail]{rr} & & B
		\ar[phantom,from={2-1},to={4-3},"\square"] \ar[phantom,from={1-2},to={4-3},shift right=3,"\square"]
		\ar[tail,from={2-1},to={2-3},crossing over] \ar[etail,from={2-3},to={4-3},crossing over]
	\end{tikzcd} 
	\qquad\qquad\to\qquad\qquad
	\begin{tikzcd}[row sep=scriptsize,column sep=scriptsize]
		&[18] C' \ar[color=red]{dl}{\cong} \ar[tail]{rr} \ar[etail]{dd} &[-18] &[18] D' \ar[color=red]{dl}{\cong} \ar[etail]{dd} \\
		C \ar[etail]{dd} & & D \\
		& A' \ar[color=red]{dl}{\cong} \ar[tail]{rr} & & B' \ar[color=red]{dl}{\cong} \\
		A \ar[tail]{rr} & & B, 
		\ar[phantom,from={2-1},to={4-3},"\square"] \ar[phantom,from={1-2},to={3-4},"\square"]
		\ar[phantom,from={1-2},to={4-1},color=red,"\circlearrowleft"] \ar[phantom,from={1-4},to={4-3},color=red,"\circlearrowleft"]
		\ar[phantom,from={1-2},to={2-3},color=red,"\parallelogram"] \ar[phantom,from={3-2},to={4-3},color=red,"\parallelogram"]
		\ar[tail,from={2-1},to={2-3},crossing over] \ar[etail,from={2-3},to={4-3},crossing over]
	\end{tikzcd} \]
	where the diagonal arrows denote shared isomorphisms. This double functor is defined by sending each object among $A,B,C,C',D$ in $\ilift(\toprightcorner^\vop)$ to itself in $[1] \otimes ([1] \boxtimes [1])$, where the ``back'' square containing $C'$ is sent to the composite
	\[ \begin{tikzcd}
		C' \rar[tail] \dar[etail] & D' \rar{\cong} \dar[etail] & D \dar[etail] \\
		A' \rar[tail] \dar[swap]{\cong} & B' \rar{\cong} \dar{\cong} & B \dar[equals] \\
		A \rar[tail] & B \rar[equals] & B. 
		\dist{1-1}{2-2} \dist{1-2}{2-3} \dist{2-1}{3-2} \dist{2-2}{3-3}
	\end{tikzcd} \]
	This double functor is a section, as the map $[1] \otimes ([1] \boxtimes [1]) \to \ilift(\toprightcorner^\vop)$ sends each of the objects $A,B,C,C'$, and $D$ in $[1] \otimes ([1] \boxtimes [1])$ to itself in $\ilift(\toprightcorner^\vop)$. It also restricts to a section $\tlift(\toprightcorner^\vop) \to \arrlift(\toprightcorner^\vop)$ by discarding the shared isomorphism $C' \cong C$ and all squares containing it from both the domain and codomain. We therefore have both of the sections required to apply Lemma~\ref{strongto2}, completing the proof.
\end{proof}

The key idea in the previous proof is that the three squares in $[1] \otimes ([1] \boxtimes [1])$ containing the object $B'$ can be ``composed'' into a single square. This fact relies on the specific composition available in a double category and the variety of squares between shared isomorphisms, rather than any general result about lifts unique up to isomorphism.  

%\jbnoteil{???}\bsnoteil{Is this clearer? I'm trying to convey that this result shouldn't be seen as obvious}

%We now prove an analogous result for axiom (Z') from Definition~\ref{reversecgw} and pointedness of the classifying diagram. In this simpler case we give a proof without relying on our results for 2-categorical lifting properties, to illustrate the different possible approaches. 
%
%\bsnoteil{This can also be done using a lifting properties argument, and I'm not sure if the value of another example of that outweighs that of demonstrating this different type of proof. Here the direct proof isn't as long as it was for stability, while the relevant lifting property with respect to the functor $(0 \;\;\; \bullet) \hookleftarrow (0 \to \bullet)$ is a bit less intuitive I think. I don't think the length would be much different either way.}

\subsection{Reverse CGW categories as pointed stable double Segal spaces} \label{subsec.cgw_psdss}

%\bsnoteil{made minor edits in this section for pointedness}

We now summarize our results on lifting properties and use them to show that classifying diagrams of reverse CGW categories are pointed stable double Segal spaces. We also identify precisely when a pointed stable double Segal space arises in this manner. %We assume throughout that a reverse CGW category is preaugmented \jbnoteil{adapt to pointed} by its set of zero objects satisfying the conditions of axiom (Z').

Using Corollaries~\ref{doubleclassifyinglift} and~\ref{reversecgwlift}, along with Lemma~\ref{compositelift}, we can characterize the image of reverse CGW categories in pointed bisimplicial spaces under the classifying diagram functor $N^\square$.  Recall from Corollary~\ref{classifyingliftingcondition} that a double category is isomorphic to the classifying diagram of a category if and only if it is strictly local with respect to
\[ S_{\cl} = \{\flatten,\truncateh,\rotater,\rotatel,\V(\invert),\doublerotate,\hiso\}, \]
meaning it is simple, antiglobular, vertically flexible, and the vertical category includes into the horizontal category as its maximal subgroupoid.

\begin{corollary}\label{classifyingcgwlift}  
	A pointed bisimplicial space $Y$ is isomorphic to the classifying diagram of a reverse CGW category if and only if it is strictly local with respect to
\[
\left\{\iota_\ell \boxtimes \iota_m \boxtimes \iota_n \mid \ell,m,n \geq 0\right\} \cup N^\square\{\H(\mono),\V(\mono^\opp)\} %\cup N^\square\{\H(\uniquearrow),\V(\uniquearrow^\opp)\}
\]
\[
\cup \nerve^{\mycube}\bigg{(}\left(\H(S_{\cl})\right) \cup \left(\V(S_{\cl})\right) \cup \left([1] \boxtimes_v S_{\cl}\right) \cup \left([1] \boxtimes_h S_{\cl})\right)\bigg{)}
\]
and also strongly $N^\square\{\toprightcorner^\vop, \bottomleftcorner^\vop\}$-local.
\end{corollary}

Table~\ref{classifyingcgwlifttable} specifies the specific purpose and reference for each component of this set.

\begin{table}
\centering
\begin{tabular}{|l|p{7.5cm}|p{3cm}|}
\hline Morphisms & significance & reference \\
\hline $\left\{\iota_\ell \boxtimes \iota_m \boxtimes \iota_n \mid \ell,m,n \geq 0\right\}$ & $Y$ is triple Segal, as the triple nerve of a triple category & Proposition~\ref{tripleSegallocal} \\
\hline $N^\square\{\H(\mono),\V(\mono^\opp)\}$ & every horizontal morphism of $\D$ is monic and every vertical morphism is epic & Axiom (M'), \newline Definition~\ref{reversecgw} \\
%\hline $N^\square\{\H(\uniquearrow),\V(\uniquearrow^\opp)\}$ & every object in $\D$ admits a unique horizontal morphism to (and vertical morphism from) every object in the preaugmentation & Axiom (Z'), \newline Definition~\ref{reversecgw} \\
\hline $N^\square\{\toprightcorner^\vop, \bottomleftcorner^\vop\}$ & every mixed span or cospan extends to a square in $\D$ uniquely up to unique isomorphism (by strong locality) & Axiom (K''), \newline Definition~\ref{reversecgw} \\
\hline $\nerve^{\mycube} \left(\H(S_{\cl})\right)$ & $Y_{*,0}$ is isomorphic to the classifying diagram of a category & Proposition~\ref{ndchar} \\
\hline $\nerve^{\mycube} \left(\V(S_{\cl})\right)$ & $Y_{0,*}$ is isomorphic to the classifying diagram of a category & Proposition~\ref{ndchar} \\
\hline $\nerve^{\mycube} \left([1] \boxtimes_v S_{\cl}\right)$ & $Y_{*,1}$ is isomorphic to the classifying diagram of a category & Proposition~\ref{ndchar} \\
\hline $\nerve^{\mycube} \left([1] \boxtimes_h S_{\cl}\right)$ & $Y_{1,*}$ is isomorphic to the classifying diagram of a category & Proposition~\ref{ndchar} \\
\hline
\end{tabular}
\caption{}
\label{classifyingcgwlifttable}
\end{table}

We can now conclude with our main results comparing reverse CGW categories and pointed stable double Segal spaces.

\begin{theorem}
	If $\D$ is a reverse CGW category, then $N^\square\D$ is a pointed stable double Segal space.
\end{theorem}

\begin{proof}
	 Since it is the triple nerve of a triple category, $N^\square\D$ is strictly double Segal. By Proposition~\ref{stableproof}, axiom (K'') for reverse CGW categories implies that $N^\square\D$ is stable. By Proposition~\ref{pointedproof}, axiom (Z') implies that $N^\square\D$ is pointed.
\end{proof}

We can also characterize what additional properties are needed for a pointed stable double Segal space to arise from a reverse CGW category.

\begin{theorem}
	A pointed stable double Segal space is isomorphic to the classifying diagram of a reverse CGW category if and only if it is isomorphic to the classifying diagram of a double category with shared isomorphisms in which all horizontal morphisms are monic and all vertical morphisms are epic.
\end{theorem}

\begin{proof}
The ``only if'' direction follows from axioms (I') and (M') of a reverse CGW category. For the ``if'' direction, assume $Y$ is pointed stable double Segal and isomorphic to $N^\square\D$. Axioms (I') and (M') hold by assumption. By Proposition~\ref{stableproof}, axiom (K'') follows from $N^\square\D$ being stable. By Proposition~\ref{pointedproof}, axiom (Z') follows from $N^\square\D$ being pointed. Hence $\D$ is reverse CGW.
\end{proof}

This theorem shows that the reverse CGW categories are the pointed stable double Segal spaces that are strictly local with respect to the set
\[ \nerve^{\mycube} \left(\H(S_{\cl}) \cup \V(S_{\cl}) \cup ([1] \boxtimes_v S_{\cl}) \cup ([1] \boxtimes_h S_{\cl})\right) \;\bigcup\; N^\square\{\H(\mono),\V(\mono^\opp)\}. \]

It also shows that a double category with shared isomorphisms whose classifying diagram is pointed and stable need only additionally satisfy axiom (M') to be reverse CGW.

\section{$S_\bullet$-constructions} \label{sdotsection}

%\subsection{The $S_\bullet$-construction and 2-Segal spaces}

%\bsnoteil{made minor edits in this section for pointedness}

In this section, we recall the $S_\bullet$-construction and use it to put some of our main results in context.  We briefly sketch the role of pointed stable double Segal spaces as a kind of universal input for algebraic $K$-theory, and then explain how the $S_\bullet$-construction for CGW categories fits into this picture.

%\jbnoteil{If we minimize the use of ``augmented" elsewhere, do we want to do so here, too?}
%\bsnoteil{Yeah I think pointed should be enough here too}
%
%\jbnoteil{Can make the following definition pointed by identifying all the objects $(i,i)$}

\begin{definition}
	For any $n \geq 0$, define a discrete pointed bisimplicial space $W[n]$ by
	\[ W[n]_{k,\ell} = \{(i_0, \ldots, i_k, j_0, \ldots, j_\ell) \mid i_0 \leq \cdots \leq i_k \leq j_0 \leq \cdots \leq j_\ell \} \]
	for $k,\ell \geq 0$, with the $(0,0)$-simplices  
	\[ W[n]_{-1} = \{(i,i) \mid 0 \leq i \leq n\} \]
	identified to a single point.
\end{definition}

The idea is that $W[n]$ should encode diagrams of the shape given in \eqref{sdotdiagram} with the diagonal objects $(i,i)$ identified so that $W[n]$ is pointed.

%shaped like the following staircase diagram including all possible compositions of the arrows and squares therein:
%\[ \begin{tikzcd}
%(0,0) \rar[tail] & (0,1) \rar[tail] \dar[etail] & (0,2) \rar[phantom]{\cdots} \dar[etail] & (0,n \text{-} 1) \rar[tail] \dar[etail] & (0,n) \dar[etail] \\
%& (1,1) \rar[tail] & (1,2) \rar[phantom]{\cdots} & (1,n \text{-} 1) \rar[tail] \dar[phantom]{\vdots} & (1,n) \dar[phantom]{\vdots} \\
%& & \ddots & (n \text{-} 2,n \text{-} 1) \rar[tail] \dar[etail] & (n \text{-} 2,n) \dar[etail] \\
%& & & (n \text{-} 1,n \text{-} 1) \rar[tail] & (n \text{-} 1,n) \dar[etail] \\
%& & & & (n,n).
%\dist{1-2}{2-3} \dist{1-4}{2-5} \dist{3-4}{4-5}
%\end{tikzcd} \]

\begin{definition}\label{defSdot}
	Let $Y$ be a pointed stable double Segal space.  Define the simplicial space $S_\bullet(Y)$ by
	\[ S_n(Y) = \Map(W[n],Y), \]
	the mapping space in the category of bisimplicial spaces.
\end{definition}

The motivation for the development of pointed stable double Segal spaces was the desire for the following theorem.  A reader unfamiliar with the theory of model categories should take this result as giving a suitable equivalence on the level of homotopy theories.

\begin{theorem}[{\cite[Proposition 7.5]{BOORSobj}}] \label{Sdotequiv}
	There is a Quillen equivalence between the model category for pointed stable double Segal spaces and the model category for reduced 2-Segal spaces whose right adjoint is the functor $S_\bullet$.
	%There is a Quillen equivalence between the model category for augmented stable double Segal spaces and the model category for 2-Segal spaces whose right adjoint is the functor $S_\bullet$.  Restricting to pointed stable double Segal spaces, we obtain a Quillen equivalence with reduced 2-Segal spaces, namely, those whose 0-space is contractible.
\end{theorem}

We do not need the notion of a 2-Segal space elsewhere in this paper, but we briefly recall their definition for the purposes of understanding the content of \cref{Sdotequiv}.

\begin{definition}
	A \emph{2-Segal space} is a simplicial space $V \colon \Deltaop \rightarrow \SSet$ such that certain \emph{2-Segal maps}
	\[ V_n \rightarrow \underbrace{V_2 \times_{V_1} \cdots \times_{V_1} V_2}_{n-1} \]
	induced from triangulations of regular $(n+1)$-gons for each $n\geq 3$, are weak equivalences of simplicial sets. A 2-Segal space $V$ is \emph{reduced} if $V_0$ is contractible.
\end{definition}

We refer the reader to \cite{dk} and \cite{gckt} for more information; in the latter reference 2-Segal spaces are instead called \emph{decomposition spaces}.  Both sets of authors proved that a main source of examples of 2-Segal spaces is the output of the $S_\bullet$-construction when applied to an exact category.  Theorem \ref{Sdotequiv} is a generalization that establishes that pointed stable double Segal spaces provide a universal input to the $S_\bullet$-construction so that the output has the structure of a 2-Segal space.

%\subsection{The relationship to the $S_\bullet$-construction for CGW categories}

We have shown that a CGW category give rise to a pointed stable double Segal space; here, we prove that the construction of its $K$-theory space factors through the $S_\bullet$-construction for pointed stable double Segal spaces.

%\bsnoteil{Potentially depending on whether we see this as the main result or as a byproduct of the results above, it may be worth writing out the original definition(s) of $K(\D)$.}

Recall that a CGW category $\D$ has an associated $K$-theory space $K(\D)$, as shown in \cite[Definition 4.1]{CZ-cgw}.  While $K(\D)$ is originally defined using an analogue of Quillen's $Q$-construction, in the following proof we use an equivalent formulation using a variant of the $S_\bullet$-construction.

\begin{corollary}
	If $\D$ is a CGW category, its $K$-theory space $K(\D)$ is weakly equivalent to $\Omega|S_\bullet(N\D^v)|$.
\end{corollary}

\begin{proof}
	 As shown in \cite[Theorem 7.8]{CZ-cgw} $K(\D)$ is weakly equivalent  $\Omega|s_\bullet\D|$, where $s_\bullet\D$ is a simplicial set with $s_n\D$ in bijection with $\Hom_{\SSSet}(W[n],\nerve^\square\D^v)$. By a variant of a standard argument of Waldhausen \cite{Waldhausen}, $|s_\bullet\D|$ is weakly equivalent to the realization of the simplicial space $iS_\bullet\D$, where $iS_n\D$ the nerve of the groupoid with object set $s_n\D$ and morphisms given by natural transformations between diagrams in $\D$.  A double categorical version of this argument is given in \cite[Lemma 5.8]{SS-ecgw}.  This simplicial space $iS_\bullet\D$ is precisely equal to $S_\bullet(N\D^v)$ as defined in \cref{defSdot}, completing the proof.
\end{proof}

\bibliographystyle{alpha}
\bibliography{BZ}

@article{CZ-cgw,
  title={D{\'e}vissage and localization for the {G}rothendieck spectrum of varieties},
  author={Campbell, Jonathan A. and Zakharevich, Inna},
  journal={Adv. Math.},
  volume={411},
  pages={108710},
  year={2022},
  publisher={Elsevier}
}

@article{BOORSset,
  title={2-{S}egal sets and the {W}aldhausen construction},
  author={Bergner, Julia E. and Osorno, Ang{\'e}lica M. and Ozornova, Viktoriya and Rovelli, Martina and Scheimbauer, Claudia I.},
  journal={Topology Appl.},
  volume={235},
  pages={445--484},
  year={2018},
  publisher={Elsevier}
}

@article{BOORSobj,
  title={2--{S}egal objects and the {W}aldhausen construction},
  author={Bergner, Julia E. and Osorno, Ang{\'e}lica M. and Ozornova, Viktoriya and Rovelli, Martina and Scheimbauer, Claudia I.},
  journal={Algebr. Geom. Topol.},
  volume={21},
  number={3},
  pages={1267--1326},
  year={2021},
  publisher={Mathematical Sciences Publishers}
}

@book{jbbook,
  title={The homotopy theory of $(\infty, 1)$-categories},
  author={Bergner, Julia E},
  volume={90},
  year={2018},
  publisher={Cambridge University Press}
}

@article{rezk2001model,
  title={A model for the homotopy theory of homotopy theory},
  author={Rezk, Charles},
  journal={Trans. Amer. Math. Soc.},
  volume={353},
  number={3},
  pages={973--1007},
  year={2001}
}

@article{intercategories,
  title={Intercategories},
  author={Grandis, Marco and Par{\'e}, Robert},
  journal={Theory Appl. Categ.},
  volume={30},
  number={38},
  pages={1215--1255},
  year={2015}
}

@article{nfold,
  title={A {T}homason model structure on the category of small $n$-fold categories},
  author={Fiore, Thomas M. and Paoli, Simona},
  journal={Algebr. Geom. Topol.},
  volume={10},
  number={4},
  pages={1933--2008},
  year={2010},
  publisher={Mathematical Sciences Publishers}
}

@book{dk,
  title={Higher {S}egal Spaces},
  author={Dyckerhoff, Tobias and Kapranov, Mikhail M. and others},
  volume={2244},
  year={2019},
  publisher={Springer}
}

@article{gckt,
  title={Decomposition spaces, incidence algebras and {M}{\"o}bius inversion {I}: Basic theory},
  author={G{\'a}lvez-Carrillo, Imma and Kock, Joachim and Tonks, Andrew},
  journal={Adv. Math.},
  volume={331},
  pages={952--1015},
  year={2018},
  publisher={Elsevier}
}

@inproceedings{Waldhausen,
  title={Algebraic {K}-theory of spaces},
  author={Waldhausen, Friedhelm},
  booktitle={Algebraic and Geometric Topology: Proceedings of a Conference held at Rutgers University, New Brunswick, USA July 6--13, 1983},
  pages={318--419},
  year={2006},
  organization={Springer}
}

@inproceedings{KellyAdjunction,
  title={Adjunction for enriched categories},
  author={Kelly, G. Max},
  booktitle={Reports of the Midwest Category Seminar III},
  pages={166--177},
  year={2006},
  organization={Springer}
}

@article{SS-ecgw,
    author = {Sarazola, Maru and Shapiro, Brandon T.},
    title = {Additivity and fiber sequences for combinatorial {K}-theory},
    year = {2021},
    eprint={2107.07701},
    archivePrefix={arXiv},
    primaryClass={math.KT},
    note = {\href{https://arxiv.org/abs/2107.07701}{arXiv:2107.07701}}
}

@article{modelbicats,
  title={Model bicategories and their homotopy bicategories},
  author={Descotte, Mar{\'\i}a E. and Dubuc, Eduardo J. and Szyld, Mart{\'\i}n},
  journal={Adv. Math.},
  volume={404},
  pages={108455},
  year={2022},
  publisher={Elsevier}
}

@book{hirschhorn,
  title={Model Categories and Their Localizations},
  author={Hirschhorn, Philip S.},
  number={99},
  year={2003},
  publisher={American Mathematical Soc.}
}

@article{fiorepaolipronk,
  title={Model structures on the category of small double categories},
  author={Fiore, Thomas M. and Paoli, Simona and Pronk, Dorette},
  journal={Algebr. Geom. Topol.},
  volume={8},
  number={4},
  pages={1855--1959},
  year={2008},
  publisher={Mathematical Sciences Publishers}
}

@article{segal1968classifying,
  title={Classifying spaces and spectral sequences},
  author={Segal, Graeme},
  journal={Publ. Math. IH\'ES},
  volume={34},
  pages={105--112},
  year={1968}
}

\end{document}